\pdfoutput=1
\documentclass[a4paper,11pt,oneside]{amsart}
\usepackage[T1]{fontenc}
\usepackage[utf8]{inputenc}
\usepackage[unicode]{hyperref}
\hypersetup{
bookmarks=true,
colorlinks=true,
citecolor=[rgb]{0,0,0.5},
linkcolor=[rgb]{0,0,0.5},
urlcolor=[rgb]{0,0,0.75},
pdfpagemode=UseNone,
pdfstartview=FitH,
pdfdisplaydoctitle=true,
pdftitle={Decoupling for moment manifolds},
pdfauthor={Guo, Zorin-Kranich},
pdflang=en-US
}

\usepackage[style=alphabetic,maxalphanames=4]{biblatex}
\usepackage{esint}
\usepackage{tikz}

\usepackage[paperheight=279mm,paperwidth=18cm,textheight=26cm,textwidth=14cm,includehead]{geometry}
\usepackage{amsfonts,amssymb,amsmath,amsthm}
\usepackage{mathtools}
\usepackage{etoolbox}

\def\PZdefchar#1{
  \expandafter\def\csname frak#1\endcsname{\mathfrak{#1}}
  \expandafter\def\csname bf#1\endcsname{\mathbf{#1}}
  \expandafter\def\csname scr#1\endcsname{\mathcal{#1}}
  \expandafter\def\csname cal#1\endcsname{\mathcal{#1}}}
\def\PZdefloop#1{\ifx#1\PZdefloop\else\PZdefchar#1\expandafter\PZdefloop\fi}
\PZdefloop abcdefghijklmnopqrstuvwxyzABCDEFGHIJKLMNOPQRSTUVWXYZ\PZdefloop

\newcommand{\R}{\mathbb{R}}
\newcommand{\N}{\mathbb{N}}
\newcommand{\Z}{\mathbb{Z}}
\newcommand{\Uncert}{\mathcal{U}} 
\def\tA{\tilde{A}}

\def\tGamma{\tilde{\Gamma}}
\def\tgamma{\tilde{\gamma}}
\def\C{\mathbb{C}}
\newcommand{\dif}{\mathrm{d}}
\newcommand{\level}{\mathcal{V}}
\newcommand{\sublevel}{\mathcal{S}}
\newcommand{\upset}{{\uparrow}} 
\newcommand{\one}{\mathbf{1}}

\numberwithin{equation}{section}
\newtheorem{theorem}{Theorem}
\numberwithin{theorem}{section}
\newtheorem{proposition}[theorem]{Proposition}
\newtheorem{lemma}[theorem]{Lemma}
\newtheorem{corollary}[theorem]{Corollary}
\theoremstyle{definition}
\newtheorem{definition}[theorem]{Definition}
\newtheorem{notation}[theorem]{Notation}

\theoremstyle{remark}
\newtheorem{remark}[theorem]{Remark}
\newtheorem{example}[theorem]{Example}

\newcommand{\Dset}{\mathcal{D}}

\DeclarePairedDelimiter\abs{\lvert}{\rvert}

\DeclarePairedDelimiter\norm{\lVert}{\rVert}
\DeclarePairedDelimiter\card{\lvert}{\rvert}
\DeclarePairedDelimiter\floor{\lfloor}{\rfloor}
\providecommand\given{}
\newcommand\SetSymbol[1][]{%
\nonscript\:#1\vert
\allowbreak
\nonscript\:
\mathopen{}}
\DeclarePairedDelimiterX\Set[1]\{\}{\renewcommand\given{\SetSymbol[\delimsize]}#1}
\DeclarePairedDelimiterXPP\EE[1]{\E}{\lparen}{\rparen}{}{\renewcommand\given{\SetSymbol[\delimsize]}#1} 

\makeatletter
\newcommand\@avsum[2]{%
  {\sbox0{$\m@th#1\sum$}%
   \vphantom{\usebox0}%
   \ooalign{%
     \hidewidth
     \smash{\vrule height\dimexpr\ht0+1pt\relax depth\dimexpr\dp0+1pt\relax}%
     \hidewidth\cr
     $\m@th#1\sum$\cr
   }%
  }%
}
\newcommand{\avsum}{\mathop{\mathpalette\@avsum\relax}\displaylimits}
\newcommand\@avprod[2]{%
  {\sbox0{$\m@th#1\prod$}%
   \vphantom{\usebox0}%
   \ooalign{%
     \hidewidth
     \smash{\vrule height\dimexpr\ht0+1pt\relax depth\dimexpr\dp0+1pt\relax}%
     \hidewidth\cr
     $\m@th#1\prod$\cr
   }%
  }%
}
\newcommand{\avprod}{\mathop{\mathpalette\@avprod\relax}\displaylimits}
\newcommand{\@avL}[2]{%
\ooalign{{$\m@th#1\mbox{--}$}\cr {$\m@th#1 L$}\cr}}
\newcommand{\avL}{\mathpalette\@avL\relax}
\newcommand{\@avell}[2]{%
\ooalign{{$\m@th#1\mbox{--}$}\cr {$\m@th#1 \ell$}\cr}}
\newcommand{\avell}{\mathpalette\@avell\relax}
\newcommand{\@avD}{%
  \ooalign{{$\mathrm{D}$}\cr \hidewidth\raise.2ex\hbox{$\vert$}\hidewidth\cr}}
\newcommand{\avDec}{\@avD\mathrm{ec}}
\makeatother
\DeclareMathOperator{\rk}{rk}
\DeclareMathOperator{\lin}{lin}
\DeclareMathOperator{\Span}{span}
\DeclareMathOperator{\diam}{diam}
\DeclareMathOperator{\supp}{supp}
\DeclareMathOperator{\rank}{rank}
\newcommand{\BL}{\mathrm{BL}}

\newcommand{\Part}[2][]{\calP(\ifstrempty{#1}{}{#1,}#2)} 

\newcommand{\RHS}{\mathrm{RHS}}
\newcommand{\LHS}{\mathrm{LHS}}

\begin{document}
\title[Decoupling for moment manifolds]{Decoupling for moment manifolds associated to\\ Arkhipov--Chubarikov--Karatsuba systems}
\author{Shaoming Guo}
\address[SG]{Department of Mathematics\\ Chinese University of Hong Kong\\ Hong Kong}
\curraddr{Department of Mathematics\\ University of Wisconsin--Madison\\ Madison\\ WI}
\author{Pavel Zorin-Kranich}
\address[PZ]{Mathematical Institute\\ University of Bonn\\ Germany}
\maketitle
\begin{abstract}
We prove $\ell^{p}L^{p}$ decoupling inequalities for a class of moment manifolds.
These inequalities imply optimal mean value estimates for multidimensional Weyl sums of the kind considered by Arkhipov, Chubarikov, and Karatsuba and by Parsell.

In our proofs we take a new point of view on the Bourgain--Demeter--Guth induction on scales argument.
This point of view substantially simplifies even the proof of $\ell^{2}L^{p}$ decoupling for the moment curve.
\end{abstract}

\section{Introduction}
The sharp $\ell^{2}L^{p}$ decoupling inequality for the moment curve was proved by Bourgain, Demeter, and Guth in \cite{MR3548534}. It implies asymptotically optimal mean value estimates for one-dimensional Weyl sums.
In a series of subsequent works \cite{MR3614930,MR3709122,arxiv:1804.02488}, sharp decoupling inequalities were proved for many moment manifolds (graphs of systems of monomials) of higher dimensions.
We continue this line of investigation and obtain sharp $\ell^{p}L^{p}$ decoupling inequalities that imply in particular asymptotically optimal mean value estimates for multidimensional Weyl sums considered in the work of Arkhipov, Chubarikov, and Karatsuba \cite{MR2113479}. For earlier works in the decoupling literature, in particular, works prior to Bourgain and Demeter \cite{MR3374964}, we refer to Wolff \cite{MR1800068}, \L aba and Wolff \cite{MR1956533}, \L aba and Pramanik \cite{MR2264215}, Garrigos and Seeger \cite{MR2546636}, \cite{MR2664568}, Bourgain \cite{MR3038558}, and references therein.

In order to keep our presentation self-contained, we include in Section~\ref{sec:multilinear} several arguments which have been used throughout decoupling literature.
These are formulated in a way that permits using them both in $\ell^{2}L^{p}$ and $\ell^{p}L^{p}$ decoupling inequalities.
In Section~\ref{sec:induction-on-scales} we simplify and extend the Bourgain--Demeter--Guth induction on scales argument.
Here the central result is Theorem~\ref{thm:PF-eigenvector}, which allows one to exploit the web of inequalities in Figure~\ref{fig:graph}.
A key input in the induction on scales argument is a transversality condition, which is verified in Section~\ref{sec:transverse}.
Section~\ref{sec:lower} shows that our upper bounds are $\epsilon$-close to the existing lower bounds.

\subsection{Notation and statement of the main result}
We begin with the description of the $\ell^{q}L^{p}$ decoupling problem.
For $d\in\Set{1,2,\dotsc}$ and a finite set of exponents $\calD \subset \N^{d} \setminus \Set{0}$, we are interested in functions with Fourier support near the graph of the function $\Phi : \R^{d} \to \R^{\calD}$, $t \mapsto (t^{\bfi})_{\bfi \in \calD}$.
Here and later boldface letters denote elements of $\N^{d}$, $\N=\Set{0,1,\dotsc}$, $\R^{\calD}$ is the product of $\abs{\calD}$ copies of $\R$ indexed by $\calD$, and we use the multiindex notation $t^{\bfi} := t_{1}^{i_{1}} \dotsm t_{d}^{i_{d}}$ for monomials.
For $\bfa = (a_{1},\dotsc,a_{d}) \in \N^{d}$, we write $\abs{a} := a_{1} + \dotsb + a_{d}$.
Following \cite{MR3132907}, we refer to $d$ as the \emph{dimension} of $\calD$, the cardinality $\rk\calD := \abs{\calD}$ as the \emph{rank} of $\calD$, and the maximal absolute value $\deg\calD := \max_{\bfi \in\calD} \abs{\bfi}$ as the \emph{degree} of $\calD$.
Deviating from the number-theoretic terminology, we call
\begin{equation}
\label{eq:homdim}
\calK(\calD) := \sum_{\bfi \in \calD} \abs{\bfi}
\end{equation}
the \emph{homogeneous dimension} of $\calD$.

For $\delta>0$ and a dyadic cube $\alpha\subseteq [0,1]^{d}$ with side length $\geq \delta$, let $\Part[\alpha]{\delta}$ denote the collection of smallest dyadic cubes with side length $\geq \delta$ that are contained in $\alpha$.
In the case $\alpha = [0,1]^{d}$ we omit $\alpha$ and write $\Part{\delta} := \Part[{[0,1]^{d}}]{\delta}$.
For a dyadic cube $\alpha \subseteq [0,1]^{d}$, we denote by $\Uncert(\alpha)$ an essentially minimal parallelepiped in $\R^{\calD}$ that contains $\Phi(\alpha)$, see Section~\ref{sec:scaling} for a more precise definition.

For $2 \leq q \leq p < \infty$ and $0 < \delta < 1$, let $\avDec(\calD, p, q, \delta)$ denote the infimum over all constants $C$ such that the inequality
\begin{equation}
\label{eq:dec-const-def}
\norm{\sum_{\theta \in \Part{\delta}} f_{\theta}}_{L^p(\R^{\calD})}
\leq
C \bigl( \avsum_{\theta \in \Part{\delta}} \norm{f_{\theta}}_{L^p(\R^{\calD})}^q \bigr)^{1/q}
\end{equation}
holds for any functions $f_{\theta}$ with $\supp \widehat{f_{\theta}} \subseteq \Uncert(\theta)$.
Here and later we denote averages by $\avsum_{\theta \in \calJ} := \abs{\calJ}^{-1} \sum_{\theta \in \calJ}$.
The vertical line in the notation $\avDec$ reminds of the average and indicates a change in the convention from previous works, where the sum in $\theta$ is not normalized.
Our convention is motivated by the more direct connection with the number of solutions to Vinogradov systems and by the need to use Jensen's inequality in the sum over $\theta$ that would produce extraneous terms without the normalization.
Since we are mostly interested in the case $p=q$, we will abbreviate $\avDec(\calD, p, \delta) := \avDec(\calD, p, p, \delta)$.

Now we describe the sets $\calD$ that we will consider.
For $\bfk = (k_{1},\dotsc,k_{d}) \in \N_{>0}^{d}$ let $\calD(\bfk) := \prod_{j=1}^{d} \Set{0,\dotsc,k_{j}} \subset \N^{d}$.
For $l\in\N$ we define \emph{level} and \emph{sublevel} sets
\begin{align}
\label{eq:level}
\level_{l} &:=
\Set{ \bfa \in \N^{d} \given \abs{\bfa} = l},\\
\label{eq:sublevel}
\sublevel_{l} &:=
\Set{ \bfa \in \N^{d} \given 1 \leq \abs{\bfa} \leq l}.
\end{align}
We write $\calD(\bfk,= l) := \calD(\bfk) \cap \level_{l}$ and $\calD(\bfk,\leq l) := \calD(\bfk) \cap \sublevel_{l}$.

Our main result is the following.
\begin{theorem}
\label{thm:main}
Let $d \geq 1$, $\bfk = (k_{1},\dotsc,k_{d}) \in \N^{d}$ with $1 \leq k_{1} \leq \dotsc \leq k_{d}$, and $1 \leq k$.
Then for every $2 \leq p < \infty$ and $\epsilon>0$ we have
\begin{equation}\label{eq:main-est}
\avDec(\calD(\bfk,\leq k), p, \delta)
\lesssim_{\epsilon}
\delta^{-\tgamma-\epsilon},
\end{equation}
where
\begin{equation}
\label{eq:tgamma}
\tgamma
=
\tgamma(\bfk,k,p)
=
\max \Bigl( \frac{d}{2}, \max_{(d+1)/2\leq j\leq d} \bigl( j + \frac{d-j}{p} - \frac{\calK(\calD((k_{1},\dotsc,k_{j}),\leq k))}{p} \bigr) \Bigr).
\end{equation}
\end{theorem}
Here and later we denote by $C_{\epsilon}$ finite constants that are allowed to depend on $\epsilon$ and may change from line to line. They are also always allowed depend on the parameteres $d, \bfk, k, p, E$, but never on $\delta$ and $f_{\theta}$.
The notation $A \lesssim_{\epsilon} B$ means that $A \leq C_{\epsilon} B$.

In order to illustrate Theorem~\ref{thm:main}, we compute the exponent \eqref{eq:tgamma} more explicitly in several important special cases.
\begin{example}
\label{ex:dec:Vin}
The case $d=1$, $k_{1}=k$ corresponds to the classical Vinogradov system.
In this case the maximum in \eqref{eq:tgamma} reduces to the term $j=d=1$, and we obtain
\[
\calK(\calD((k),\leq k)) = 1+\dotsc+k = \frac{k(k+1)}{2},
\quad
\tgamma
= \max\Bigl( \frac{1}{2}, 1 - \frac{k(k+1)}{2 p} \Bigr).
\]
This should be compared with the result in \cite{MR3548534}, which is stronger because it is an $\ell^{2}L^{p}$ decoupling.
The additional ingredient needed to prove $\ell^{2}L^{p}$ decoupling is explained in Appendix~\ref{sec:l2}.
Moreover, there is a difference in normalization: we split the moment curve in pieces of size $\delta$, whereas in \cite{MR3548534} pieces of size $\delta^{1/k}$ are used.
\end{example}

\begin{example}
\label{ex:dec:PV}
The case of arbitrary $d$ and $k_{1}=\dotsc=k_{d}=k$ is the Parsell--Vinogradov case treated in \cite{arxiv:1804.02488}.
In this case we have
\[
\calK_{j,k} := \calK(\calD(k_{1}=k,\dotsc,k_{j}=k),\leq k) = \sum_{l=0}^{k} l \binom{j+l-1}{j-1} = \frac{jk}{j+1} \binom{k+j}{j},
\]
hence
\[
\tgamma = \max\Bigl( \frac{d}{2}, \max_{(d+1)/2<j\leq d} \bigl( j + \frac{d-j}{p} - \frac{\calK_{j,k}}{p} \bigr) \Bigr).
\]
This is the same estimate as \cite[Theorem 1.2]{arxiv:1804.02488}, taking into account that we normalize sums over $\Part{\delta}$.
\end{example}

\begin{example}
\label{ex:dec:d=2,k1=1}
Let $d=2$ and $\bfk=(1, k_2)$ for some integer $k_2\ge 1$.
In this case
\[
\abs{\calD((1, k_2), =l)} =
\begin{cases}
1, & l \in \Set{ 0,k_{2}+1 },\\
2, & 1 \leq l \leq k_{2},\\
0, & \text{otherwise.}
\end{cases}
\]
Hence for $k\le k_2$ we obtain
\[
\calK(\calD((1, k_2), \le k))=2(1+2+\dotsb+k)=k(k+1),
\quad
\tgamma = \max\Bigl( 1, 2-\frac{k(k+1)}{p} \Bigr).
\]
In the case $k = k_2+1$, we obtain
\[
\calK(\calD((1, k_2), \le k))=2(1+2+\dots+k_2)+(k_2+1)=(k_2+1)^2,
\quad
\tgamma = \max\Bigl( 1, 2-\frac{(k_2+1)^2}{p} \Bigr).
\]
\end{example}

\begin{example}
\label{ex:dec:box}
Let $d\geq 1$ be arbitrary, $1 \leq k_{1}\leq \dotsb \leq k_{d}$, and $k\geq k_{1}+\dotsc+k_{d}$.
Then
\begin{align*}
\calK((k_{1},\dotsc,k_{j}),\leq k)
&=
\sum_{i_{1} = 0}^{k_{1}} \dots \sum_{i_{j} = 0}^{k_{j}} (i_{1}+\dotsb+i_{j})
\\ &=
\sum_{m=1}^{j} (k_{1}+1) \dotsm (k_{m-1}+1) \Bigl( \sum_{i_{m} \leq k_{m}} i_{m} \Bigr) (k_{m+1}+1) \dotsm (k_{j}+1)
\\ &=
\sum_{m=1}^{j} (k_{1}+1) \dotsm (k_{m-1}+1) \frac{k_{m}(k_{m}+1)}{2} (k_{m+1}+1) \dotsm (k_{j}+1)
\\ &=
\frac{1}{2} (k_{1}+1) \dotsm (k_{j}+1) \sum_{m=1}^{j} k_{m}.
\end{align*}
\end{example}

\begin{example}
\label{ex:dec:binary}
Specializing to $d=2$ in Example~\ref{ex:dec:box}, we obtain
\[
\tgamma = \max\Bigl( 1, 2-\frac{(k_{1}+1)(k_2+1)(k_{1}+k_{2})}{2p} \Bigr).
\]
\end{example}

\begin{example}
\label{ex:dec:cube}
Specializing to $k_{1}=\dotsb=k_{d}=k_{0}$ in Example~\ref{ex:dec:box}, we obtain
\[
\calK((k_{1},\dotsc,k_{j}),\leq k) = j (k_{0}+1)^{j} k_{0}/2,
\]
\[
\tgamma = \max \Bigl( \frac{d}{2}, \max_{(d+1)/2\leq j\leq d} \bigl( j + \frac{d-j}{p} - \frac{j (k_{0}+1)^{j} k_{0}}{2p} \bigr) \Bigr).
\]
The maximum over $j$ cannot be replaced by the term $j=d$ already in the case $d=3$, $k_{0}=2$.
\end{example}

\subsection{Consequences for multidimensional Vinogradov systems}
\label{sec:intro:optimality}
Let $s \in \Set{1,2,\dotsc}$ and consider the system of equations
\begin{equation}
\label{eq:Vinogradov-system}
\sum_{j=1}^{s} (x_{j})^{\bfi} = \sum_{j=1}^{s} (y_{j})^{\bfi},
\quad
\bfi \in \calD,
\end{equation}
in $2sd$ unknowns, where $x_{j},y_{j} \in \N^{d}$.
Given $X \geq 1$, let
\[
J_{s}(X; \calD) := \# \Set{ (x_{1},\dotsc,x_{s},y_{1},\dotsc,y_{s}) \in (\N^{d} \cap [1,X]^{d})^{2s} \given \text{\eqref{eq:Vinogradov-system} holds} }
\]
denote the number of solutions to \eqref{eq:Vinogradov-system} all of whose entries are bounded by $X$.
By the reduction in \cite[Section 4]{MR3548534}, it is known that
\begin{equation}
\label{eq:J-est}
J_{s}(X; \calD) \lesssim \avDec(\calD, 2s, X^{-1})^{2s}.
\end{equation}
The argument given in \cite[Section 4]{MR3548534} uses a localized version of Theorem~\ref{thm:main}, but it can be also carried out with the global version applied to the functions $f_{\theta}(x) = e(\Phi(c_{\theta}) \cdot x) \phi(x)$, where $\phi$ is a Schwartz function with $\supp \widehat{\phi} \subseteq B(0,\delta^{-k})$ and $c_{\theta} \in \theta$ are suitable rational points.

Thus Theorem~\ref{thm:main} has the following consequence.
\begin{corollary}
\label{cor:Vinogradov-system}
Let $d \geq 1$, $k\geq 1$, and $1 \leq k_{1} \leq \dotsb \leq k_{d}$ be integers.
Then, for every $\epsilon>0$, we have
\begin{equation}
\label{eq:Js-bound}
J_{s}(X; \calD(\bfk,\leq k))
\lesssim_{\epsilon}
X^{2s \tgamma+\epsilon},
\end{equation}
where $\tgamma=\tgamma(\bfk,k,2s)$ is given by \eqref{eq:tgamma}.
\end{corollary}
The upper bound \eqref{eq:Js-bound} matches (up to the $\epsilon$ loss) the lower bound in \cite[Section 3]{MR3132907}.
This is proved in Section~\ref{sec:lower}.
In particular, the exponent \eqref{eq:tgamma} in Theorem~\ref{thm:main} is optimal when the exponent $p$ is an even integer.

Let us pause and mention a few special cases of our theorem.
Let $\bfk=(k_1, \dotsc, k_d)$.
As mentioned in Example~\ref{ex:dec:PV}, the case $\calD(\bfk, k)$ with $k\le \min\{k_1, \dotsc, k_d\}$ covers \emph{Parsell--Vinogradov systems}, see \cite{MR2149529,MR3132907,arxiv:1804.02488}.
We refer to the introduction of \cite{arxiv:1804.02488} for a discussion of applications of these systems.

The system \eqref{eq:Vinogradov-system} with $\calD=\calD(\bfk, \leq k)$ and $k=k_1+ \dotsb + k_d$ (Example~\ref{ex:dec:box}) was extensively studied by Arkhipov, Chubarikov, and Karatsuba, who summarized their results in the book \cite{MR2113479}.
In this case, since for sufficiently large $p$ the term $j=d$ dominates in \eqref{eq:tgamma}, Corollary~\ref{cor:Vinogradov-system} gives the bound
\[
J_{s}(X;\calD) \lesssim_{\epsilon} X^{\alpha+\epsilon},
\quad
\alpha = 2sd - \frac{1}{2} (k_{1}+1) \dotsm (k_{d}+1) (k_{1}+\dotsb+k_{d}),
\]
for sufficiently large $s$ and any $\epsilon>0$.
This can be compared with \cite[Theorem 4.3]{MR2113479} (with $r=d$, $n_{j}=k_{j}$, $k=s$, $P_{1}=\dotsb=P_{d}=X$), in which the exponent $\alpha$ is replaced by
\[
\alpha + \frac{1}{2} (k_{1}+1)\dotsm (k_{d}+1) (k_{1}+\dotsb+k_{d}) (1-1/(k_{1}+\dotsb+k_{d}))^{\lfloor s/((k_{1}+1)\dotsm(k_{d}+1)) \rfloor}.
\]
One situation in which the precise exponent is important occurs in \cite[Theorem 1.3]{MR3653250}.

When $d=2$ and $\calD=\calD(\bfk, k)$ with $k=k_1+ k_2$ (Example~\ref{ex:dec:binary}), the associated system \eqref{eq:Vinogradov-system} is called a \emph{simple binary system}, and it appeared in recent work in quantitative arithmetic geometry (Section 4.15 of \cite{MR2498064} and \cite{MR2817374}).
Moreover, it is a particular case of Prediville's systems \cite{MR3092339} with the generating polynomial $t_1^{k_1}t_2^{k_2}$.
Applications of exponential sum estimates associated to these systems have been carefully worked out in \cite{MR3092339}.

\subsection{Relation to previous works}
Theorem~\ref{thm:main} is proved by induction on the dimension $d \geq 1$ and degree $k \geq 1$.
The base case $k=1$ is given by $L^{2}$ orthogonality and interpolation.

The Bourgain--Guth argument originating in \cite{MR2860188} begins with splitting the left-hand side of \eqref{eq:dec-const-def} in Heisenberg uncertainty regions at a suitable scale.
Inside each region either transverse or non-transverse contributions dominate.
Non-transverse contributions come from neighborhoods of low degree varieties in $[0,1]^{d}$ and are handled using the inductive hypothesis with a lower $d$.
In the case of the paraboloid in \cite{MR3374964}, these low degree subvarieties were hyperplanes.
Higher degree varieties first appeared in \cite{MR3614930}; our treatment mostly follows \cite{arxiv:1804.02488}.

Transverse contributions are handled using an induction on scales argument.
For $k=2$, this argument was introduced by Bourgain and Demeter \cite{MR3374964} (see also the more streamlined exposition in \cite{MR3592159} and \cite{Demeter-ICM}), and it was extended to $k\geq 3$ by Bourgain, Demeter, and Guth \cite{MR3548534}.
This argument consists of three main ingredients:
\begin{enumerate}
\item ``ball inflation'' (Lemma~\ref{lem:ball-inflation}),
\item lower degree and smaller scale (by ``rescaling'') decoupling (Lemma~\ref{lem:scaled-decoupling}), and
\item a bootstrapping argument in which the former two ingredients are applied iteratively, yielding a gain over a trivial estimate.
\end{enumerate}
Ball inflation relies on a common generalization of multilinear Kakeya and Brascamp--Lieb inequalities.
Such an estimate was first proved in \cite{MR3783217}.
It is more convenient to use an endpoint version from \cite{arxiv:1807.09604}.
In order to apply it, one has to verify a transversality condition found in \cite{MR2377493}.
For moment manifolds, the transversality condition was reduced to a conjecture in linear algebra in \cite[Conjecture 3.1]{MR3709122} (in the case $\calD = \sublevel_{k}$ corresponding to Parsell--Vinogradov systems, a similar reduction can be made for arbitrary down-sets $\calD$, see Definition~\ref{def:up+down}).
For Parsell--Vinogradov systems this conjecture was verified in \cite{arxiv:1804.02488} using an extension of the Schwartz--Zippel lemma.
Our first contribution is the verification of the conjecture for the wider class of sets $\calD$ in Theorem~\ref{thm:main}, see Section~\ref{sec:transverse}.

In \cite{MR3592159} and \cite{MR3548534}, the bootstrapping argument is run at a certain critical exponent $p$, and results for other $p$'s follow by interpolation with easy endpoints at $p=2$ and $p$ near $\infty$.
In the higher-dimensional setting there are typically many critical exponents, which makes a case by case treatment difficult.
This problem was solved in \cite{arxiv:1804.02488}, where all values of $2 \leq p < \infty$ are treated directly.
We further unify these arguments by removing the distinction between small and large values of $p$ present in \cite{MR3592159} and \cite{arxiv:1804.02488}.

More importantly, we view the ``tree-growing'' procedure in previous works from a different perspective that is summarized in Figure~\ref{fig:graph}.
A similar idea (in the case $d=1$) independently appeared in a blog post by Terence Tao, from which we adopted the definition \eqref{eq:tilde-a}.
Putting all estimates in the induction on scales procedure on an equal footing allows us to replace a host of ad hoc calculations of \cite{arxiv:1804.02488} by Theorem~\ref{thm:PF-eigenvector} that describes the right Perron--Frobenius eigenvector of the matrix that contains all essential information about inequalities used.
Theorem~\ref{thm:PF-eigenvector} holds for arbitrary down-sets $\calD$ (see Definition~\ref{def:up+down}), thus reducing possible generalizations of Theorem~\ref{thm:main} to the verification of the transversality condition.

The exponent \eqref{eq:tgamma} is a compressed way to express the recursive upper bound \eqref{eq:tGamma-recursion} that comes out of our proof.
In Section~\ref{sec:lower}, we show that our upper bounds \eqref{eq:tgamma} and \eqref{eq:tGamma-recursion} coincide with the lower bound obtained in \cite{MR3132907}.
We hope that the argument in Section~\ref{sec:lower}, which is more streamlined than that of \cite{arxiv:1804.02488}, is also more robust and can be applied to more general translation-dilation invariant systems.

\subsection{\texorpdfstring{$\ell^{2}L^{p}$}{\9041\023\9000\262 Lp} decoupling}
\label{sec:intro:l2}
The decoupling constant $\avDec(\calD,p,q,\delta)$ is, with our normalization, a monotonically decreasing function of $q$.
Thus it should morally be easier to estimate it for large $q$.
On the other hand, the most important ingredient of the proof, Lemma~\ref{lem:ball-inflation}, in its current form only works for $q \leq p$.
Since the value of $q$ is not important for the purpose of estimating the number of solutions of Vinogradov systems \eqref{eq:J-est}, in hindsight it appears natural to consider $q=p$.

Nevertheless, all our proofs also work for other values of $2 \leq q \leq p$, see Appendix~\ref{sec:l2}.
However, the growth rate of the decoupling constant as $\delta\to 0$ may be worse than in the case $q=p$.
In the one-dimensional case $d=1$ we do obtain the same growth rate also for $q=2$, thus recovering the $\ell^{2}L^{p}$ decoupling inequalities in \cite{MR3548534} with a simpler induction on scales argument.
The reader primarily interested in the case $d=1$ should also notice that the treatment of transversality in Section~\ref{sec:transverse} drastically simplifies in this case, see Remark~\ref{rem:Vandermonde}.

\subsection{Open problems}
In view of the results in \cite{MR3132907}, it would be interesting to extend Theorem~\ref{thm:main} to arbitrary down-sets $\calD \subset \N^{d}\setminus\Set{0}$ (see Definition~\ref{def:up+down}).

More generally, one can ask which decoupling inequalities hold for general translation-dilation invariant systems of polynomials as in \cite{MR3132907}.
Most known examples are of this type \cite{MR3447712, arxiv:1609.04107, MR3614930, MR3592159, arxiv:1701.06732, arxiv:1804.02488, arxiv:1902.03450} or perturbations thereof.

Here we provide one simple example of a moment surface of dimension $d=2$, degree $3$, and rank $5$ for which the argument used in the current paper fails.
Let $\calD := \Set{ (1, 0), (2, 0), (3, 0), (0, 1), (1, 1)}$.
The associated surface is given by
\begin{equation}\label{nonexample}
\Set{ (t_1, t_2, t_1^2, t_1 t_2, t_1^3): (t_1, t_2)\in [0, 1]^2 }.
\end{equation} 
To apply the multilinear approach of Bourgain and Demeter \cite{MR3374964} and Bourgain, Demeter, and Guth  \cite{MR3548534}, one needs to verify a transversality condition (see \eqref{eq:BL-transversality}).
In order for transverse sets to exist, there has to exist a collection of $M \geq 1$ points $\{t_j\}_{j=1}^M\subset [0, 1]^2$ such that
\begin{equation}
\dim(V) \le \frac{5}{4} \frac{1}{M}\sum_{j=1}^M \dim(\pi_j(V)), \text{ for every } V\subset \R^5,
\end{equation}
where $\pi_j$ denotes the orthogonal projection onto $V^{(2)}(t_j)$, and $V^{(2)}(t_j)$ is the second order tangent space of the surface \eqref{nonexample} at $t_j$  (see \eqref{order_tangent}).
However, if one takes $V$ to be the span of three vectors $e_1, e_3, e_5$ from the standard basis in $\R^5$, it is not difficult to check that $\dim(\pi_j(V))=2$ for every $j$.
This prevents us from applying the ball inflation Lemma~\ref{lem:ball-inflation}, and a new idea seems to be needed to handle the surface \eqref{nonexample}.

\subsection*{Acknowledgement}
SG was partially supported by a direct grant for research (4053295) from the Chinese University of Hong Kong.
PZ was partially supported by the Hausdorff Center for Mathematics (DFG EXC 2047).
The authors are very grateful to the referee for their careful reading of the manuscript and numerous valuable suggestions, which significantly improved the exposition of the paper.

\section{Reduction of linear to multilinear decoupling}
\label{sec:multilinear}
We prove Theorem~\ref{thm:main} by induction on $d \geq 1$ and $k \geq 1$.
Since the formula \eqref{eq:tgamma} for the exponents does not reflect the inductive structure of this proof, it is more appropriate to use a different formula.
For a finite set of exponents $\calD \subset \N^{d}\setminus \{0\}$ with degree $k$, let
\begin{equation}
\label{eq:tGamma-recursion}
\tGamma_{\calD}(p) :=
\begin{cases}
0 & \text{if } d=0 \text{ or } k=0,\\
d \bigl( 1 - \frac{1}{p} \bigr) & \text{if } k=1,\\
\max( \max\limits_{1 \leq j \leq d} \tGamma_{\bfP_{j} \calD}(p) + \frac{1}{p}, \tGamma_{\calD \cap \calS_{k-1}}(\max(2,p \frac{\calK (\calD \cap \calS_{k-1})}{\calK (\calD)})) )
& \text{otherwise,}
\end{cases}
\end{equation}
where $\bfP_{j}$ denotes the projection onto $\N^{d-1}$ that deletes the $j$-th coordinate.
We will prove Theorem~\ref{thm:main} with $\tgamma$ replaced by $\tGamma_{\calD}(p)$.
In Section~\ref{sec:lower} we show that in fact $\tGamma_{\calD}(p) = \tgamma$.
We used formula \eqref{eq:tgamma} in Theorem~\ref{thm:main} because it is the shortest expression that we could find for these exponents.

The recursive formula \eqref{eq:tGamma-recursion} reflects the structure of the proof.
The base case of the inductive proof of Theorem~\ref{thm:main} is $k=1$, which essentially follows by interpolation between orthogonality at $p=2$ and Minkowski's inequality at $p=\infty$ (see Appendix~\ref{sec:k=1} for details).
One could also think of the trivial case $d=0$, in which the sum in \eqref{eq:dec-const-def} consists of one term, as a base case, although it is not included in the statement of Theorem~\ref{thm:main}.
These are also the base cases in the definition of $\tGamma$.

The application of lower-dimensional cases to non-transverse terms in the Bourgain--Guth argument is responsible for the lower dimensional term $\tGamma_{\bfP_{j} \calD}$ in \eqref{eq:tGamma-recursion}.
The use of lower degree decoupling in the induction on scales argument is responsible for the lower degree term $\tGamma_{\calD \cap \sublevel_{k-1}}$ in \eqref{eq:tGamma-recursion}.

Henceforth we will assume that Theorem~\ref{thm:main} is known with $\calD = \calD(\bfk, \leq k)$ replaced by $\calD(\bfk, \leq l)$ for any $1\leq l < k$.
If $d\geq 2$, then we also assume that Theorem~\ref{thm:main} is known with $\calD$ replaced by $\bfP_{j} \calD$ for any $1\leq j \leq d$ (the distinction between the cases $d=1$ and $d\geq 2$ will appear in Lemma~\ref{lem:dim-reduction-variety}, in which we deal with subvarieties of $\R^{d}$).
In the remaining part of Section~\ref{sec:multilinear} and in Section~\ref{sec:induction-on-scales} we view $d,\bfk,k,p$, and $\calD := \calD(\bfk,k)$ as fixed.

For $1 \leq l \leq k$, let $\calD_{l} := \calD \cap \sublevel_{l}$ and $n_{l} := \rk \calD_{l}$.
When $l=k$, we see that $\calD_{k}=\calD$ and $n_k=\rk \calD$.

\subsection{Transversality}
\label{sec:transversality}

Let $M$ be a positive integer and $1\leq l < k$.
For $1\le j\le M$, let $V_j \subset \R^{n_{k}}$ be a linear subspace of dimension $n_l$.
Let $\pi_j: \R^{n_{k}}\to V_j$ denote the orthogonal projection onto $V_j$.
The \emph{Brascamp--Lieb constant} $\BL((V_{j})_{j=1}^{M})$ is the smallest constant $C$ (possibly $\infty$) such that the inequality
\begin{equation}
\label{eq:BL-const-def}
\int_{\R^{n_{k}}} \prod_{j=1}^M f_j (\pi_j (x))^{\frac{n_{k}}{n_{l} M}} \dif x
\leq
C \prod_{j=1}^M \bigl( \int_{V_{j}} f_j (x) \dif x \bigr)^{\frac{n_{k}}{n_{l} M}}
\end{equation}
holds for all non-negative measurable functions $f_j: V_j\to \R$.
By scaling, $\frac{n_{k}}{n_{l} M}$ is the only exponent for which \eqref{eq:BL-const-def} can hold with a finite constant.
We recall a special case of the characterization of boundedness of Brascamp--Lieb multilinear forms due to Bennett, Carbery, Christ, and Tao.
\begin{theorem}[{\cite{MR2661170}}]
\label{thm:bcct}
The constant $\BL((V_{j})_{j=1}^{M})$ is finite if and only if
\begin{equation}
\label{eq:BL-transversality}
\dim(V) \le \frac{n_{k}}{n_{l} M} \sum_{j=1}^M \dim(\pi_j(V))
\end{equation}
holds for every linear subspace $V\subset \R^{n_{k}}$ with $0 < \dim V < n_{k}$.
\end{theorem}
We need Brascamp--Lieb inequalities with different choices of $n_l$ because the graph of $\Phi$, which is a $d$-dimensional surface, can appear to be $n_l$-dimensional at certain scales.
More precisely, we use the $l$-th order tangent spaces
\begin{equation}\label{order_tangent}
V^{(l)}(t):=\lin \Set{\partial^{\bfj} \Phi(t) \given \bfj \in \calD_{l}},
\quad t\in [0, 1]^d.
\end{equation}
Notice that $\dim V^{(l)}(t) = n_{l}$ for all $t\in [0,1]^{d}$.
For $t=0$ this is easy to see, since $\partial^{\bfj}\Phi(0)$ is a non-zero multiple of the $\bfj$-th unit vector, and for other $t$ this follows using affine symmetry of the graph of $\Phi$.

\begin{definition}\label{0713fdefi1.2}
Sets $R_1, \dotsc, R_M\subset [0, 1]^d$ are called \emph{$\nu$-transverse} if for each $1 \leq l < k$ and every choice of $x_{j} \in R_{j}$ the $l$-th order tangential spaces $V^{(l)}(x_{j})$ satisfy
\[
\BL((V^{(l)}(x_{j}))_{j=1}^{M})
\leq \nu^{-1}.
\]
\end{definition}

This definition of transversality is motivated by Lemma~\ref{lem:ball-inflation}.

\begin{remark}
While it is easy to see that no $\nu$-tranverse tuples of non-empty sets exist for small $M$, it is a priori not clear how large $M$ has to be for such tuples to exist, or whether such $M$ exists at all.
For this reason varying degree of multilinearity $M$ was introduced in \cite{MR3709122}.
\end{remark}

The next lemma says that a tuple of dyadic cubes is transverse if it is not clustered near any low degree subvariety.

\begin{lemma}
\label{lem:not-clustered-implies-transverse}
There exists $\theta = \theta(\calD) > 0$ such that for every $K \in \N_{>0}$ there exists $\nu_{K} = \nu_{K}(\calD)$ such that for every tuple of cubes $R_1, \dotsc, R_M \in \Part{K^{-1}}$ at least one of the following statements holds.
\begin{enumerate}
\item\label{it:clustered} There exists a non-zero polynomial $P$ in $d$ variables of degree $\leq D(d,k) = k^{k^{d}}$ such that $2 R_{j} \cap Z_{P} \neq \emptyset$ for at least $\theta M$ many $j$'s, or
\item\label{it:transverse} the sets $R_{1},\dotsc,R_{M}$ are $\nu$-transverse.
\end{enumerate}
\end{lemma}
Here $Z_{P} := \Set{x\in \R^d \given P(x)=0 }$ denotes the zero set of a polynomial.\\

The proof of Lemma \ref{lem:not-clustered-implies-transverse} is based on the following theorem. 
\begin{theorem}\label{thm:rank}
For every $d\ge 1$, $k\ge 2$, $1\le l\le k-1$, and every linear subspace $V=\operatorname{span} \Set{ v_1, \dotsc, v_{H} } \subset \R^{n_k}$ with $0 < \dim V < n_{k}$, the matrix
\begin{equation}\label{eq:rank-matrix}
\calM_V^{(l)}(t)
:=
\bigl( v_1, \dotsc, v_{H} \bigr)^T \times \bigl( \partial^{\bfj} \Phi(t) \bigr)_{\bfj \in \calD_{l}}
\end{equation}
satisfies at least one of the following two statements:
\begin{enumerate}
\item\label{rank_inequality}
it has a minor determinant of order
\[
\floor[\big]{\frac{\dim(V)\cdot n_l}{n_k}}+1
\]
that does not vanish identically when viewed as a function of $t\in [0, 1]^d$, or
\item\label{rank_equality}
it has a minor determinant of order
\[
\floor[\big]{\frac{\dim(V)\cdot n_l}{n_k}}=\frac{\dim(V)\cdot n_l}{n_k}
\]
that vanishes at no point $t \in [0, 1]^d$.
\end{enumerate}
\end{theorem}
We recall that $\Phi$ is the vector of monomials of all orders in $\calD$.
In particular, $\calM_V^{(l)}(t)$ is a matrix is of order $H \times n_l$.

A more precise version of Theorem~\ref{thm:rank}, Theorem~\ref{thm:BLcond}, is proved in Section~\ref{sec:transverse}.
In this section we use Theorem~\ref{thm:rank} as a black box.

\begin{proof}[Proof of Lemma~\ref{lem:not-clustered-implies-transverse}.]
For a given $K$, there are finitely many choices of $R_{1},\dotsc,R_{M}$, and for each choice the set of possible $x_{j}\in R_{j}$ is compact.
Since Brascamp--Lieb constants depend continuously on data, see \cite{MR3783217} and \cite{MR3723636}, it suffices to show that if alternative \eqref{it:clustered} of Lemma~\ref{lem:not-clustered-implies-transverse} does not hold, then the Brascamp--Lieb constant is finite for each choice of $x_{j} \in R_{j}$.
To this end it suffices to verify the transversality condition \eqref{eq:BL-transversality}.

Fix a linear space $V\subset \R^{n_k}$ with basis $(v_1, \dotsc, v_{H})$ that is not the full space and not the trivial subspace.
We need to show that
\begin{equation}\label{0713f1.11}
\dim(V)\le \frac{n_{k}}{n_{l} M} \sum_{j=1}^M \dim(\pi_j(V)).
\end{equation}
Here $\pi_j(V)$ denotes the orthogonal projection of $V$ onto $V^{(l)}(x_j)$.
Observe that $\dim(\pi_j(V))$ equals the rank of the matrix $\calM_V^{(l)}(x_j)$.
There are two cases.
If alternative \eqref{rank_inequality} of Theorem~\ref{thm:rank} holds, then the matrix $\calM_V^{(l)}(x)$ has at least one minor determinant of order at least
\begin{equation}
\label{eq:rank-lower-bound}
\floor[\big]{\frac{\dim(V)\cdot n_l}{n_k}}+1
\end{equation}
that is a non-zero polynomial in $x$.
We denote this polynomial by $P$.
Since $P$ is the determinant of a square matrix of order at most $n_{k} \times n_{k}$ whose entries are polynomials of degree at most $k$, we have
\begin{equation}
\deg P
\leq
k^{n_{k}}
\leq
k^{k^{d}}.
\end{equation}
Since we assumed that alternative \eqref{it:clustered} of Lemma~\ref{lem:not-clustered-implies-transverse} does not hold, the polynomial $P$ does not vanish at $x_{j}$ for at least $M \cdot (1-\theta)$ many $j$'s.
Hence, for these $j$'s, the matrix $\calM_V^{(l)}(x_{j})$ has rank at least \eqref{eq:rank-lower-bound}.
Hence, the right hand side of \eqref{0713f1.11} is at least
\begin{equation}
\frac{n_k}{n_l} (1-\theta) \Bigl( \floor[\big]{ \frac{\dim(V)\cdot n_l}{n_k} } + 1 \Bigr).
\end{equation}
By choosing $\theta$ small enough, the last display can be made $\geq \dim(V)$.
This finishes the proof of the estimate \eqref{0713f1.11} in the first case.

The second case is that the alternative \eqref{rank_equality} of Theorem~\ref{thm:rank} holds.
In this case $\dim (\pi_{j}(V))$ is bounded below by $\frac{n_{l}}{n_{k}} \dim V$ for every $j$, and this immediately implies \eqref{0713f1.11}.
\end{proof}

It would be desirable to replace the above compactness argument using continuity of BL constants by an explicit estimate for BL constants.

\subsection{Dimensional reduction}
\label{sec:dim-reduction}

\begin{definition}
\label{def:dim-reduction-variety}
For $2 \leq q \leq p < \infty$ and $K \geq 1$, let $\avDec_{\mathrm{var}}(\calD,p,q,K^{-1})$ denote the smallest constant $C$ such that, for every non-zero polynomial $P$ of $d$ variables with degree $\leq D(d,k)$, every collection $\calG \subset \Part{K^{-1}}$ of cubes that intersect the zero set of $P$, and any measurable functions $f_{\beta} : \R^{\calD} \to \C$ with $\supp f_{\beta} \subseteq \Uncert(\beta)$, we have
\begin{equation}
\label{eq:dim-reduction-variety}
\norm[\big]{ \sum_{\beta\in \calG} f_{\beta} }_{L^p(\R^{\calD})}
\leq
C
\Bigl( \avsum_{\beta\in\Part{K^{-1}}} \one_{\beta\in \calG} \norm{f_{\beta}}_{L^p(\R^{\calD})}^{q} \Bigr)^{1/q}.
\end{equation}
\end{definition}

\begin{lemma}
\label{lem:dim-reduction-variety}
If $d\geq 2$, then for every $2 \leq q \leq p < \infty$ and $K \geq 1$ we have
\begin{equation}
\label{eq:Dec-subvariety}
\avDec_{\mathrm{var}}(\calD,p,q,K^{-1})
\lesssim
\max_{1 \leq j \leq d} \avDec(\bfP_{j}\calD,p,q,K^{-1}) K^{1/q}.
\end{equation}
If $d=1$, then for every $2 \leq q \leq p < \infty$ and $K \geq 1$ we have
\begin{equation}
\label{eq:Dec-subvariety:d=1}
\avDec_{\mathrm{var}}(\calD,p,q,K^{-1})
\lesssim
1.
\end{equation}
\end{lemma}

Lemma~\ref{lem:dim-reduction-variety} is proved by splitting the collection $\calG$ in subcollections with boundedly overlapping projections onto coordinate hyperplanes.

For $1\le j\le d$ let the \emph{$j$-multiplicity} of a collection $\calG \subset \Part{K^{-1}}$ be the largest number $\calM_{j}(\calG)$ of cubes from $\calG$ that a line parallel to the $j$-th coordinate axis can pass through.
\begin{lemma}[{\cite[Lemma 5.4]{arxiv:1804.02488}}]
\label{lem:variety-multiplicity}
Let $P$ be a non-zero polynomial of $d$ variables and $K\geq 1$.
Let $\calG \subset \Part{K^{-1}}$ be a collection of cubes that intersect the zero set of $P$.
Then we can split $\calG = \cup_{j=1}^{d} \calG_{j}$ in such a way that $\calM_{j}(\calG_{j}) \leq C(d, \deg(P))$, where $C(d, \deg(P))$ is a constant that depends only on the dimension $d$ and the degree of $P$.
\end{lemma}

\begin{proof}[Proof of Lemma~\ref{lem:dim-reduction-variety} assuming Lemma~\ref{lem:variety-multiplicity}]
In the case $d=1$ we have $\card{\calG} \lesssim 1$, and the estimate~\eqref{eq:Dec-subvariety:d=1} follows from Minkowski's inequality.

Suppose now that $d \geq 2$.
By applying Lemma~\ref{lem:variety-multiplicity} to the collection of cubes $\calG$, we obtain at most $d\cdot C(d, \deg(P))$ many disjoint collections of $K$-cubes, each of which is of $j$-multiplicity one for some $1\le j\le d$.

Hence we may assume that $\calG$ has $j$-multiplicity one for some fixed $j$.
Let $\calD' := \calD \setminus \bfP_{j}\calD$.
Let $f_{\beta}$ with $\supp\widehat{f_{\beta}} \subseteq \Uncert(\beta)$ be arbitrary.
Then, for almost every $x' \in \R^{\calD'}$, the Fourier support of $f_{\beta}(\cdot,x')$ is contained in $\Uncert_{\bfP_{j}\calD}(\pi_{(j)}(\beta))$, where $\pi_{(j)} : \R^{d} \to \R^{d-1}$ is the projection that removes the $j$-th coordinate.
Moreover, the projections $\pi_{(j)}(\beta)$ of the cubes $\beta\in\calG$ are pairwise disjoint, since $\calG$ has $j$-multiplicity one.
It follows that
\[
\norm[\big]{ \sum_{\beta\in \calG} f_{\beta}(\cdot,x') }_{L^p(\R^{\bfP_{j}\calD})}
\leq
\avDec(\bfP_{j}\calD,p,q,K^{-1}) K^{1/q}
\Bigl( \avsum_{\beta\in\Part{K^{-1}}} \one_{\beta\in \calG} \norm{f_{\beta}(\cdot,x')}_{L^p(\R^{\bfP_{j}\calD})}^{q} \Bigr)^{1/q}.
\]
The factor $K^{1/q}$ appears because the sum over $\beta$ is normalized differently in different dimensions.
Taking the $L^{p}(\R^{\calD'})$ norm, using Minkowski's inequality, and recalling that $f_{\theta}$ are arbitrary, we obtain \eqref{eq:Dec-subvariety}.
\end{proof}

\begin{proof}[Proof of Lemma~\ref{lem:variety-multiplicity}.]
We include the proof of {\cite[Lemma 5.4]{arxiv:1804.02488}} for completeness.
The proof is inspired by an argument due to Wongkew \cite{MR1211391}.

The proof is by induction on the dimension $d$.
In the case $d=1$ we have $\calM_{1}(\calG) = \abs{\calG} \leq 2\deg P$.

Suppose that $d>1$ and that the result is already known with $d$ replaced by $d-1$.
Let
\[
\calG' := \Set{\beta \in \calG \given \partial\beta \cap Z_{P} = \emptyset }.
\]
Then each $\beta \in \calG'$ contains a distinct connected component of $Z_{P}$.
It follows from \cite[Theorem 2]{MR0161339} that the number of such connected components is at most $C(d, \deg P)$.
Hence $\abs{\calG'} \leq C(d,\deg P)$, and we can put the elements of $\calG'$ in any $\calG_{j}$.

It remains to treat $\calG \setminus \calG'$.
Let $\calH_{j}$ be the collection of affine hyperplanes perpendicular to the $j$-th coordinate direction spaced by $K^{-1}$.
Then
\[
\calG \setminus \calG'
=
\bigcup_{j=1}^{d} \bigcup_{H \in \calH_{j}} \calG_{H}
\text{ with }
\calG_{H}
:=
\Set{\beta\in\calG\setminus\calG' \given \beta \cap H \cap Z_{P} \neq \emptyset }.
\]
For each $j$ let $\calH_{j}' \subset \calH_{j}$ be the subset of hyperplanes on which $P$ vanishes identically.
Then $\abs{\calH_{j}'} \leq \deg P$, and we put all elements of $\calG_{H}$ for $H \in \calH_{j}'$ in $\calG_{j}$.
For the remaining hyperplanes $H \in \calH_{j} \setminus \calH_{j}'$, by the inductive hypothesis we have a decomposition
\[
\calG_{H} = \bigcup_{\substack{1 \leq l \leq d\\ l \neq j}} \calG_{H,l}
\]
such that the number of cubes $\beta \cap H$ with $\beta \in \calG_{H,l}$ intersecting any given line in the $l$-th coordinate direction is $O(1)$.
We put all elements of $\calG_{H,l}$ for $j\neq l$ and $H \in \calH_{j}\setminus\calH_{j}'$ in $\calG_{l}$.
\end{proof}

\subsection{Localization}
\label{sec:cutoff}
For a ball $B = B(c,R) \subset \R^\calD$ and $E > 0$ we consider the weights
\begin{equation}
\label{eq:w_B}
w_{B,E}(x)
:=
\bigl( 1 + \frac{\abs{x-c}}{R} \bigr)^{-E}.
\end{equation}
We think of the weight $w_{B,E}$ as an approximation of the characteristic function $\one_{B}$.
Typically we fix an exponent $E>\rk \calD$ and omit it from the notation: $w_{B} := w_{B,E}$.
All implicit constants are allowed to depend on $E$.

A key property of the weights \eqref{eq:w_B} is the inequality
\begin{equation}
\label{eq:1-w}
\one_B
\lesssim
\sum_{B' \in \calB(B,R)} w_{B'}
\lesssim
w_B,
\end{equation}
which holds for all balls $B \subset \R^{n}$ and all $0<R$ that are smaller than the radius of $B$.
Here and later $\calB(B,R)$ denotes a boundedly overlapping covering of a set $B$ by balls of radius $R$.
The implicit constants in \eqref{eq:1-w} do not depend on $B$ and $R$.

The following result allows one to deduce inequalities for $L^{p}(w_{B})$ norms from inequalities for $L^{p}(\one_{B})$ norms.
It is necessitated by the fact that the reverse of the inequalities in \eqref{eq:1-w} do not hold.
\begin{lemma}[{\cite[Lemma 4.1]{MR3592159}}]
\label{lem:1-w}
Let $\calW$ be the collection of all weights, that is, positive, integrable functions on $\R^n$.
Fix $R>0$ and $E>n$.
Let $O_1,O_2 : \calW\to[0,\infty]$ be any functions with the following properties.
\begin{enumerate}
\item\label{it:O:1<w} $O_1(\one_B)\leq O_2(w_{B,E})$ for all balls $B\subset \R^n$ with radius $R$
\item\label{it:O:subadd} $O_1(\alpha u+\beta v)\le \alpha O_1(u)+\beta O_1(v)$, for each $u,v\in\calW$ and $\alpha,\beta > 0$
\item\label{it:O:supadd} $O_2(\alpha u+\beta v)\ge \alpha O_2(u)+\beta O_2(v)$, for each $u,v\in\calW$ and $\alpha,\beta > 0$
\item\label{it:O:monotone} If $u\le v$, then $O_i(u)\le O_i(v)$.
\item\label{it:O:continuous} If $(u_{j})_{j=1}^{\infty} \subset \calW$ is a monotonically increasing sequence with $u_{j} \to u \in \calW$ as $j\to \infty$ pointwise almost everywhere, then $O_1(u) = \lim_{j\to \infty} O_1(u_{j})$.
\end{enumerate}
Then for each ball $B \subset \R^{n}$ with radius $R$ we have
\[
O_1(w_{B,E})
\lesssim_{n,E}
O_2(w_{B,E})
\]
The implicit constant depends only on $n$ and $E$.
\end{lemma}

\begin{proof}
Let $\calB := \calB(\R^n,R)$.
Note that
\[
w_B(x)
\leq
C \sum_{B'\in\calB} w_B(c_{B'}) \one_{B'}(x)
\]
and that
\[
\sum_{B'\in\calB}w_B(c_{B'}) w_{B'}(x)
\leq
C w_B(x)
\]
for a sufficiently large constant $C=C(n,E)>0$.
Hence,
\begin{align*}
O_{1}(w_{B})
&\leq
\sup_{\calB' \subset \calB \ \textrm{finite}} O_{1}\bigl( C \sum_{B'\in\calB'} w_B(c_{B'}) \one_{B'} \bigr)
&\text{by \eqref{it:O:continuous}}\\
&\leq
\sup_{\calB' \subset \calB \ \textrm{finite}} C \sum_{B'\in\calB'} w_B(c_{B'}) O_{1}( \one_{B'} )
&\text{by \eqref{it:O:subadd}}\\
&\leq
C \sup_{\calB' \subset \calB \ \textrm{finite}} \sum_{B'\in\calB} w_B(c_{B'}) O_{2}(w_{B'})
&\text{by \eqref{it:O:1<w}}\\
&\leq
C^{2} \sup_{\calB' \subset \calB \ \textrm{finite}} O_{2}\bigl( C^{-1} \sum_{B'\in\calB} w_B(c_{B'}) w_{B'}  \bigr)
&\text{by \eqref{it:O:supadd}}\\
&\leq
C^{2} O_{2}(w_{B}).
&\text{by \eqref{it:O:monotone}}
&\qedhere
\end{align*}
\end{proof}

\begin{remark}
\label{rem:1-w}
Lemma~\ref{lem:1-w} will be usually applied with functionals of the form
\begin{align}
O_1(v) &:= \norm{f}_{L^p(v)}^p \label{eq:O_1:practical}\\
O_2(v) &:= A (\sum_{i}\norm{f_i}_{L^p(v)}^{q})^{\frac{p}{q}}, \label{eq:O_2:practical}
\end{align}
where $1 \leq q \leq p$.
See for instance the proof of Corollary~\ref{cor:scaled+loc-decoupling} and the proof of Theorem~\ref{thm:multilinear-to-linear} (but not the proof Corollary~\ref{cor:rev-holder}). 
It is clear that conditions \eqref{it:O:subadd} and \eqref{it:O:monotone} hold for these choices.
Condition \eqref{it:O:supadd} follows from the reverse triangle inequality in $\ell_{\frac{q}{p}}$.
Condition \eqref{it:O:continuous} is very mild and follows from the monotone convergence theorem.
The main hypothesis is the condition \eqref{it:O:1<w}.
\end{remark}

We close this section with the following reverse H\"older inequality.
\begin{corollary}[{cf.\ \cite[Corollary 4.1]{MR3592159}}]
\label{cor:rev-holder}
For each $1 \leq t \leq p < \infty$, each $E>n$, each $R>0$ and $\delta>0$ with $R\delta \geq 1$, each function $f : \R^{n} \to \C$ with $\diam(\supp \hat{f}) \lesssim \delta$, and each ball $B \subset \R^n$ with radius $R$, we have
\begin{equation}
\label{eq:rev-holder}
\norm{f}_{\avL^p(w_{B,E})}
\lesssim
(R\delta)^{n/t-n/p}
\norm{f}_{\avL^t (w_{B,\frac{E t}{p}})},
\end{equation}
with the implicit constant independent of $R$, $\delta$, $B$, and $f$.
\end{corollary}
\begin{notation}
Here and later we denote normalized $L^{p}$ norms by
\begin{equation}
\label{eq:avL}
\norm{f}_{\avL^{p}(B)} := \abs{B}^{-1/p} \norm{f}_{L^{p}(B)},
\quad
\norm{f}_{\avL^{p}(w_{B})} := \abs{B}^{-1/p} \norm{f}_{L^{p}(w_{B})}.
\end{equation}
\end{notation}
\begin{proof}[Proof of Corollary \ref{cor:rev-holder}.]
By translation and modulation we may assume that $B$ is centered in $0$ and $\supp \hat{f} \subset B(0,C\delta)$.

Let $\eta$ be a positive Schwartz function on $\R^n$ with $\one_{B(0,1)}\le \eta$ and such that $\supp (\widehat{\eta}) \subset B(0,1)$.
We can thus write
\[
\norm{f}_{L^p(B)}
\le
\norm{\eta_B f}_{L^p(\R^n)},
\]
where $\eta_B(\cdot):=\eta(\cdot/R)$.
Let $\theta$ be a Schwartz function on $\R^{n}$ such that $\hat{\theta}(\xi) = 1$ for $\abs{\xi} \leq 2C$.
Since
\[
\supp \widehat{\eta_{B} f}
\subseteq
\supp \widehat{\eta_{B}} + \supp \hat{f}
\subseteq
B(0,1/R) + B(0,C\delta)
\subseteq
B(0,2C\delta),
\]
we have that
\[
\eta_B f
=
(\eta_B f)* \theta_Q,
\]
where $\theta_Q(x):= \delta^{n}\theta(\delta x)$.
By Young's convolution inequality with exponents
\[
\frac{1}{p}=\frac{1}{t}+\frac{1}{r}-1=\frac{1}{t}-\frac1{r'},
\]
we can write
\[
\norm{\eta_B f}_{L^p(\R^n)}
\le
\norm{\eta_B f}_{L^t(\R^n)} \norm{\theta_Q}_{L^r(\R^n)}
\lesssim
\delta^{n/r'} \norm{f}_{L^t(\eta_B^{t})}.
\]
Rearranging this inequality and estimating $\eta_{B} \lesssim w_{B,E}^{1/p}$, we obtain
\[
\abs{B}^{-1/p} \norm{f}_{L^p(B)}
\lesssim_{n,E}
(R\delta)^{n/r'} \abs{B}^{-1/t} \norm{f}_{L^t(w_{B,E}^{t/p})}
\]
for any $E>0$.
Now we can apply Lemma~\ref{lem:1-w} with
\begin{align*}
O_{1}(v) &:= R^{-n} \int_{\R^{n}} \abs{f}^{p} v,\\
O_{2}(v) &:= A (R\delta)^{p(n/t-n/p)} R^{-n p/t} \Bigl( \int_{\R^{n}} \abs{f}^{t} v^{t/p} \Bigr)^{p/t},
\end{align*}
for some large constant $A$. 
\end{proof}

\subsection{Affine scaling}
\label{sec:scaling}
Let $\alpha \in \Part{\sigma}$ with $\sigma \in 2^{-\N}$.
Denote by $c_{\alpha}$ the lowest corner of $\alpha$ (with respect to coordinatewise ordering).
For a function $f$ on $\R^{\calD}$ let
\begin{equation}
\label{eq:scaling+modulation}
M_{\alpha}f(x) := e\bigl( \sum_{\bfi \in\calD} (c_{\alpha})^{\bfi} x_{\bfi} \bigr) f (L_{\alpha} x),
\end{equation}
where
\begin{equation}
\label{eq:affine-scaling:space}
\begin{split}
L_{\alpha} &= L_{\alpha,\mathrm{scale}}L_{\alpha,\mathrm{shear}},
\\
(L_{\alpha,\mathrm{scale}}(x))_{\bfj} &=
\sigma^{\abs{\bfj}} x_{\bfj},
\\
(L_{\alpha,\mathrm{shear}}(x))_{\bfj} &=
\sum_{\bfi \in\calD} \binom{\bfi}{\bfj} (c_{\alpha})^{\bfi-\bfj} x_{\bfi}.
\end{split}
\end{equation}
Then
\[
\widehat{M_{\alpha}f}(\xi) = \sigma^{-\calK(\calD)} \widehat{f}( \hat{M}_{\alpha}(\xi) ),
\]
where
\[
\hat{M}_{\alpha}(\xi) = L_{\alpha}^{-*}(\xi - ((c_{\alpha})^{\bfi})_{\bfi}).
\]
Let $\Uncert_{\calD}(\alpha) := \hat{M}_{\alpha}([-2,2]^{\calD})$; this is essentially the smallest parallelepiped that contains the moment surface over the cube $\alpha$.
We omit the subscript $\calD$ from $\Uncert_{\calD}$ unless several $\calD$'s are involved.

\begin{definition}
We will denote by $f_{\theta}$ functions such that $\supp \widehat{f_{\theta}} \subseteq \Uncert_{\calD}(\theta)$.
Given $0 < \delta \leq 1$ and a collection of functions $f_{\theta}$ with $\theta \in \Part{\delta}$, we write
\[
f_{\alpha} := \sum_{\theta \in \Part[\alpha]{\delta}} f_{\theta}
\]
for dyadic cubes $\alpha$ with side length $\geq \delta$.
\end{definition}

\begin{lemma}
\label{lem:scaled-decoupling}
Let $1 \leq q \leq p < \infty$, $\epsilon>0$, $0<\delta\leq \sigma \leq 1$ (with $\sigma \in 2^{-\N}$), and $\alpha \in \Part{\sigma}$.
Then
\begin{equation}
\label{eq:scaled-decoupling}
\norm{\sum_{\theta\in \Part[\alpha]{\delta}} f_{\theta}}_{L^p(\R^{\calD})}
\leq
\avDec(\calD_{k}, p, q, \delta/\sigma)
\Bigl( \avsum_{\theta\in \Part[\alpha]{\delta}} \norm{f_{\theta}}_{L^p(\R^{\calD})}^q \Bigr)^{1/q}.
\end{equation}
\end{lemma}
\begin{proof}
We can write $f_{\theta} = M_{\alpha} g_{\theta'}$ with $\theta'\in\Part{\delta/\sigma}$ and $\supp \widehat{g_{\theta'}} \subseteq \Uncert(\theta')$, where $M_{\alpha}$ is defined by \eqref{eq:scaling+modulation}.
Both sides of the inequality \eqref{eq:scaled-decoupling} scale in the same way under $M_{\alpha}$, and the conclusion follows.
\end{proof}

\begin{corollary}
\label{cor:scaled+loc-decoupling}
Let $1 \leq l \leq k$, $1 \leq q \leq p < \infty$, $\epsilon>0$, $0<\delta\leq \sigma \leq 1$ (with $\sigma \in 2^{-\N}$), and $\alpha \in \Part{\sigma}$.
Then, for every ball $B \subset \R^{\calD}$ of radius $\delta^{-l}$, we have
\begin{equation}
\label{eq:scaled+loc-decoupling}
\norm{\sum_{\theta\in \Part[\alpha]{\delta}} f_{\theta}}_{L^p(w_{B})}
\lesssim
\avDec(\calD_{l}, p, q, \delta/\sigma)
\Bigl( \avsum_{\theta\in \Part[\alpha]{\delta}} \norm{f_{\theta}}_{L^p(w_{B})}^q \Bigr)^{1/q}.
\end{equation}
\end{corollary}
\begin{proof}
Writing $\R^{\calD} = \R^{\calD_{l}} \times \R^{\calD \setminus \calD_{l}}$ and considering the fibers over each $x' \in \R^{\calD \setminus \calD_{l}}$ separately, we may assume $l=k$.

By Lemma~\ref{lem:1-w}, \eqref{eq:scaled+loc-decoupling} will follow from
\[
\norm{\sum_{\theta\in \Part[\alpha]{\delta}} f_{\theta}}_{L^p(B)}
\lesssim
\avDec(\calD_{k}, p, q, \delta/\sigma)
\Bigl( \avsum_{\theta\in \Part[\alpha]{\delta}} \norm{f_{\theta}}_{L^p(w_{B})}^q \Bigr)^{1/q}.
\]
Let $\phi_{B}$ be a smooth bump function adapted to $B$, in the sense that $\one_{B} \lesssim \phi_{B} \lesssim w_{B,E}$ for every $E$, with $\supp \widehat{\phi_{B}} \subseteq B(0,\delta^{-l})$.
Then
\[
\norm{\sum_{\theta\in \Part[\alpha]{\delta}} f_{\theta}}_{L^p(B)}
\lesssim
\norm{\sum_{\theta\in \Part[\alpha]{\delta}} \phi_{B}f_{\theta}}_{L^p(\R^{\calD})}.
\]
For each $\theta\in \Part[\alpha]{\delta}$, we have
\[
\supp \widehat{\phi_{B}f_{\theta}}
\subseteq
\supp \widehat{\phi_{B}} + \supp \widehat{f_{\theta}}
\subseteq
\Uncert(\hat{\theta}),
\]
where $\hat{\theta}$ denotes the dyadic parent of $\theta$.
Hence, applying Lemma~\ref{lem:scaled-decoupling} with $\delta$ replaced by $2\delta$, we obtain
\begin{align*}
\norm{\sum_{\theta\in \Part[\alpha]{\delta}} \phi_{B}f_{\theta}}_{L^p(\R^{\calD})}
&\leq
\avDec(\calD_{k}, p, q, 2\delta/\sigma)
\Bigl( \avsum_{\theta'\in \Part[\alpha]{2\delta}} \norm{\sum_{\theta\in\Part[\theta']{\delta}}\phi_{B}f_{\theta}}_{L^p(\R^{\calD})}^q \Bigr)^{1/q}
\\ &\lesssim
\avDec(\calD_{k}, p, q, \delta/\sigma)
\Bigl( \avsum_{\theta\in \Part[\alpha]{\delta}} \norm{f_{\theta}}_{L^p(w_{B})}^q \Bigr)^{1/q}.
\qedhere
\end{align*}
\end{proof}

\subsection{Bourgain--Guth argument}
\label{sec:bourgain-guth}
We will use a few pieces of notation that will help us to keep multilinear expressions short.
\begin{notation}
For a sequence of real numbers $\Set{A_i}_{i=1}^M$, we abbreviate $\avprod A_{i} := \bigl(\prod_{i=1}^{M} A_{i}\bigr)^{1/M}$.
We also write
\[
L^{p}_{x \in \R^{\calD}}F(x) := \norm{ F }_{L^{p}(\R^{\calD})},
\quad
\ell^{q}_{\theta\in\calJ} B_{\theta} := ( \sum_{\theta\in\calJ} \abs{B_{\theta}}^{q} )^{1/q},
\quad\text{and }
\avell^{q}_{\theta\in\calJ} B_{\theta} := ( \avsum_{\theta\in\calJ} \abs{B_{\theta}}^{q} )^{1/q}.
\]
\end{notation}

\begin{definition}
For a positive integer $K$ and $0 < \delta < K^{-1}$, we denote by $\avDec(\calD, p, q, \delta, K, \nu)$ the smallest constant $C$ such that the inequality
\begin{equation}
\label{eq:multilin-dec-const-KM}
L^{p}_{x\in\R^{\Dset}} \avprod \norm{ \sum_{\theta \in \Part[R_i]{\delta}} f_{\theta} }_{\avL^{p}(B(x,K))}
\le C
\avprod \avell^{q}_{\theta \in \Part[R_i]{\delta}} \norm{f_\theta}_{L^p(\R^{\Dset})}
\end{equation}
holds for all $\nu$-transverse tuples $R_{1},\dotsc,R_{M} \in \Part{K^{-1}}$ with $1\leq M \leq K^{d}$.

In the case $q=p$ we write $\avDec(\calD, p, \delta, K, \nu) := \avDec(\calD, p, q, \delta, K, \nu)$.
\end{definition}
This kind of multi-linear decoupling constant with varying degree of multilinearity $M$ first appeared in \cite{MR3709122} in the decoupling literature.
It can be contrasted with Wooley's efficient congruencing approach to Vinogradov mean value estimates, which only uses bilinear, rather than multilinear, expressions, see the auxiliary mean value in \cite[(2.1)]{arxiv:1708.01220}.
For the cubic moment curve, a bilinear approach to decoupling inequalities is also taken in \cite{arxiv:1906.07989}.

\begin{theorem}
\label{thm:multilinear-to-linear}
For any $2 \leq p < \infty$ and $\epsilon>0$, there exists $K\geq 1$ such that for all $0 < \delta < 1$ we have
\begin{equation}
\avDec(\calD, p, \delta)
\lesssim
\delta^{-\tGamma'-\epsilon}
+ \log_{+} \delta \max_{\delta\le \delta'\le 1}(\delta/\delta')^{-\tGamma'-\epsilon} \avDec(\calD, p, \delta', K, \nu_{K}),
\end{equation}
where
\begin{equation}
\label{eq:tGamma'}
\tGamma' := \max_{1\leq j\leq d} \tGamma_{\bfP_{j}\calD}(p) + \frac{1}{p}.
\end{equation}
\end{theorem}
Here and later
\begin{equation}
\label{eq:log+}
\log_{+}\delta := \max(\abs{\log\delta}, 1).
\end{equation}
Theorem~\ref{thm:multilinear-to-linear} is obtained by iterating Corollary~\ref{cor:bourgain-guth-arg:scaled}, which is a rescaled version of the following Proposition~\ref{prop:bourgain-guth-arg}.
This iteration goes back to \cite{MR2860188}.
\begin{proposition}
\label{prop:bourgain-guth-arg}
For every $2 \leq q \leq p < \infty$, $K\ge 2$, and $0<\delta<K^{-1}$ we have
\begin{equation}
\begin{split}
\label{eq:BG-arg}
\norm{f}_{L^p(\R^{\Dset})}
&\lesssim
K^{d/q} \Bigl( \avsum_{\alpha\in \Part{K^{-1}}} \norm{f_{\alpha}}_{L^p(\R^{\Dset})}^{q} \Bigr)^{1/q}\\
&+
(\log K) \avDec_{\mathrm{var}}(\calD,p,q,K^{-1/k}) \Bigl( \avsum_{\beta\in \Part{K^{-1/k}}} \norm{f_{\beta}}_{L^p(\R^{\Dset})}^{q} \Bigr)^{1/q}\\
&+ C_{K} \avDec(\calD, p, q, \delta, K, \nu_{K}) \Bigl( \avsum_{\theta\in\Part{\delta}} \norm{f_{\theta}}_{L^p(\R^{\Dset})}^{q} \Bigr)^{1/q}.
\end{split}
\end{equation}
Recall that $\avDec_{\mathrm{var}}$ was introduced in \eqref{eq:dim-reduction-variety}.
Morever, $C_{K}$ is a constant depending on $K,d,p$, whose growth rate in $K$ is not important.
\end{proposition}
\begin{proof}[Proof of Proposition~\ref{prop:bourgain-guth-arg}]
For each $B' = B(x,K)$ we will cover the portion of $[0,1]^{d}$ with large contribution to $\norm{f_{[0,1]^{d}}}_{L^{p}(B')}$ by a small number of ``lower dimensional'' collections $\calG_{m}(B') \subset \Part{K^{-1/k}}$, each of which is close to a subvariety of bounded degree, and a transverse collection $\calT(B') \subset \Part{K^{-1}}$.
We choose the cubes in the collections $\calG_{m}(B')$ to have side length $K^{-1/k}$ in order to use lower dimensional decoupling at spatial scale $K$ for these collections.

In the following inductive algorithm we will construct collections $\calS_{m}(B') \subseteq \Part{K^{-1}}$ and $\calG_{m}(B') \subseteq \Part{K^{-1/k}}$.
Initialise
\begin{equation}
\label{eq:initial-stock}
\calS_{0}(B') := \Set{ \alpha \in \Part{K^{-1}} \given \norm{f_\alpha}_{\avL^{p}(B')} \ge K^{-d} \max_{\alpha' \in \Part{K^{-1}}} \norm{f_{\alpha'}}_{\avL^{p}(B')} }.
\end{equation}
For $m\geq 0$ we repeat the following algorithm.
Suppose that $\calS_{m}(B') \neq \emptyset$ and there is a polynomial $Q = Q_{m,B'} \in \R[\xi_{1},\dotsc,\xi_{d}]$ with $\deg Q \leq D(d,k)$ (as defined in alternative \ref{it:clustered} of Lemma~\ref{lem:not-clustered-implies-transverse}) such that
\begin{equation}
\label{eq:stock-near-variety}
\abs{\Set{ \alpha\in \calS_{m}(B') \given 2\alpha \cap Z_{Q_{m,B'}} \neq \emptyset}}
>
\theta \abs{\calS_{m}(B')},
\end{equation}
where $\theta > 0$ is given by Lemma~\ref{lem:not-clustered-implies-transverse}.
Then we choose one such $Q_{m,B'}$, let
\begin{align*}
\calG_{m}(B')
&:=
\Set{ \beta\in \Part{K^{-1/k}} \given 2\beta \cap Z_{Q_{m,B'}} \neq \emptyset} \setminus (\calG_{0}(B') \cup \dotsb \cup \calG_{m-1}(B')),\\
\calS_{m+1}(B')
&:=
\calS_{m}(B') \setminus \bigcup_{\beta \in \calG_{m}(B')} \Part[\beta]{K^{-1}},
\end{align*}
and repeat the algorithm.

Since in each step we remove at least a fixed proportion $\theta$ of $\calS_{m}(B')$, this algorithm terminates after $O(\log K)$ steps.
In the end we set
\[
\calT(B') := \calS_{m}(B').
\]
If $\calT(B') \neq \emptyset$, then in the last step of the selection algorithm there was no polynomial $Q$ with $\deg Q \leq D(d,k)$ for which \eqref{eq:stock-near-variety} holds.
Thus alternative~\ref{it:clustered} of Lemma~\ref{lem:not-clustered-implies-transverse} is violated for the collection $\calT(B')$.
Therefore, alternative~\ref{it:transverse} of Lemma~\ref{lem:not-clustered-implies-transverse} holds, that is, the cubes in $\calT(B')$ are $\nu_{K}$-transverse for some $\nu_K>0$.

We estimate
\begin{align}
\norm{ f }_{\avL^{p}(B')}
\label{eq:BG':small}&\leq
\sum_{\alpha \in \Part{K^{-1}} \setminus \calS_{0}(B')} \norm{ f_{\alpha} }_{\avL^{p}(B')}\\
\label{eq:BG':variety}&+
\sum_{m \lesssim \log K} \norm{ \sum_{\beta \in \calG_{m}(B')} f_{\beta} }_{\avL^{p}(B')}\\
\label{eq:BG':transverse}&+
\sum_{\alpha \in \calT(B')} \norm{ f_{\alpha} }_{\avL^{p}(B')}
\end{align}
By definition of $\calS_{0}(B')$, we obtain
\[
\eqref{eq:BG':small}
\lesssim
\max_{\alpha' \in \Part{K^{-1}}} \norm{ f_{\alpha'} }_{\avL^{p}(B')}.
\]
By Definition~\ref{def:dim-reduction-variety} and Lemma~\ref{lem:1-w}, we have
\begin{align*}
\eqref{eq:BG':variety}
&\lesssim_{\epsilon}
\sum_{m\lesssim \log K} \avDec_{\mathrm{var}}(\calD,p,q,K^{-1/k})
\avell^{q}_{\beta \in \Part{K^{-1/k}}} \one_{\beta \in \calG_{m}(B')} \norm{ f_{\beta} }_{\avL^{p}(w_{B'})}
\\ &\lesssim
(\log K) \avDec_{\mathrm{var}}(\calD,p,q,K^{-1/k})
\avell^{q}_{\beta \in \Part{K^{-1/k}}} \norm{ f_{\beta} }_{\avL^{p}(w_{B'})}.
\end{align*}
If $\calT(B') \neq \emptyset$, then, since by the definition of $\calS_{0}(B')$ all $\norm{f_\alpha}_{\avL^{p}(B')}$, $\alpha \in \calT(B')$, are comparable up to a factor $K^{C}$, we obtain
\[
\eqref{eq:BG':transverse}
\lesssim
K^{C}
\min_{\alpha \in \calT(B')} \norm{f_\alpha}_{\avL^{p}(B')}
\leq
K^{C} \max_{\substack{\calT \subseteq \Part{K^{-1}}\\ \nu_{K}-\text{transverse}}} \prod_{\alpha \in \calT} \norm{ f_{\alpha} }_{\avL^{p}(B')}^{1/\abs{\calT}}.
\]
It follows that
\begin{align}
\notag
\norm{f }_{L^{p}(\R^{\Dset})}
&=
L^{p}_{x\in\R^{\Dset}} \norm{f }_{\avL^{p}(B(x,K))}\\
\label{eq:BG:small} & \lesssim
L^{p}_{x\in\R^{\Dset}} \max_{\alpha \in \Part{K^{-1}}} \norm{f_{\alpha} }_{\avL^{p}(B(x,K))}\\
\label{eq:BG:variety}&+
(\log K) \avDec_{\mathrm{var}}(\calD,p,q,K^{-1/k})
L^{p}_{x\in\R^{\Dset}} \avell^{q}_{\beta \in \Part{K^{-1/k}}} \norm{ f_{\beta}  }_{\avL^{p}(w_{B(x,K)})}\\
\label{eq:BG:transverse}&+
K^{C} L^{p}_{x\in\R^{\Dset}} \max_{\substack{\calT \subseteq \Part{K^{-1}}\\ \nu_{K}-\text{transverse}}} \prod_{\alpha\in\calT} \norm{ f_{\alpha}  }_{\avL^{p}(B(x,K))}^{1/\abs{\calT}}.
\end{align}
The terms \eqref{eq:BG:small} and \eqref{eq:BG:variety} can be estimated as claimed in \eqref{eq:BG-arg} using Minkowski's inequality.
In the last term, using \eqref{eq:multilin-dec-const-KM}, we estimate
\begin{align*}
& \eqref{eq:BG:transverse}
\leq
K^{C} \Bigl( \sum_{\substack{\calT \subseteq \Part{K^{-1}}\\ \nu_{K}-\text{transverse}}} \Bigl( L^{p}_{x\in\R^{\Dset}} \prod_{\alpha\in\calT} \norm{ f_{\alpha}  }_{\avL^{p}(B(x,K))}^{1/\abs{\calT}} \bigr)^{p} \Bigr)^{1/p}\\
& \leq
K^{C} \avDec(\calD, p, q, \delta, K, \nu_{K})
\bigl( \sum_{\substack{\calT \subseteq \Part{K^{-1}}\\ \nu_{K}-\text{transverse}}} \prod_{\alpha\in\calT} \bigl( \avsum_{\theta \in \Part[\alpha]{\delta}} \norm{ f_{\theta}  }_{L^{p}(\R^{\Dset})}^{q} \bigr)^{\frac{p}{q \abs{\calT}}} \bigr)^{1/p}\\
& \leq
K^{C} 2^{K^{d}/p} \avDec(\calD, p, q, \delta, K, \nu_{K})
\bigl( \max_{\substack{\calT \subseteq \Part{K^{-1}}\\ \nu_{K}-\text{transverse}}} \prod_{\alpha\in\calT} \bigl( \avsum_{\theta \in \Part{\delta}} \norm{ f_{\theta}  }_{L^{p}(\R^{\Dset})}^{q} \bigr)^{\frac{p}{q \abs{\calT}}} \bigr)^{1/p}\\
& \leq
K^{C} 2^{K^{d}/p} \avDec(\calD, p, q, \delta, K, \nu_{K})
\bigl( \avsum_{\theta \in \Part{\delta}} \norm{ f_{\theta}  }_{L^{p}(\R^{\Dset})}^{q} \bigr)^{\frac{1}{q}}.
\qedhere
\end{align*}
\end{proof}

\begin{remark}
The loss of $K^{C}$ in the estimate for \eqref{eq:BG':transverse} can be replaced by $\log K$ by splitting \eqref{eq:initial-stock} into $O(\log K)$ collections with comparable $\norm{f_\alpha}_{\avL^{p}(B')}$; this is also useful in the setting of \cite{MR3374964}.
It would be interesting to know whether the more substantial loss of $2^{K^{d}/p}$ in the estimate for \eqref{eq:BG:transverse} can be avoided.
\end{remark}

\begin{corollary}
\label{cor:bourgain-guth-arg:scaled}
In the situation of Proposition~\ref{prop:bourgain-guth-arg}, we have
\begin{multline}
\label{eq:BG-arg:scaled}
\avDec(\calD, p, \delta)
\lesssim_{\epsilon} \max\Bigl(
K^{d/p} \avDec(\calD, p, K \delta),
K^{\tGamma'/k+\epsilon} \avDec(\calD, p, K^{1/k} \delta),\\
C_{K} \avDec(\calD, p, \delta, K, \nu_{K}) \Bigr).
\end{multline}
\end{corollary}

\begin{proof}
Use Lemma~\ref{lem:scaled-decoupling} and Lemma \ref{lem:dim-reduction-variety} in the first two terms on the right-hand side of \eqref{eq:BG-arg}, and take the supremum over all $f_{\theta}$.
\end{proof}

\begin{proof}[Proof of Theorem~\ref{thm:multilinear-to-linear}]
Observe $d/p \leq \tGamma'$.
Choose $K \in 2^{k\N}$ so large that the implicit constant on the right-hand side of \eqref{eq:BG-arg:scaled} is bounded by $K^{\epsilon}$.
For $\delta < K^{-1}$ iterate the inequality \eqref{eq:BG-arg:scaled} (at most) $\floor[\big]{k \frac{\log \delta}{\log K}}$ times and use a trivial estimate for $\avDec$ at the end.
\end{proof}

\section{Induction on scales}
\label{sec:induction-on-scales}
Let $\eta$ be the optimal exponent in Theorem~\ref{thm:main}, that is, the infimum of all $a>0$ such that the estimate
\[
\avDec(\calD,p,\delta) \lesssim_{a} \delta^{-a}
\]
holds.
By Minkowski's inequality we can see that $\eta < \infty$.

If $\eta \leq \tGamma'$, where $\tGamma'$ was defined in \eqref{eq:tGamma'}, then Theorem~\ref{thm:main} holds, so we assume without loss of generality that $\eta > \tGamma'$.
In this case, Theorem~\ref{thm:main} will follow if we can show that $\eta \leq \tGamma''$, where the latter quantity is defined by
\begin{equation}
\label{eq:tGamma''}
\tGamma'' := \max_{1 \leq l < k}
\tGamma_{\calD \cap \calS_{l}}(\max(2,p \frac{\calK (\calD \cap \calS_{l})}{\calK (\calD)})).
\end{equation}
After reducing to the multilinear quantities \eqref{eq:A} below using Theorem~\ref{thm:multilinear-to-linear}, this is accomplished by iterating several estimates that are summarized in Figure~\ref{fig:graph}, following \cite{MR3548534}.

Throughout this section, $R_{1},\dotsc,R_{M} \in \Part{K^{-1}}$ will denote a tuple of $\nu_{K}$-transverse cubes.
For scale parameters $\delta$, $0<b \leq 1$, and $0<s\leq k$ such that $0 < \delta^{b} \leq K^{-1}$ and Lebesgue exponents $1 \leq t \leq p$, define
\begin{equation}
\label{eq:A}
\tA_{p,t} (b, s)
:=
L^{p}_{x\in\R^{\Dset}} \avprod \avell^{t}_{J \in \Part[R_i]{\delta^{b}}} \norm{f_{J} }_{\avL^t(w_{B(x,\delta^{-s})})}.
\end{equation}
This convention for $\tA$ differs from previous articles in that we use an average sum over $J$.
This convention makes $\tA_{p,t}$ monotonically increasing in $t$.

The goal of the iterative procedure is to increase the parameter $b$ in \eqref{eq:A} using already established lower degree decoupling inequalties.
Decoupling inequalities of different degrees have to be applied at different Lebesgue exponents $t$.
This leads to repeated use of H\"older's inequality to pass between different exponents, and as a result the number of terms in our estimates doubles in each step of the iteration.
It is therefore more convenient to formulate the iteration as a bootstrapping argument involving the quantities \eqref{eq:a_*} and \eqref{eq:tilde-a}, in which all relevant information is distilled.

\subsection{Notation}
Let $\calK_l := \calK(\calD \cap \calS_{l})$.
For $1 \leq l \leq k$ define
\begin{align*}
q_{l} &:= \max\{2, p \frac{\calK_{l}}{\calK_{k}}\},\\
t_{l} &:= \max\{2, p \frac{n_{l}}{n_{k}}\}.
\end{align*}
Define $\alpha_l$ and $\beta_l$ by
\begin{align}
\label{eq:alpha}
\frac{1}{\frac{n_{l}}{n_{k}}}
&=
\frac{\alpha_l}{\frac{n_{l+1}}{n_{k}}}+\frac{1-\alpha_l}{\frac{\calK_{l}}{\calK_{k}}},
& 1 \leq l < k,\\
\label{eq:beta}
\frac{1}{\frac{\calK_{l}}{\calK_{k}}}
&=
\frac{1-\beta_l}{\frac{\calK_{l-1}}{\calK_{k}}}+\frac{\beta_l}{\frac{n_{l}}{n_{k}}},
& 1 < l < k,
\end{align}
and $\beta_{1}:=1$.
We claim that $\alpha_{l},\beta_{l} \in [0,1]$ for $1 \leq l < k$.
This will follow from
\begin{equation}
\label{eq:plpk<nlnk}
\frac{\calK_{l}}{\calK_{k}} \leq \frac{n_{l}}{n_{k}},
\quad
0 \leq l \leq k.
\end{equation}
Indeed,
\[
\eqref{eq:plpk<nlnk}
\iff
\frac{\calK_{k}}{\calK_{l}} \geq \frac{n_{k}}{n_{l}}
\iff
\frac{\calK_{k}-\calK_{l}}{\calK_{l}} \geq \frac{n_{k}-n_{l}}{n_{l}}.
\]
Now we write the left-hand side of the last inequality as
\[
\frac{\sum_{j=l+1}^{k} j (n_{j}-n_{j-1})}{\sum_{j=1}^{l} j (n_{j}-n_{j-1})}
\geq
\frac{\sum_{j=l+1}^{k} (n_{j}-n_{j-1})}{\sum_{j=1}^{l} (n_{j}-n_{j-1})}
=
\frac{n_{k}-n_{l}}{n_{l}}.
\]
This finishes the proof of \eqref{eq:plpk<nlnk}.

The induction on scales argument will involve the quantities
\[
\tA_{t(l)}(b) := \tA_{p,t_{l}}(b,lb),
\]
\[
\tA_{q(l)}(b) := \tA_{p,q_{l}}(b,(l+1)b).
\]
Here $t(l)$ and $q(l)$ are formal expressions and can be read ``of type $t$ with degree $l$'' and ``of type $q$ with degree $l$''.
For $0<b<1$ and $*=t(l),q(l)$ let
\begin{equation}
\label{eq:a_*}
a_{*}(b) := \inf \Set{ a \given \tA_{*}(b) \lesssim_{a,K} \delta^{-a} \RHS\eqref{eq:multilin-dec-const-KM} \text{ for all } K }.
\end{equation}

\subsection{Entering the iterative procedure}
First we estimate the left-hand side of \eqref{eq:multilin-dec-const-KM} by the quantities involved in the iterative procedure.
For $1 \leq l \leq k$, $1 \leq t \leq p < \infty$, $0<b \leq s$, and $\delta$ sufficiently small so that $\delta^{-s}\geq K$, we have
\begin{equation}
\label{eq:multlin<A}
\begin{split}
\LHS\eqref{eq:multilin-dec-const-KM}
&=
L^{p}_{x\in \R^{\calD}} \avprod \norm{ f_{R_{i}} }_{\avL^{p}(B(x,K))}
\\ &\lesssim
L^{p}_{x\in \R^{\calD}} \avprod \norm{ f_{R_{i}} }_{\avL^{p}(B(x,\delta^{-s}))}
\\ &\leq
L^{p}_{x\in \R^{\calD}} \avprod \sum_{J \in \Part[R_{i}]{\delta^{b}}} \norm{ f_{J} }_{\avL^{p}(B(x,\delta^{-s}))}
\\ &\lesssim
\delta^{-bd-(s-b)n_{k}(1/t-1/p)} L^{p}_{x\in \R^{\calD}} \avprod \avell^{t}_{J \in \Part[R_{i}]{\delta^{b}}} \norm{ f_{J} }_{\avL^{t}(w_{B(x,\delta^{-s})})}
\\ &\leq
\delta^{-Cb} \tA_{*}(b).
\end{split}
\end{equation}
Here we have used the reverse H\"older inequality (Corollary~\ref{cor:rev-holder}) to estimate the $\avL^{p}$ norm by the $\avL^{t}$ norm at the cost of increasing the weight.

By Theorem~\ref{thm:multilinear-to-linear} and the assumption that $\eta > \tGamma'$, the estimate \eqref{eq:multlin<A} implies that
\begin{equation}
\label{eq:a*:BG}
\eta \leq Cb + a_{*}(b).
\end{equation}

\subsection{Ball inflation}
\label{sec:ball-inflation}
\begin{lemma}[Ball inflation]
\label{lem:ball-inflation}
Let $1\le l < k$, $1 \leq t < \infty$, and $0 < p \leq t \frac{n_{k}}{n_{l}}$.
Let $\rho \leq K^{-1}$ and let $B \subset \R^{\calD}$ be a ball of radius $\rho^{-(l+1)}$.
Then we have
\begin{equation}
\label{eq:ball-inflation}
\avL^{p}_{x \in B} \avprod \ell^{t}_{J \in \Part[R_i]{\rho}} \norm{f_{J}}_{\avL^{t}(w_{B(x,\rho^{-l})})}\\
\lesssim \nu^{-n_{l}/(t n_{k})}
\avprod \ell^{t}_{J \in \Part[R_i]{\rho}} \norm{f_{J}}_{\avL^{t}(w_B)}
\end{equation}
\end{lemma}
Lemma~\ref{lem:ball-inflation} extends \cite[Theorem 6.6]{MR3548534}, \cite[Lemma 6.5]{MR3709122}, and \cite[Lemma 4.4]{arxiv:1804.02488} with an almost identical proof.
The additional flexibility in the choice of exponents allows us to also handle smaller values of $p$ by the same argument as large values.
Some of the cited results feature $\ell^{q}L^{t}$ ball inflation for $q<t$.
These results can be recovered from the $\ell^{t}L^{t}$ ball inflation inequality in Lemma~\ref{lem:ball-inflation}, see Corollary~\ref{cor:ball-inflation:ell-r}.

The proof of Lemma~\ref{lem:ball-inflation} relies on an extension of the Brascamp--Lieb inequality \eqref{eq:BL-const-def}, in which each subspace $V_{j}$ is replaced by a family of subspaces.
This extension goes back to \cite[Theorem 1.15]{MR2275834}, and a version that is sufficiently general for our needs was first obtained in \cite[Theorem 1.2]{MR3783217}.
It is nevertheless more convenient to use an endpoint result, Theorem~\ref{thm:KBL} below, which was first obtained using methods from \cite{MR2746348} in \cite[Theorem 8.1]{MR3738255}.
We refer to \cite{arxiv:1807.09604} for a newer exposition of its proof that incorporates the insights of \cite{MR3019726}.

\begin{theorem}[Kakeya--Brascamp--Lieb]
\label{thm:KBL}
Let $M$ be a positive integer and $1\leq l < k$.
For $1\le j\le M$, let $\calV_{j}$ be a family of linear subspaces of $\R^{n_{k}}$ of dimension $n_l$.
For a linear subspace $V \subseteq \R^{n_{k}}$, let $\pi_{V}: \R^{n_{k}}\to V$ denote the orthogonal projection onto $V$.
Assume that
\[
A := \sup_{V_{1} \in \calV_{1}, \dotsc, V_{M} \in \calV_{M}}\BL((V_{j})_{j=1}^{M}) < \infty.
\]
Then, for any non-negative integrable functions $f_{j,V_{j}} : V_{j} \to \R$ with $V_{j}\in\calV_{j}$, we have
\begin{equation}
\label{eq:KBL}
\int_{\R^{n_{k}}} \prod_{j=1}^M \Bigl( \sum_{V_{j} \in \calV_{j}} f_{j,V_{j}} (\pi_{V_{j}} (x)) \Bigr)^{\frac{n_{k}}{n_{l} M}} \dif x
\lesssim
A \prod_{j=1}^M \Bigl( \sum_{V_{j} \in \calV_{j}} \int_{V_{j}} f_{j,V_{j}} (x) \dif x \Bigr)^{\frac{n_{k}}{n_{l} M}}.
\end{equation}
\end{theorem}

\begin{proof}[Proof of Lemma~\ref{lem:ball-inflation}]
Since the left-hand side of \eqref{eq:ball-inflation} is monotonically increasing in $p$, it suffices to consider $p = t \frac{n_{k}}{n_{l}}$.
In this case, $\frac{p}{t M} = \frac{n_{k}}{n_{l} M}$.
Our goal is to control the expression
\begin{equation}
\label{ne4}
\fint_{x\in B} \prod_{i=1}^{M} \Bigl( \sum_{J_i} \norm{ f_{J_i}  }_{\avL^{t}(w_{B(x,\rho^{-l})})}^t \Bigr)^{\frac{n_{k}}{n_{l} M}},
\end{equation}
where $\fint_{B} := \abs{B}^{-1}\int_{B}$ denotes the average integral, and where $J_{i}$ ranges over $\Part[R_{i}]{\rho}$.
For each cube $J \in \Part[R_{i}]{\rho}$ with center $t_{J}$, we cover $B$ with a family $\calT_J$ of disjoint tiles $T_J$, which are rectangular boxes with $n_{l}$ short sides of length $\rho^{-l}$ pointing in the directions of $V_l(t_J)$ and $n_{k}-n_{l}$ longer sides of length $\rho^{-l-1}$ pointing in the directions $V_l(t_J)^{\perp}$.
Moreover, we can assume that these tiles are contained in the ball $3B$.
We let $T_J(x)$ be the tile containing $x$, and for $x\in \cup_{ T_J\in\calT_J} T_J$ we define
\[
F_J(x) := \sup_{y \in T_J(x)} \norm{ f_J  }_{\avL^{t}(w_{B(y, \rho^{-l})})}.
\]
Then
\[
\fint_{x\in B} \prod_{i=1}^{M} \Bigl( \sum_{J_i} \norm{ f_{J_i}  }_{\avL^{t}(w_{B(x,\rho^{-l})})}^t \Bigr)^{\frac{n_{k}}{n_{l} M}}
\leq
\fint_{B} \prod_{i=1}^{M} \bigl( \sum_{J_i} F_{J_i}^t \bigr)^{\frac{n_{k}}{n_{l} M}}.
\]
Since the function $F_J$ is constant on each tile $T_J\in\calT_J$, it can be treated as if it was constant in the direction $V_{l}(t_{J})^{\perp}$ on $B$.
By Theorem~\ref{thm:KBL}, we obtain
\begin{equation}
\label{eq:5}
\fint_{B} \prod_{i=1}^{M} \bigl( \sum_{J_i} F_{J_i}^t \bigr)^{\frac{n_{k}}{n_{l} M}}
\lesssim
\nu^{-1} \prod_{i=1}^{M} \Bigl( \sum_{J_i} \fint_{2B} F_{J_i}^t \Bigr)^{\frac{n_{k}}{n_{l} M}}.
\end{equation}
The most convenient way to justify the inequality \eqref{eq:5} is to observe that it is scale-invariant, so we may think of $B$ as a ball of radius $1$ and of $F_{J}$ as functions that are constant in the direction $V_{l}(t_{J})^{\perp}$ on $B$.
Then the average integrals in \eqref{eq:5} can be replaced by non-normalized integrals.
The only difference from \eqref{eq:KBL} is then that the integrals on the right-hand side of \eqref{eq:5} are taken over a subset of $\R^{n_{k}}$ rather than $V_{l}(t_{J})$, but this does not make a difference, since the integration domain has unit size.

It remains to check that for each $J=J_i$ we have
\begin{equation}
\label{FJbound}
\norm{ F_J }_{\avL^{t}(2B)}
\lesssim
\norm{ f_J  }_{\avL^{t}(w_B)}.
\end{equation}
Once this is established, it follows that \eqref{ne4} is dominated by
\[
\nu^{-1} \Bigl[ \avprod \bigl( \sum_{J_i} \norm{f_{J_i}}_{\avL^{t}(w_B)}^t \bigr)^{1/t} \Bigr]^{p},
\]
as desired.

In order to prove \eqref{FJbound}, fix a Schwartz function $\psi$ on $\R^{\calD}$ such that $\widehat{\psi} \equiv 1$ on $[-2,2]^{\calD}$.
Recall \eqref{eq:affine-scaling:space} and let
\[
(L_{J,\mathrm{scale},l}(x))_{\bfj} := \rho^{\min(\abs{\bfj},l+1)} x_{\bfj}.
\]
Define $L^{1}$ normalized bump functions $\psi_{J}$ by
\[
\widehat{\psi_{J}}(\xi) := \widehat{\psi}(L_{J,\mathrm{scale},l}^{-*}L_{J,\mathrm{shear}}^{-*}(\xi - ((c_{J})^{\bfi})_{\bfi})).
\]
Then $\widehat{\psi_{J}} \equiv 1$ on $\Uncert(J)$, so that $f_{J} = f_{J} * \psi_{J}$.
Moreover, $\widehat{\psi_{J}}$ has moral Fourier support of size $\gtrsim \rho^{l}$ in the directions of $V_{l}(t_{J})$ and of size $\approx \rho^{l+1}$ in the orthogonal directions.
Fix $x=(x_\gamma)_{\gamma\in\Dset}$ and $y\in T_J(x)$.
Then
\begin{align*}
\norm{ f_{J}  }_{L^t(w_{B(y, \rho^{-l})})}^t
&=
\int \abs{f_{J} * \psi_{J}}^t(u) w_{B(y, \rho^{-l})}(u) \dif u
\\ &\leq
\norm{\psi_{J}}_{L^{1}}^{t-1}
\int (\abs{f_{J}}^{t} * \abs{\psi_{J}})(u) w_{B(y, \rho^{-l})}(u) \dif u
\\ &\lesssim
\int \abs{f_{J}}^{t}(u) (\abs{\psi_{J}}*w_{B(y, \rho^{-l})})(u) \dif u.
\end{align*}
Now, $\abs{\psi_{J}}*w_{B(y, \rho^{-l})} \lesssim \tilde{w}_{J}*w_{B(x, \rho^{-l})}$, where $\tilde{w}_{J}$ is an $L^{1}$ normalized cutoff function centered at $0$ adapted to the dimensions of $T_{J}(x)$.
Taking a supremum over $y$, we obtain
\[
F_{J}(x)^t
\lesssim
\int \abs{f_{J}}^{t}(u) (\tilde{w}_{J}*w_{B(x, \rho^{-l})})(u) \dif u.
\]
Integrating in $x$, we obtain \eqref{FJbound}.
\end{proof}

Note that for $1\leq l < k$ we have
\begin{equation}
\label{eq:alpha-ineq}
\frac{1}{t_{l}}
\geq
\frac{\alpha_l}{t_{l+1}}+\frac{1-\alpha_l}{q_{l}}.
\end{equation}
Indeed, in the case $t_{l} = 2$ this is immediate, while in the case $t_{l} = p n_{l}/n_{k}$ we can apply \eqref{eq:alpha}.

For $1 \leq l < k$, by Lemma~\ref{lem:ball-inflation} with $\rho=\delta^{b}$ and H\"older's inequality together with \eqref{eq:alpha-ineq} in place of the usual scaling condition (this inequality suffices because we are dealing with norms on normalized measure spaces), we obtain
\begin{equation}
\begin{split}
\label{eq:est1}
\tA_{t(l)}(b)
& =
L^{p}_{x\in\R^{\Dset}} \avprod \avell^{t_{l}}_{J \in \Part[R_i]{\delta^{b}}} \norm{f_{J} }_{\avL^{t_{l}}(w_{B(x,\delta^{-lb})})}
\\ & =
L^{p}_{x'\in\R^{\Dset}} \avL^{p}_{x\in B(x',\delta^{-(l+1)b})} \avprod \avell^{t_{l}}_{J \in \Part[R_i]{\delta^{b}}} \norm{f_{J} }_{\avL^{t_{l}}(w_{B(x,\delta^{-lb})})}
\\ &\lesssim
L^{p}_{x'\in\R^{\Dset}} \avprod \avell^{t_{l}}_{J \in \Part[R_i]{\delta^{b}}} \norm{f_{J} }_{\avL^{t_{l}}(w_{B(x',\delta^{-(l+1)b})})}
\\ &=
\tA_{p,t_{l}}(b,(l+1)b)
\\ & \leq
\tA_{p, t_{l+1}}(b,(l+1)b)^{\alpha_{l}}
\tA_{p, q_{l}}(b,(l+1)b)^{1-\alpha_{l}}
\\ & =
\tA_{t(l+1)}(b)^{\alpha_{l}} \tA_{q(l)}(b)^{1-\alpha_{l}}.
\end{split}
\end{equation}
This implies
\begin{equation}
\label{eq:a*:ball-inflation}
a_{t(l)}(b) \leq \alpha_{l} a_{t(l+1)}(b) + (1-\alpha_{l}) a_{q(l)}(b).
\end{equation}
This is the first family of inequalities that we will be iterating.

\subsection{Lower degree decoupling}
The second type of estimate does not use transversality.
Instead, we just apply a (lower-degree) linear decoupling inequality on each cube $J \in \Part[\delta^{b}]{R_{j}}$ individually.

Let $1 \leq l < k$.
Similarly to \eqref{eq:alpha-ineq}, we have
\begin{equation}
\label{eq:beta-ineq}
\frac{1}{q_{l}}
\geq
\frac{1-\beta_l}{q_{l-1}}+\frac{\beta_l}{t_{l}}.
\end{equation}
Using Corollary~\ref{cor:scaled+loc-decoupling} with $(\delta^{b},\delta^{(l+1)b/l})$ in place of $(\sigma,\delta)$, the inductive hypothesis that the decoupling inequality holds with exponent \eqref{eq:tGamma-recursion} with $\calD$ replaced by $\calD_{l}$, and H\"older's inequality with \eqref{eq:beta-ineq}, we obtain
\begin{equation}
\begin{split}
\label{eq:est2}
\tA_{q(l)}(b)
&=
L^{p}_{x\in\R^{\Dset}} \avprod \avell^{q_{l}}_{J \in \Part[R_i]{\delta^{b}}} \norm{f_{J} }_{\avL^{q_{l}}(w_{B(x,\delta^{-(l+1)b})})}
\\ &\lesssim_{\epsilon}
\delta^{-b (\tGamma_{\calD_{l}}(q_{l})+\epsilon)/l}
L^{p}_{x\in\R^{\Dset}} \avprod \avell^{q_{l}}_{J \in \Part[R_i]{\delta^{(l+1)b/l}}} \norm{f_{J} }_{\avL^{q_{l}}(w_{B(x,\delta^{-(l+1)b})})}
\\ &=
\delta^{-b (\tGamma_{\calD_{l}}(q_{l})+\epsilon)/l}
\tA_{p,q_{l}}(\frac{l+1}{l} b,(l+1)b)\\
& \leq
\delta^{-b (\tGamma_{\calD_{l}}(q_{l})+\epsilon)/l}
\tA_{p,t_{l}}(\frac{l+1}{l} b,(l+1)b)^{\beta_{l}}
\tA_{p,q_{l-1}}(\frac{l+1}{l} b,(l+1)b)^{1-\beta_{l}}\\
& =
\delta^{-b (\tGamma_{\calD_{l}}(q_{l})+\epsilon)/l}
\tA_{t(l)}(\frac{(l+1)b}{l})^{\beta_{l}} \tA_{q(l-1)}(\frac{(l+1)b}{l})^{1-\beta_{l}}.
\end{split}
\end{equation}
This implies
\begin{equation}
\label{eq:a*:lower-deg-dec}
a_{q(l)}(b) \leq
\frac{b}{l}\tGamma_{\calD_{l}}(q_{l})
+ \beta_{l} a_{t(l)}((l+1)b/l)
+ (1-\beta_{l}) a_{q(l-1)}((l+1)b/l)
\end{equation}
for $0<b<l/(l+1)$.
This is the second family of inequalities that we will be iterating.

\subsection{Exiting the iterative procedure: linear decoupling}
We can use H\"older's inequality to eliminate all multilinearity and use the rescaled linear decoupling estimate, Lemma~\ref{lem:scaled-decoupling}.
For $1 \leq t \leq p$ and $1 \leq l \leq k$, this gives the bound
\begin{equation}
\label{eq:A<prod}
\begin{split}
\tA_{p,t}(b,s)
&=
L^{p}_{x} \avprod \avell^{t}_{J \in \Part[R_{i}]{\delta^{b}}} \norm{ f_{J} }_{\avL^{t}(w_{B(x,\delta^{-lb})})}
\\ &\leq
\avprod \avell^{t}_{J \in \Part[R_{i}]{\delta^{b}}} L^{p}_{x} \norm{ f_{J} }_{\avL^{p}(w_{B(x,\delta^{-lb})})}
\\ &=
\avprod \avell^{p}_{J \in \Part[R_{i}]{\delta^{b}}} \norm{ f_{J} }_{L^{p}(\R^{\calD})}
\\ &\leq
\avDec(\calD,p,\delta^{1-b})
\avprod \avell^{p}_{J \in \Part[R_{i}]{\delta}} \norm{ f_{J} }_{p}
\\ &\lesssim_{\epsilon}
\delta^{-\eta(1-b)-\epsilon}
\avprod \avell^{p}_{J \in \Part[R_{i}]{\delta}} \norm{ f_{J} }_{p}.
\end{split}
\end{equation}
The factor $\delta^{\eta b}$ is a gain over the trivial estimate for \eqref{eq:multilin-dec-const-KM} that arises if one starts with H\"older's inequality on the left-hand side.
This shows
\begin{equation}
\label{eq:a*:linear-dec}
a_{*}(b) \leq \eta(1-b).
\end{equation}

\subsection{Reduction to a finite system of inequalities}
So far we have obtained several families of inequalities for the quantities \eqref{eq:a_*}, the most important of which are summarized in Figure~\ref{fig:graph}.

\begin{figure}
\centering
\begin{tikzpicture}[bend angle=20,auto,swap]
\matrix[row sep=10mm,column sep=16mm] {
\node (q1) {$q(1)$}; & \node (q2) {$q(2)$}; & \node (q3) {$q(3)$}; & \node (qdots) {\dots}; & \node (qk-1) {$q(k-1)$}; & \\
\node (p1) {$t(1)$}; & \node (p2) {$t(2)$}; & \node (p3) {$t(3)$}; & \node (pdots) {\dots}; & \node (pk-1) {$t(k-1)$}; & \node (pk) {$t(k)$}; \\
};
\draw [->] (q1) edge [bend right] node {$\frac{2}{1}$} (p1)
(q2) edge [bend right] node {$\frac{3\beta_{2}}{2}$} (p2)
(q3) edge [bend right] node {$\frac{4\beta_{3}}{3}$} (p3)
(qk-1) edge [bend right] node {$\frac{k\beta_{k-1}}{k-1}$} (pk-1)
(q2) edge node {$\frac{3 (1-\beta_{2})}{2}$} (q1)
(q3) edge node {$\frac{4 (1-\beta_{3})}{3}$} (q2)
(qdots) edge (q3)
(qk-1) edge (qdots)
(p1) edge node {$\alpha_{1}$} (p2)
(p2) edge node {$\alpha_{2}$} (p3)
(p3) edge (pdots)
(pdots) edge (pk-1)
(pk-1) edge node {$\alpha_{k-1}$} (pk)
(p1) edge [bend right] node {$1-\alpha_{1}$} (q1)
(p2) edge [bend right] node {$1-\alpha_{2}$} (q2)
(p3) edge [bend right] node {$1-\alpha_{3}$} (q3)
(pk-1) edge [bend right] node {$1-\alpha_{k-1}$} (qk-1);
\end{tikzpicture}
\caption{Relations between quantities \eqref{eq:a_*} given by estimates \eqref{eq:a*:ball-inflation} and \eqref{eq:a*:linear-dec}.
  The weights on the edges are the products of the factors in front of $a_{*}$ and $b$ in those estimates.
  The tree of estimates in \cite{MR3548534} is the universal covering of this graph starting in the upper left corner.}
\label{fig:graph}
\end{figure}
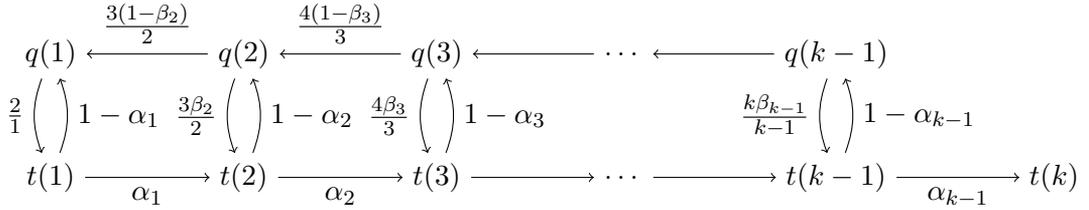

We eliminate the dependence on $b$ by considering the quanitities
\begin{equation}
\label{eq:tilde-a}
\tilde{a}_{*} := \liminf_{b \to 0} \frac{\eta - a_{*}(b)}{b},
\end{equation}
which were introduced in a blog post by Terence Tao.
The linear decoupling \eqref{eq:a*:linear-dec}, linear to multilinear reduction \eqref{eq:a*:BG}, ball inflation \eqref{eq:a*:ball-inflation}, and lower degree decoupling \eqref{eq:a*:lower-deg-dec} inequalities imply
\begin{align}
\label{eq:tilde-a*:linear-dec}
\tilde{a}_{*} &\geq \eta,
\\ \label{eq:tilde-a*:BG}
\tilde{a}_{*} &\leq C,
\\ \label{eq:tilde-a*:ball-inflation}
\tilde{a}_{t(l)} &\geq
\alpha_{l} \tilde{a}_{t(l+1)}
+ (1-\alpha_{l}) \tilde{a}_{q(l)},
& 1 \leq l < k,
\\ \label{eq:tilde-a*:lower-deg-dec}
\tilde{a}_{q(l)} &\geq
- \frac{1}{l}\tGamma_{\calD_{l}}(q_{l})
+ \frac{l+1}{l} \bigl( \beta_{l} \tilde{a}_{t(l)}
+ (1-\beta_{l}) \tilde{a}_{q(l-1)} \bigr),
& 1 \leq l < k.
\end{align}
Notice that $\tGamma_{\calD_{l}}(q_{l}) \leq \tGamma'' < \eta$ for $1 \leq l < k$.
In the case $l=k-1$ this holds by definition \eqref{eq:tGamma-recursion}, and other cases follow by an inductive argument.
Hence the last estimate \eqref{eq:tilde-a*:lower-deg-dec} implies
\begin{equation}
\label{eq:tilde-a*:lower-deg-dec'}
\tilde{a}_{q(l)} \geq
- \tGamma''
+ \frac{l+1}{l} \bigl( \beta_{l} \tilde{a}_{t(l)}
+ (1-\beta_{l}) \tilde{a}_{q(l-1)} \bigr).
\end{equation}
We claim that the inequalities \eqref{eq:tilde-a*:linear-dec}, \eqref{eq:tilde-a*:BG}, \eqref{eq:tilde-a*:ball-inflation}, and \eqref{eq:tilde-a*:lower-deg-dec'} imply $\eta \leq \tGamma''$.

The inequalities \eqref{eq:tilde-a*:ball-inflation} and \eqref{eq:tilde-a*:lower-deg-dec'} form a linear system of the form
\begin{equation}
\label{eq:linear-ineq-system}
\tilde{\bfa} \geq \calM \tilde{\bfa} + v_{1} + v_{2}\tilde{a}_{t(k)},
\end{equation}
where $\tilde{\bfa}=(\tilde{a}_{q(1)},\dotsc,\tilde{a}_{q(k-1)},\tilde{a}_{t(1)},\dotsc,\tilde{a}_{t(k-1)})$ and $\calM$ is a $2(k-1) \times 2(k-1)$-matrix given by
\begin{align*}
(\calM v)_{q(l)} &:= (1-\alpha_{l})v_{t(l)} + \frac{l+2}{l+1} (1-\beta_{l+1})v_{q(l+1)},
&& 1 \leq l \leq k-2\\
(\calM v)_{q(l)} &:= (1-\alpha_{l})v_{t(l)},
&& l = k-1\\
(\calM v)_{t(l)} &:= \frac{l+1}{l} \beta_{l} v_{q(l)} + \alpha_{l-1} v_{t(l-1)},
&& 2 \leq l \leq k-1\\
(\calM v)_{t(l)} &:= \frac{l+1}{l} \beta_{l} v_{q(l)},
&& l = 1.
\end{align*}
Recall that $\alpha_{l}$ and $\beta_{l}$ were defined in \eqref{eq:alpha} and \eqref{eq:beta}, respectively.
Here we also use the conventions $\beta_{1}=1$, $\alpha_{0}=0$, $\beta_{k}=1$.

It is an observation going back to \cite{MR3479572} in the context of Vinogradov's mean value theorem, and made more explicit in \cite{arxiv:1512.03272}, that the matrix $\calM$ should have a positive eigenvalue that is $\geq 1$ in order to extract useful information from \eqref{eq:linear-ineq-system}.
In our case, the matrix $\calM$ is irreducible (with period $2$, as can be seen upon removing the vertex $t(k)$ in Figure~\ref{fig:graph}) and has non-negative entries.
The Perron--Frobenius theorem tells that $\calM$ has a unique positive right eigenvector, which dominates its asymptotics in the sense that the corresponding eigenvalue equals the spectral radius of $\calM$.
Since we can compute this so-called Perron--Frobenius eigenvector explicitly, we will not actually have to apply Perron--Frobenius theory, but it motivated our approach.
\begin{theorem}
\label{thm:PF-eigenvector}
Let $v \in \R^{2(k-1)}$ be the vector given by
\begin{equation}
\label{eq:PF-EV}
v_{q(l)} := \frac{\calK_{l}}{l+1},
\quad
v_{t(l)} := n_{l}
\end{equation}
for $1 \leq l < k$.
Then $\calM v=v$.
\end{theorem}
Since the \emph{right} eigenvector $v$ in \eqref{eq:PF-EV} is positive, it is in fact the right Perron--Frobenius eigenvector of $\calM$.
The \emph{left} Perron--Frobenius eigenvector of $\calM$ is essentially given by \cite[Lemma 8.2]{arxiv:1804.02488}, but it does not seem to be quite as useful as the right one.

\begin{proof}
We have to verify $v_{q(l)} = (\calM v)_{q(l)}$ for $1\leq l < k$.
Since $\beta_{k}=1$, this is equivalent to
\[
v_{q(l)} = (1-\alpha_{l})v_{t(l)} + \frac{l+2}{l+1} (1-\beta_{l+1})v_{q(l+1)}
\]
for $1 \leq l < k$.
We also have to verify $v_{t(l)} = (\calM v)_{t(l)}$ for $1\leq l < k$.
For $l=1$ this is easy using that $\beta_{1}=1$.
For $1<l<k$ this can be written as
\[
v_{t(l)} = \frac{l+1}{l} \beta_{l} v_{q(l)} + \alpha_{l-1} v_{t(l-1)}.
\]
Substituting the definitions \eqref{eq:PF-EV}, identities that we have to verify can be equivalently written as
\begin{align*}
\calK_{l} - (1-\beta_{l+1})\calK_{l+1}
&= (l+1) (1-\alpha_{l}) n_{l},
& 1 \leq l < k,\\
n_{l+1} - n_{l}\alpha_{l}
&= \frac{1}{l+1} \beta_{l+1} \calK_{l+1},
& 0 \leq l < k-1.
\end{align*}
Using the definition of $\beta_{l+1}$ and $\alpha_{l}$ on the respective left-hand sides, we see that this is equivalent to
\begin{align*}
\beta_{l+1} \frac{n_{k}}{n_{l+1}} \frac{\calK_{l} \calK_{l+1}}{\calK_{k}}
&= (l+1) (1-\alpha_{l}) n_{l},\\
(1-\alpha_{l}) \frac{\calK_{k}}{\calK_{l}} \frac{n_{l} n_{l+1}}{n_{k}}
&= \frac{1}{l+1} \beta_{l+1} \calK_{l+1}.
\end{align*}
Both these identities are equivalent to
\[
\beta_{l+1} n_{k} \calK_{l} \calK_{l+1}
= (l+1) (1-\alpha_{l}) n_{l} n_{l+1} \calK_{k}.
\]
Using
\[
(1-\alpha_{l}) \Bigl( \frac{\calK_{k}}{\calK_{l}} - \frac{n_{k}}{n_{l+1}} \Bigr) = \frac{n_{k}}{n_{l}} - \frac{n_{k}}{n_{l+1}}
\]
and
\[
\beta_{l+1} \Bigl( \frac{n_{k}}{n_{l+1}} - \frac{\calK_{k}}{\calK_{l}} \Bigr) = \frac{\calK_{k}}{\calK_{l+1}} - \frac{\calK_{k}}{\calK_{l}},
\]
we see that our claim becomes equivalent to
\[
\calK_{l+1}-\calK_{l}
= (l+1) (n_{l+1}-n_{l}).
\]
This is a consequence of \eqref{eq:homdim}.
\end{proof}

We will also need the following identity to handle the contribution of the constant terms $\tGamma''$ in \eqref{eq:tilde-a*:lower-deg-dec'}.
\begin{lemma}
\label{lem:total-loss}
Let $d\geq 1$, $k\geq 2$, and let $v$ be given by Theorem~\ref{thm:PF-eigenvector}.
Then
\begin{equation}
\label{eq:total-loss}
\sum_{l=1}^{k-1} \frac{v_{q(l)}}{l}
=
v_{t(k-1)} \alpha_{k-1}.
\end{equation}
\end{lemma}
\begin{proof}
Left-hand side of \eqref{eq:total-loss} equals
\begin{equation*}
\begin{split}
\sum_{l=1}^{k-1} \frac{1}{l(l+1)} \calK_{l}
& =
\sum_{j=1}^{k-1} j (n_{j}-n_{j-1}) \sum_{l=j}^{k-1} \bigl( \frac{1}{l} - \frac{1}{l+1} \bigr)\\
& =
\sum_{j=1}^{k-1} j (n_{j}-n_{j-1}) \bigl( \frac{1}{j} - \frac{1}{k} \bigr)
=
n_{k-1} - \frac{1}{k} \calK_{k-1}.
\end{split}
\end{equation*}
Right-hand side of \eqref{eq:total-loss} equals
\begin{equation*}
\begin{split}
n_{k-1}\alpha_{k-1}
& =
n_{k-1} \frac{\frac{n_{k}}{n_{k-1}} - \frac{\calK_{k}}{\calK_{k-1}}}{1 - \frac{\calK_{k}}{\calK_{k-1}}}
=
n_{k-1} + n_{k-1} \frac{\frac{n_{k}}{n_{k-1}} - 1}{1 - \frac{\calK_{k}}{\calK_{k-1}}}\\
& =
n_{k-1} + \frac{n_{k} - n_{k-1}}{\calK_{k-1} - \calK_{k}} \calK_{k-1}
=
n_{k-1} - \frac{1}{k} \calK_{k-1}
\end{split}
\end{equation*}
as well.
\end{proof}

\subsection{Conclusion of the proof of Theorem~\ref{thm:main}}
We consider the inequality
\[
\sum_{l=1}^{k-1} v_{t(l)} \cdot \eqref{eq:tilde-a*:ball-inflation}
+
\sum_{l=1}^{k-1} v_{q(l)} \cdot \eqref{eq:tilde-a*:lower-deg-dec'},
\]
where $v$ is given by Theorem~\ref{thm:PF-eigenvector}.
Since $v$ is a fixed vector of the matrix $\calM$, the summands $\tilde{a}_{t(l)}$ and $\tilde{a}_{q(l)}$ with $l<k$ cancel (here we use \eqref{eq:tilde-a*:linear-dec} and \eqref{eq:tilde-a*:BG} to ensure that we only cancel out finite quantities).
This gives
\[
0 \geq - \sum_{l=1}^{k-1} \frac{v_{q(l)}}{l}\tGamma'' + v_{t(k-1)}\alpha_{k-1} \tilde{a}_{t(k)}.
\]
By Lemma~\ref{lem:total-loss} this implies
\[
\tilde{a}_{t(k)} \leq \tGamma'',
\]
and using \eqref{eq:tilde-a*:linear-dec} we obtain $\eta \leq \tGamma''$.
As explained at the beginning of Section~\ref{sec:induction-on-scales}, this finishes the proof of Theorem~\ref{thm:main}.

\section{Transversality}
\label{sec:transverse}
In this section we prove the following more precise version of Theorem~\ref{thm:rank}.

\begin{theorem}[{cf.\ \cite[Theorem 10.8]{arxiv:1804.02488}}]
\label{thm:BLcond}
Let $d\geq 1$, $\bfk\in\N^{d}$, and $1 \leq l < k \leq k_{1} + \dotsb + k_{d}$ be positive integers.
Let $V = \Span \Set{v_{1},\dotsc,v_{H}} \subseteq \R^{\calD_{k}}$ be a linear subspace.
Then the rank of the matrix~\eqref{eq:rank-matrix} over the field of rational functions in $d$ variables satisfies
\begin{equation}
\label{eq:BLcond}
\rank_{\R(x_1, \dotsc, x_d)} \calM_V^{(l)}(x_{1},\dotsc,x_{d})
\geq
\frac{\abs{\calD_{l}}}{\abs{\calD_{k}}} \dim V.
\end{equation}
Equality in \eqref{eq:BLcond} can only hold in the following cases:
\begin{enumerate}
\item\label{thm:BLcond:full/empty} $\dim V \in \Set{0, \abs{\calD_{k}}}$, or
\item\label{thm:BLcond:d=2} $d=2$, $1=k_{1} < k_{2}$, and $V$ is spanned by the unit vectors with coordinates $(0,1),\dotsc,(0,k)$, or
\item\label{thm:BLcond:d=2'} $d=2$, $1=k_{2} < k_{1}$, and $V$ is spanned by the unit vectors with coordinates $(1,0),\dotsc,(k,0)$.
\end{enumerate}
\end{theorem}

\begin{proof}[Proof of Theorem~\ref{thm:rank} assuming Theorem~\ref{thm:BLcond}]
By Theorem~\ref{thm:BLcond}, the matrix $\calM_V^{(l)}(\bft)$ has a minor of order $\geq \frac{n_{l}}{n_{k}} \dim V$ whose determinant is a non-vanishing polynomial.
If the order of this minor is $> \frac{n_{l}}{n_{k}} \dim V$, then we are in the case \ref{rank_inequality} of Theorem~\ref{thm:rank}.

Otherwise, we may assume that one of the equality conditions in Theorem~\ref{thm:BLcond} holds.
The equality case~\ref{thm:BLcond:full/empty} of Theorem~\ref{thm:BLcond} is excluded by the assumption $0 < \dim V < n_{k}$.
In the equality case~\ref{thm:BLcond:d=2} of Theorem~\ref{thm:BLcond} notice that \eqref{eq:BLcond} does not depend on the choice of the spanning set of $V$.
Hence, in addition to having $d=2$ and $1=k_{1}<k_{2}$, we may assume $H=k$ and $v_{h}$ is the $(0,h)$-th unit vector.
Then
\[
\calM_V^{(l)}(\bft) = ( \partial^{\bfj} t_{2}^{h} )_{\bfj \in \calD_{l},1\leq h \leq k}.
\]
This matrix contains the $l\times l$ submatrix
\[
( \partial_{2}^{j} t_{2}^{h} )_{1 \leq j, h \leq l},
\]
and the determinant of the latter matrix is not just a non-trivial polynomial, but a non-vanishing constant.
Since in this case $l = \frac{H \cdot 2l}{2k} = \frac{n_{l}}{n_{k}} \dim V$, we are in the case~\ref{rank_equality} of Theorem~\ref{thm:rank}.

The equality case~\ref{thm:BLcond:d=2'} of Theorem~\ref{thm:BLcond} is similar to the equality case~\ref{thm:BLcond:d=2}.
\end{proof}

\subsection{Reduction to a Vandermonde type matrix}
Expanding the definition \eqref{eq:rank-matrix}, we obtain
\begin{equation}
\label{eq:rank-matrix-expanded}
\calM_V^{(l)}(\bft)
=
\bigl( v_1, \dotsc, v_{H} \bigr)^T \times \bigl( \partial^{\bfj} \Phi(\bft) \bigr)_{\bfj \in \calD_{l}}
=
( \partial^{\bfj} \sum_{\bfi\in\calD_{k}} v_{h,\bfi} \bft^{\bfi} )_{\bfj \in \calD_{l},1\leq h \leq H}.
\end{equation}
In Lemma~\ref{lem:rank-est-by-leading-monomials} below, we will give a lower bound for the rank of the right-hand side of \eqref{eq:rank-matrix-expanded} in terms of the leading powers of the polynomials that appear there.
To this end we need a few order theoretic notions.
\begin{definition}
A \emph{monomial order} is a translation-invariant total order relation $\leq$ on $\N^{d}$.
Translation invariance means that for every $\bfa,\bfb,\bfc \in \N^{d}$ we have $\bfa \leq \bfb \implies \bfa + \bfc \leq \bfb + \bfc$.

The \emph{leading power} of a non-zero polynomial $f(\bfx)$ in $d$ variables with respect to a given monomial order $\leq$ is the largest $\bfa \in \N^{d}$ with respect to $\leq$ such that the coefficient of $\bfx^{\bfa}$ in $f$ does not vanish.
\end{definition}

An example of a monomial order is the \emph{lexicographic order} on $\N^{d}$, which is defined by $\bfi < \bfi'$ if and only for some $1 \leq q \leq d$ we have $i_1 = i_1', \dotsc, i_{q-1} = i_{q-1}'$, and $i_q < i_q'$.

\begin{definition}
\label{def:product-order}
We denote the \emph{product order} on $\N^{d}$ by $\preceq$.
More explicitly, for $\bfa = (a_{1},\dotsc,a_{d}),\bfb = (b_{1},\dotsc, b_{d}) \in\N^{d}$ we write $\bfa \preceq \bfb$ if and only if $a_{1}\leq b_{1}, \dotsc, a_{d}\leq b_{d}$.
\end{definition}

\begin{definition}
\label{def:up+down}
Let $(P,\preceq)$ be a partially ordered set.
A subset $\Dset \subseteq P$ is called a \emph{down-set} if for every $p\in P$ and $d\in\Dset$ with $p\preceq d$ we have $p\in\Dset$.
A subset $U \subseteq P$ is called an \emph{up-set} if for every $p\in P$ and $u\in U$ with $u \preceq p$ we have $p\in U$.
For a subset $B\subset P$ we write $\upset B := \Set{p\in P \given (\exists b\in B) b \preceq p}$; this is the smallest up-set containing $B$.
\end{definition}

\begin{lemma}
\label{lem:rank-est-by-leading-monomials}
Let $f_{1},\dotsc,f_{H} \in \R[x_{1},\dotsc,x_{d}]$ be polynomials in $d$ variables. Let $\leq$ be a monomial order and let $\bfa_{1},\dotsc,\bfa_{H}$ be the leading powers of $f_{1},\dotsc,f_{H}$ with respect to this monomial order.
Let $\sublevel \subset \N^{d} \setminus \Set{0}$ be a finite down-set with respect to $\preceq$.
Then
\begin{equation}
\label{eq:1}
\rank_{\R(x_1, \dotsc, x_d)}
( \partial^{\bfj} f_{h} )_{\bfj \in \sublevel,1\leq h \leq H}
\geq
\rank_{\R}
( \bfa_{h}^{\bfj} )_{\bfj \in \sublevel,1\leq h \leq H}.
\end{equation}
\end{lemma}
\begin{proof}
Multiplying the $\bfj$-th row by the non-zero field element $\bfx^{\bfj}$, we obtain
\[
\rank_{\R(x_1, \dotsc, x_d)}
( \partial^{\bfj} f_{h} )_{\bfj \in \sublevel,1\leq h \leq H}
=
\rank_{\R(x_1, \dotsc, x_d)}
( \bfx^{\bfj} \partial^{\bfj} f_{h} )_{\bfj \in \sublevel,1\leq h \leq H}.
\]
The latter rank is
\[
\geq
\rank_{\R(x_1, \dotsc, x_d)}
( \bfx^{\bfj} \partial^{\bfj} \bfx^{\bfa_{h}} )_{\bfj \in \sublevel, 1 \leq h \leq H}.
\]
Indeed, every minor of the latter matrix is a monomial, and if it does not vanish, then it is the leading monomial of the corresponding minor of the former matrix.

Multiplying the $h$-th row by $\bfx^{-\bfa_{h}}$, we obtain the matrix
\[
\begin{pmatrix}
(a_{h,1} \dotsm (a_{h,1}-j_{1}+1)) \dotsm (a_{h,d} \dotsm (a_{h,d}-j_{d}+1))
\end{pmatrix}_{\bfj \in \sublevel,1\leq h \leq H},
\]
all of whose entries are scalars.
The rank does not change under row operations, and by row operations this matrix can be brought into the form
\[
\begin{pmatrix}
\bfa_{h}^{\bfj}
\end{pmatrix}_{\bfj \in \sublevel, 1\leq h \leq H}.
\]
Here we used that $\sublevel$ is a down-set.
The rank of the latter matrix over the field of rational functions $\R(x_{1},\dotsc,x_{d})$ coincides with its rank over $\R$.
\end{proof}

\begin{remark}
\label{rem:Vandermonde}
In the case $d=1$, the matrix on the right-hand side of \eqref{eq:1} has full rank by the Vandermonde determinant formula, provided that the leading powers $\bfa_{1},\dotsc,\bfa_{H}$ are distinct.
\end{remark}
In the remaining part of Section~\ref{sec:transverse} we estimate the rank on the right-hand side of \eqref{eq:1} from below for general $d$ by refining the arguments in \cite{arxiv:1804.02488}.

\subsection{An abstract Schwartz--Zippel type lemma}

The following result extends and simplifies \cite[Lemmas 10.5 and 10.6]{arxiv:1804.02488}.
We obtain \cite[Lemma 10.5]{arxiv:1804.02488} as the special case $\Dset = \Set{a\in \N^{d} \given \abs{a} \leq k}$ and \cite[Lemma 10.6]{arxiv:1804.02488} as the special case $\Dset = \Set{0,\dotsc,k}^{d}$.
The latter case, with sets $R_{*}$ defined as in the proof of Lemma~\ref{lem:powers-rank-est}, recovers several, but not all, versions of the Schwartz--Zippel lemma in \cite{MR3788163}.

\begin{lemma}
\label{lem:SZ}
Let $d\geq 1$ be an integer and $\Dset \subset \N^{d}$ a finite down-set with respect to $\preceq$.
Let $A \subseteq \Dset$ and $B \subseteq \N^{d}$.
Suppose that for every $\bfb = (b_1, \ldots, b_d) \in B$ there is a family of inductively defined subsets $R_{*;*;\bfb} \subset \N$ with the following properties.
\begin{enumerate}
\item\label{lem:SZ:1}
For every $1\leq l\leq d$ and every
\[
n_d \in \N \setminus R_{d; \bfb}, n_{d-1} \in \N \setminus R_{d-1; n_{d}; \bfb}, \dotsc, n_{l+1} \in \N \setminus R_{l+1; n_{l+2}, \ldots, n_d; \bfb},
\]
we have $\abs{R_{l; n_{l+1}, \ldots, n_d; \bfb}} \leq b_l$.
\item\label{lem:SZ:2}
If for some $\bfa=(a_1, \ldots, a_d)\in A$ we have
\[
a_d \notin R_{d; \bfb}, a_{d-1} \notin R_{d-1; a_d; \bfb}, \ldots, a_{2} \notin R_{2; a_3, \ldots, a_d; \bfb},
\]
then $a_1 \in R_{1; a_2, a_3, \ldots, a_d; \bfb}$.
\end{enumerate}
Then
\begin{equation}\label{geomineqCk}
\abs{A} \leq \abs{\Dset\setminus \upset B}.
\end{equation}
\end{lemma}

\begin{remark}
\label{rem:schwartz-zippel:sharpness}
The estimate \eqref{geomineqCk} is sharp since it is possible to take $A = \Dset \setminus \upset B$.
In fact, if $(a_1, \dotsc, a_d) \in A$, then by definition for every $(b_1, \ldots, b_d) \in B$ we have $0 \leq a_i < b_i$ for some $1\leq i \leq d$.
Therefore the sets $R_{i;*;\bfb} := \Set{0, \dotsc, b_i - 1}$ satisfy the hypothesis~\ref{lem:SZ:2}.
Also, these sets clearly satisfy the hypothesis~\ref{lem:SZ:1}.
\end{remark}

\begin{proof}[Proof of Lemma~\ref{lem:SZ}]
We will prove \eqref{geomineqCk} by induction on $d$.

We verify the induction basis $d=1$.
For every $\bfb = (b_{1}) \in B$, by hypothesis~\ref{lem:SZ:1} we have $\abs{R_{1;;\bfb}} \leq b_{1}$, and by hypothesis~\ref{lem:SZ:2} for every $\bfa = (a_{1}) \in A$ we have $a_{1} \in R_{1;;\bfb}$, so that $\abs{A} \leq b_{1}$.
Hence
\[
\abs{A}
\leq
\min(\abs{\calD},\min_{\bfb \in B} b_{1})
=
\abs{\calD \cap \Set{0,\dotsc,\min_{\bfb \in B} b_{1}-1}}
=
\abs{\Dset \setminus \upset B},
\]
which shows \eqref{geomineqCk}.
This finishes the proof in the case $d=1$.

From now on we assume that $d>1$ and that the lemma is already known with dimension $d$ replaced by $d-1$.
The conclusion \eqref{geomineqCk} is equivalent to the statement
\begin{equation}\label{geomineq2}
\abs{\Dset \setminus A} \geq \abs{\Dset \cap \upset B}.
\end{equation}

For $j \in \N$ and a subset $\tilde{A} \subset \N^{d}$, let
\begin{equation}
\label{eq:slice-map}
S_j \tilde{A} := \Set{\bfa'\in \N^{d-1} \given (\bfa',j) \in \tilde{A}}
\end{equation}
denote the \emph{$j$-th slice} of $\tilde{A}$.
For any subset $\tilde{B} \subset \N^{d}$ define the projection
\[
\bfP \tilde{B} :=
\Set{\bfb' \in \N^{d-1} \given \exists b_d \text{ s.t. } (\bfb', b_d) \in \tilde{B}}.
\]

Fix $j\in\N$.
The slice $S_{j} \Dset$ is a finite down-set in $(\N^{d-1},\preceq)$.
Let $B_j := \Set{\bfb \in B: j \notin R_{d; \bfb}}$.
Then the sets $S_{j}A \subseteq S_{j}\calD$ and $\bfP B_j \subseteq \N^{d-1}$ satisfy the hypothesis of Lemma~\ref{lem:SZ} in dimension $d-1$.
Indeed, for each $\bfb'\in \bfP B_{j}$ fix a $b_{d}=b_{b}(\bfb')$ such that $\bfb := (\bfb',b_{d}(\bfb')) \in B_{j}$ and let
\[
R'_{l;n_{l+1},\dotsc,n_{d-1};\bfb'} := R_{l;n_{l+1},\dotsc,n_{d-1},j;\bfb},
\quad
1 \leq l \leq d-1.
\]
Then hypothesis~\ref{lem:SZ:1} for the sets $R'_{*;*;\bfb'}$ follows directly from hypothesis~\ref{lem:SZ:1} for the sets $R_{*;*;\bfb}$.
Let now $\bfa' = (a_{1},\dotsc,a_{d-1}) \in S_{j}A$, so that $\bfa = (a_{1},\dotsc,a_{d}) := (\bfa',j) \in A$.
If the conditions
\[
a_{d-1} \not\in R'_{d-1;\bfb'},\dotsc,a_{2}\not\in R'_{2;a_{3},\dotsc,a_{d-1};\bfb'}
\]
hold, then we also have
\begin{align*}
j=a_{d} &\not\in R_{d;\bfb},\\
a_{d-1} &\not\in R_{d-1;a_{3},\dotsc,a_{d};\bfb}=R'_{d-1;a_{3},\dotsc,a_{d-1};\bfb'},\\
&\vdots\\
a_{2} &\not\in R_{2;a_{3},\dotsc,a_{d};\bfb}=R'_{2;a_{3},\dotsc,a_{d-1};\bfb'},
\end{align*}
and by hypothesis~\ref{lem:SZ:2} we obtain $a_{1} \in R_{1;a_{2},\dotsc,a_{d};\bfb} = R'_{1;a_{2},\dotsc,a_{d-1};\bfb'}$.
This shows that hypothesis~\ref{lem:SZ:2} holds for the sets $R'_{*;*;\bfb'}$.

By the inductive hypothesis we obtain
\begin{align}
\notag
\abs{\Dset \setminus A}
&=
\sum_{j\in\N} \abs{S_{j}\Dset \setminus S_{j}A}\\
\label{eq:D-A-slice-bound}
&\geq
\sum_{j\in\N} \abs{S_{j}\Dset \cap \upset \bfP B_j}.
\end{align}
It remains to show that the last sum is bounded below by $\abs{\calD \cap \upset B}$.
Indeed,
\begin{align*}
\eqref{eq:D-A-slice-bound}
&=
\sum_{\bfc' \in \N^{d-1}} \abs{\Set{j\in\N \given \bfc' \in S_{j}\Dset \cap \upset \bfP B_j}}\\
&=
\sum_{\bfc' \in \N^{d-1}} \abs{\Set{j\in\N \given (\bfc',j) \in \Dset, \exists \bfb = (\bfb',b_{d}) \in B_j : \bfb' \preceq \bfc'}}\\
&=
\sum_{\bfc' \in \N^{d-1}} \abs{\Set{j\in\N \given (\bfc',j) \in \Dset, \exists \bfb = (\bfb',b_{d}) \in B : \bfb' \preceq \bfc', j \notin R_{d;\bfb}}}\\
&\geq
\sum_{\bfc' \in \N^{d-1}} \max_{\bfb = (\bfb',b_{d}) \in B : \bfb' \preceq \bfc'} \abs{\Set{j\in\N \given (\bfc',j) \in \Dset, j \notin R_{d;\bfb}}}\\
&\geq
\sum_{\bfc' \in \N^{d-1}} \max_{\bfb = (\bfb',b_{d}) \in B : \bfb' \preceq \bfc'}
\max(\abs{\Set{j\in\N \given (\bfc',j) \in \Dset}} - b_{d}, 0)\\
&=
\sum_{\bfc' \in \N^{d-1}} \max_{c_{d} : (\bfc',c_{d}) \in \upset B}
\max(\abs{\Set{j\in\N \given (\bfc',j) \in \Dset}} - c_{d}, 0)\\
&=
\sum_{\bfc' \in \N^{d-1}} \abs{ \Set{ c_{d} \given (\bfc',c_{d}) \in \calD \cap \upset B}}.
\qedhere
\end{align*}
\end{proof}

\subsection{Inequalities for level sets}
For $\bfk\in\N^{d}$ and $l\in\Z$, we define sublevel sets by
\[
\sublevel_{l}^{\bfk} := \Set{ \bfa \in \N^{d} \given \bfa \preceq \bfk \text{ and } 1 \leq \abs{\bfa} \leq l},
\]
and level sets by
\[
\level_{l}^{\bfk} := \Set{ \bfa \in \N^{d} \given \bfa \preceq \bfk \text{ and } \abs{\bfa} = l}.
\]
For $\bfk=(k_{1},\dotsc,k_{d}) \in \N^{d}$, we write $\bfk' := (k_{1},\dotsc,k_{d-1}) \in \N^{d-1}$.

Denote the cardinality of a level set by $\Lambda^{\bfk}_{l} := \abs{\level_{l}^{\bfk}}$.
It can be computed by the following algorithm.
Initialize
\begin{equation}
\label{eq:level-recursive}
\Lambda^{()}_{l} =
\begin{cases}
1, & l=0,\\
0, & l\neq 0.
\end{cases}
\end{equation}
Then we apply the recursive definition
\begin{equation}
\label{eq:level-recursive2}
\Lambda^{(\bfk',k_{d})}_{l} = \sum_{j=0}^{k_{d}} \Lambda^{\bfk'}_{l-j}.
\end{equation}
The subscripts $l$ and $l-j$ are allowed to be negative here.\\

The following estimate generalizes \cite[(10.30)]{arxiv:1804.02488} and is crucial for setting up the induction in Theorem~\ref{thm:upset-levels}.
\begin{lemma}
\label{lem:level-set-concavity}
Let $d\in\N$ and $\bfk\in\N^{d}$.
Then, for every $a,b,a',b' \in \Z$ with $a+b = a'+b'$ and $b' \geq \max(a,b)$, we have
\[
\Lambda^{\bfk}_{a}\Lambda^{\bfk}_{b} \geq \Lambda^{\bfk}_{a'}\Lambda^{\bfk}_{b'}.
\]
\end{lemma}
\begin{proof}
We induct on $d$.
For $d=0$ the right-hand side of the conclusion is non-zero only if $a'=b'=0$.
But in this case $a+b=a'+b'=0$ and $\max(a,b)\leq b' = 0$, so that $a=b=0$, and we obtain equality.

Suppose that the conclusion is known with $d$ replaced by $d-1$ and $\bfk$ replaced by $\bfk'$.
Assume without loss of generality $a \leq b$, so that $a' \leq a \leq b \leq b'$.
The case $b'=b$ is trivial, and it remains to consider the case $b'=b+1$, since for general $b'$ we can iterate the inequality as follows:
\[
\Lambda^{\bfk}_{a}\Lambda^{\bfk}_{b}
\geq
\Lambda^{\bfk}_{a-1}\Lambda^{\bfk}_{b+1}
\geq
\Lambda^{\bfk}_{a-2}\Lambda^{\bfk}_{b+2}
\geq \dotsb \geq
\Lambda^{\bfk}_{a'}\Lambda^{\bfk}_{b'}.
\]
In the case $b'=b+1$ we have $a'=a-1$ and
\begin{align*}
\Lambda^{\bfk}_{a}\Lambda^{\bfk}_{b} - \Lambda^{\bfk}_{a'}\Lambda^{\bfk}_{b'}
&=
\sum_{i=0}^{k_{d}} \sum_{j=0}^{k_{d}} \Lambda^{\bfk'}_{a-i} \Lambda^{\bfk'}_{b-j}
-
\sum_{i=0}^{k_{d}} \sum_{j=0}^{k_{d}} \Lambda^{\bfk'}_{a-1-i} \Lambda^{\bfk'}_{b+1-j}\\
&=
\sum_{i=0}^{k_{d}} \Lambda^{\bfk'}_{a-i} \Lambda^{\bfk'}_{b-k_{d}}
+
\sum_{j=0}^{k_{d}-1} \Lambda^{\bfk'}_{a} \Lambda^{\bfk'}_{b-j}
-
\sum_{i=0}^{k_{d}-1} \Lambda^{\bfk'}_{a-1-i} \Lambda^{\bfk'}_{b+1}
-
\sum_{j=0}^{k_{d}} \Lambda^{\bfk'}_{a-1-k_{d}} \Lambda^{\bfk'}_{b+1-j}\\
&=
\sum_{i=0}^{k_{d}} (\Lambda^{\bfk'}_{a-i} \Lambda^{\bfk'}_{b-k_{d}} - \Lambda^{\bfk'}_{a-1-k_{d}} \Lambda^{\bfk'}_{b+1-i})
+
\sum_{j=0}^{k_{d}-1} (\Lambda^{\bfk'}_{a} \Lambda^{\bfk'}_{b-j} - \Lambda^{\bfk'}_{a-1-j} \Lambda^{\bfk'}_{b+1}).
\end{align*}
Each summand is non-negative by the induction hypothesis.
\end{proof}

The following result generalizes the proof of \cite[Lemma 10.7]{arxiv:1804.02488}.

\begin{theorem}
\label{thm:upset-levels}
For every $d\in\N$, $\bfk\in\N^{d}$, $m\in\N$, and every subset $T \subset \level^{\bfk}_{m}$, we have
\begin{equation}
\label{eq:upset-levels}
\abs{\level^{\bfk}_{m+1}} \abs{T} \leq \abs{\level^{\bfk}_{m}} \abs{T^{+}},
\end{equation}
where $T^{+} := (\upset T) \cap \level_{m+1}^{\bfk}$.
\end{theorem}

Theorem~\ref{thm:upset-levels}, applied with $T = \level^{\bfk}_{m} \cap U$, tells that the density of any up-set $U$ in the level sets $\level^{\bfk}_{m}$ increases with $m$:
\[
\frac{\abs{U \cap \level^{\bfk}_{m}}}{\abs{\level^{\bfk}_{m}}}
\leq
\frac{\abs{U \cap \level^{\bfk}_{m+1}}}{\abs{\level^{\bfk}_{m+1}}}.
\]
We prefer the formulation \eqref{eq:upset-levels} because it avoids division by zero for empty level sets.

\begin{proof}
We induct on $d$.
For $d=0$, the set $\N^{d}$ contains only the empty tuple, and one can see that the left-hand side of \eqref{eq:upset-levels} always vanishes.
Suppose that $d>0$ and the result is already known with $d$ replaced by $d-1$.

Recall the slice map \eqref{eq:slice-map}.
We have the inclusions
\[
S_{j}T \subseteq S_{j+1}(T^{+})
\quad
\text{ if } 0 \leq j < k_{d},
\]
\[
(S_{j}T)^{+} \subseteq S_{j}(T^{+})
\quad
\text{ if } 0 \leq j.
\]
(The indices here are different from \cite{arxiv:1804.02488}, where the convention $T_{j} := S_{m-j}T$ is used.)

These inclusions and the inductive hypothesis (for smaller $d$) give
\begin{equation}
\label{eq:6}
\abs{S_{j}T} \leq \abs{S_{j+1}(T^{+})},
\quad
\Lambda^{\bfk'}_{m-j+1} \abs{S_{j}T}
\leq \Lambda^{\bfk'}_{m-j} \abs{(S_{j}T)^{+}}
\leq \Lambda^{\bfk'}_{m-j} \abs{S_{j}(T^{+})}
\end{equation}
Let $j_{\min} := \max(0,m-k_{1}-\dotsb-k_{d-1})$, $j_{\max} := \min(m,k_{d})$.
Then
\[
\abs{T} = \sum_{j=j_{\min}}^{j_{\max}} \abs{S_{j}T}.
\]
The restrictions on $j$ reflect that some slices of $\level_{m}^{\bfk}$ are empty.
Suppose that we can find non-negative solutions $A_{j},B_{j} \geq 0$ with $j_{\min} \leq j \leq j_{\max}$ to the equations
\begin{equation}
\label{eq:level-split}
\Lambda_{m+1}^{\bfk} = A_{j} + \Lambda_{m-j+1}^{\bfk'} B_{j},
\quad
j_{\min} \leq j \leq j_{\max}
\end{equation}
\begin{equation}
\label{eq:level-recombine}
\Lambda_{m-j}^{\bfk'} B_{j} + A_{j-1} = \Lambda_{m}^{\bfk},
\quad
j_{\min}< j \leq j_{\max}.
\end{equation}
\begin{equation}
\label{eq:Bmin}
\Lambda^{\bfk}_{m} = B_{j_{\min}} \Lambda^{\bfk'}_{m-j_{\min}}
\end{equation}
\begin{equation}
\label{eq:Amax}
A_{j_{\max}} =
\begin{cases}
\Lambda^{\bfk}_{m} & \text{if } m < k_{d},\\
0 & \text{if } m \geq k_{d}.
\end{cases}
\end{equation}
Then we can finish the proof by estimating
\begin{align*}
\Lambda_{m+1}^{\bfk} \abs{T}
&= \sum_{j=j_{\min}}^{j_{\max}} (A_{j} + B_{j} \Lambda_{m-j+1}^{\bfk'})\abs{S_{j} T}\\
&\leq \sum_{j=j_{\min}}^{j_{\max}} A_{j} \abs{S_{j+1} (T^{+})} + \sum_{j=j_{\min}}^{j_{\max}} B_{j} \Lambda^{\bfk'}_{m-j} \abs{S_{j}(T^{+})}\\
&= \Lambda_{m}^{\bfk} \sum_{j=j_{\min}}^{\min(m+1,k_{d})} \abs{S_{j}(T^{+})}\\
&= \Lambda_{m}^{\bfk} \abs{T^{+}}.
\end{align*}
It remains to find positive solutions to the equations \eqref{eq:level-split}--\eqref{eq:Amax}.
There are more equations than unknowns, but this could have been expected, because we are comparing average densities of $T$ and $T^{+}$.

It is easy to verify that \eqref{eq:level-split} and \eqref{eq:level-recombine} hold for
\begin{align}
\label{eq:Aj}
A_{j}
&= 
\frac{1}{\Lambda^{\bfk'}_{m-j}} \bigl( \Lambda^{\bfk'}_{m-j} \Lambda^{\bfk}_{m+1} - \Lambda^{\bfk'}_{m+1}\Lambda^{\bfk}_{m} + (\Lambda^{\bfk}_{m+1} - \Lambda^{\bfk}_{m})(\Lambda^{\bfk'}_{m-j+1} + \dotsb + \Lambda^{\bfk'}_{m}) \bigr),
\\
\label{eq:Bj}
B_{j} &=
\frac{1}{\Lambda^{\bfk'}_{m-j} \Lambda^{\bfk'}_{m-j+1}} \bigl( \Lambda^{\bfk'}_{m+1}\Lambda^{\bfk}_{m} - (\Lambda^{\bfk}_{m+1} - \Lambda^{\bfk}_{m})(\Lambda^{\bfk'}_{m-j+1} + \dotsb + \Lambda^{\bfk'}_{m}) \bigr)
\end{align}
for $j_{\min} \leq j \leq j_{\max}$ (notice that the denominators in the above formulas do not vanish in this range of $j$'s).

\begin{proof}[Proof of~\eqref{eq:Bmin}]
Notice that if $j_{\min}>0$, then $\Lambda^{\bfk'}_{m-j_{\min}+1} = \dotsb = \Lambda^{\bfk'}_{m+1}=0$.
Therefore, for any value of $j_{\min}\geq 0$, we obtain $\Lambda^{\bfk'}_{m-j_{\min}+1} = \Lambda^{\bfk'}_{m+1}$ and
\[
B_{j_{\min}}
=
\frac{1}{\Lambda^{\bfk'}_{m-j_{\min}} \Lambda^{\bfk'}_{m-j_{\min}+1}} \bigl( \Lambda^{\bfk'}_{m-j_{\min}+1}\Lambda^{\bfk}_{m} \bigr)
=
\frac{\Lambda^{\bfk}_{m}}{\Lambda^{\bfk'}_{m-j_{\min}}}.
\qedhere
\]
\end{proof}
\begin{proof}[Proof of~\eqref{eq:Amax}]
In the case $j_{\max} = m < k_{d}$ we have
\begin{equation}
\label{eq:Bjmax:1}
\begin{split}
B_{m}
&=
\frac{1}{\Lambda^{\bfk'}_{0} \Lambda^{\bfk'}_{1}} \bigl( \Lambda^{\bfk'}_{m+1}\Lambda^{\bfk}_{m} - (\Lambda^{\bfk}_{m+1} - \Lambda^{\bfk}_{m})(\Lambda^{\bfk'}_{1} + \dotsb + \Lambda^{\bfk'}_{m}) \bigr)
\\ &= \frac{1}{\Lambda^{\bfk'}_{0} \Lambda^{\bfk'}_{1}} \bigl( \Lambda^{\bfk'}_{m+1}\Lambda^{\bfk}_{m} - \Lambda^{\bfk'}_{m+1}(\Lambda^{\bfk}_{m} - \Lambda^{\bfk'}_{0}) \bigr)
= \frac{\Lambda^{\bfk'}_{m+1}}{\Lambda^{\bfk'}_{1}}.
\end{split}
\end{equation}
Hence, by \eqref{eq:level-split} with $j=m$, we obtain $A_{m} = \Lambda^{\bfk}_{m+1}-\Lambda^{\bfk'}_{1}B_{m} = \Lambda^{\bfk}_{m+1} - \Lambda^{\bfk'}_{m+1} = \Lambda^{\bfk}_{m}$, and this shows \eqref{eq:Amax}.

In the case $j_{\max} = k_{d}$ we compute
\[
\Lambda^{\bfk'}_{m-k_{d}+1} + \dotsb + \Lambda^{\bfk'}_{m}
=
\Lambda^{\bfk}_{m} - \Lambda^{\bfk'}_{m-k_{d}}
=
\Lambda^{\bfk}_{m+1} - \Lambda^{\bfk'}_{m+1}.
\]
Hence,
\begin{equation}
\label{eq:Bjmax:2}
\begin{split}
B_{k_{d}}
&= \frac{1}{\Lambda^{\bfk'}_{m-k_{d}} \Lambda^{\bfk'}_{m-k_{d}+1}} \bigl( \Lambda^{\bfk'}_{m+1}\Lambda^{\bfk}_{m} - \Lambda^{\bfk}_{m+1}(\Lambda^{\bfk}_{m} - \Lambda^{\bfk'}_{m-k_{d}}) + \Lambda^{\bfk}_{m} (\Lambda^{\bfk}_{m+1} - \Lambda^{\bfk'}_{m+1}) \bigr)
\\ &=
\frac{1}{\Lambda^{\bfk'}_{m-k_{d}} \Lambda^{\bfk'}_{m-k_{d}+1}} \bigl( \Lambda^{\bfk}_{m+1} \Lambda^{\bfk'}_{m-k_{d}} \bigr)
= \frac{\Lambda^{\bfk}_{m+1}}{\Lambda^{\bfk'}_{m-k_{d}+1}}.
\end{split}
\end{equation}
By \eqref{eq:level-split} with $j=k_{d}$, it follows that $A_{k_{d}} = 0$, and this shows \eqref{eq:Amax} also in this case.
\end{proof}

\begin{proof}[Proof of $B_{j} \geq 0$]
The sequence $B_{j}$ is the quotient of a monotonic sequence and a positive sequence.
By \eqref{eq:Bmin}, we know $B_{j_{\min}} \geq 0$.
From \eqref{eq:Bjmax:1} and \eqref{eq:Bjmax:2}, we also see $B_{j_{\max}} \geq 0$.
Hence $B_{j} \geq 0$ for all $j_{\min} \leq j \leq j_{\max}$.
\end{proof}

\begin{proof}[Proof of $A_{j} \geq 0$]
We pass to $A_{j}$ with $j_{\min} \leq j \leq j_{\max}$ and compute
\begin{align*}
\Lambda^{\bfk'}_{m-j} A_{j}
&=
\Lambda^{\bfk}_{m+1}(\Lambda^{\bfk'}_{m-j} + \dotsb + \Lambda^{\bfk'}_{m}) - \Lambda^{\bfk}_{m}(\Lambda^{\bfk'}_{m-j+1} + \dotsb + \Lambda^{\bfk'}_{m+1})
\\ &=
\Bigl( \sum_{a=0}^{k_{d}} \Lambda^{\bfk'}_{m+1-a} \Bigr)
\Bigl( \sum_{b=0}^{j} \Lambda^{\bfk'}_{m-b} \Bigr)
-
\Bigl( \sum_{b=0}^{k_{d}} \Lambda^{\bfk'}_{m-b} \Bigr)
\Bigl( \sum_{a=0}^{j} \Lambda^{\bfk'}_{m+1-a} \Bigr).
\end{align*}
Canceling summands that appear both with plus and with minus, we obtain
\begin{align*}
\Lambda^{\bfk'}_{m-j} A_{j}
&=
\Bigl( \sum_{a=j+1}^{k_{d}} \Lambda^{\bfk'}_{m+1-a} \Bigr)
\Bigl( \sum_{b=0}^{j} \Lambda^{\bfk'}_{m-b} \Bigr)
-
\Bigl( \sum_{b=j+1}^{k_{d}} \Lambda^{\bfk'}_{m-b} \Bigr)
\Bigl( \sum_{a=0}^{j} \Lambda^{\bfk'}_{m+1-a} \Bigr)
\\ &=
\sum_{a=j+1}^{k_{d}} \sum_{b=0}^{j}
\bigl( \Lambda^{\bfk'}_{m+1-a}\Lambda^{\bfk'}_{m-b}
-
\Lambda^{\bfk'}_{m-a}\Lambda^{\bfk'}_{m+1-b} \bigr).
\end{align*}
By Lemma~\ref{lem:level-set-concavity}, each summand is non-negative, so $A_{j} \geq 0$.
\end{proof}
We have verified that \eqref{eq:Aj} and \eqref{eq:Bj} are indeed positive solutions of the equations \eqref{eq:level-split}--\eqref{eq:Amax}.
This finishes the proof of Theorem~\ref{thm:upset-levels}.
\end{proof}

\begin{corollary}
\label{cor:sublevel-up-set}
Let $d\geq 1$, $\bfk = (k_{1},\dotsc,k_{d}) \in\N_{>0}^{d}$, and $1 \leq l < l'$.
Then, for every up-set $B \subset \N^{d}$, we have
\begin{equation}
\label{eq:sublevel-up-set}
\abs{\sublevel^{\bfk}_{l'}} \abs{B \cap \sublevel^{\bfk}_{l}}
\leq
\abs{\sublevel^{\bfk}_{l}} \abs{B \cap \sublevel^{\bfk}_{l'}}.
\end{equation}
If $\level^{\bfk}_{l'}\neq\emptyset$, then equality can only hold in \eqref{eq:sublevel-up-set} in the following cases:
\begin{enumerate}
\item\label{it:sublevel-up-set:full/empty} $B\cap\sublevel^{\bfk}_{l'} \in \Set{\emptyset, \sublevel^{\bfk}_{l'}}$, or
\item\label{it:sublevel-up-set:d=2} $d=2$, $1=k_{1}<k_{2}$, $B = \upset\Set{(1,0)}$, and $l'\le k_2$, or
\item\label{it:sublevel-up-set:d=2'} $d=2$, $1=k_{2}<k_{1}$, $B = \upset\Set{(0,1)}$, and $l'\le k_1$.
\end{enumerate}
\end{corollary}
Corollary~\ref{cor:sublevel-up-set} recovers \cite[Lemma 10.7]{arxiv:1804.02488}, including the equality condition, upon setting $k_{1}=\dotsb=k_{d}=l'$.

Before proving Corollary~\ref{cor:sublevel-up-set}, let us give an informal outline.
Theorem~\ref{thm:upset-levels} tells that the density of $B$ in $\level^{\bfk}_{m}$ increases with $m$, and \eqref{eq:sublevel-up-set} follows by averaging this statement.
By the averaging argument, if $B$ has equal densities in $\sublevel^{\bfk}_{l'}$ and $\sublevel^{\bfk}_{l}$ and the level set $\level^{\bfk}_{l'}$ is non-empty, then $B$ must have the same density in each level set $\level^{\bfk}_{m}$.
The equality condition follows with this observation applied to $m=1,2$.

\begin{proof}
We begin with the inequality \eqref{eq:sublevel-up-set}.
We may assume $B \cap \sublevel^{\bold{k}}_{l} \neq \emptyset$.
Then also $B\cap \sublevel^{\bold{k}}_{l'} \neq \emptyset$ for every $l'>l$.
Hence it suffices to consider the case $l' = l+1$, that is,
\begin{equation}
\label{eq:sublevel-up-set-consecutive}
\abs{B \cap \sublevel^{\bfk}_{l}}
\leq
\frac{\abs{\sublevel^{\bfk}_{l}}}{\abs{\sublevel^{\bfk}_{l+1}}} \abs{B \cap \sublevel^{\bfk}_{l+1}}.
\end{equation}
Indeed, \eqref{eq:sublevel-up-set} follows from \eqref{eq:sublevel-up-set-consecutive} applied $l'-l$ times.

In proving \eqref{eq:sublevel-up-set-consecutive}, we may assume $\level^{\bfk}_{l+1} \neq \emptyset$, since otherwise the left-hand side and the right-hand side coincide.
Then also $\level^{\bfk}_{l''} \neq \emptyset$ for all $0 \leq l'' \leq l+1$.
Let
\[
B_m := B \cap \level^{\bfk}_m,
\qquad 0\leq m \leq l+1.
\]
Since $B$ is an up-set, we have $B_{r} \supseteq \level^{\bfk}_{r} \cap \upset B_{m}$ for all $0 \leq m \leq r \leq l+1$.
By Theorem~\ref{thm:upset-levels}, we obtain
\[
\abs{B_{m}} \leq \frac{\abs{\level^{\bfk}_{m}}}{\abs{\level^{\bfk}_{r}}} \abs{B_{r}}.
\]
Substituting $r=l+1$ and summing these inequalities, we deduce
\[
\abs{B \cap \sublevel^{\bfk}_{l}}
=
\sum_{m=1}^l \abs{B_m}
\leq
\frac{\sum_{m=1}^l \abs{\level^{\bfk}_m}}{\abs{\level^{\bfk}_{l+1}}} \abs{B_{l+1}}
=
\frac{\abs{\sublevel^{\bfk}_l}}{\abs{\sublevel^{\bfk}_{l+1}} - \abs{\sublevel^{\bfk}_l}} (\abs{B \cap \sublevel^{\bfk}_{l+1}} - \abs{B \cap \sublevel^{\bfk}_{l}}).
\]
Rearranging, we obtain \eqref{eq:sublevel-up-set-consecutive}.

Next we verify the equality condition.
We may assume $l'\geq 2$.
If $\level^{\bfk}_{l'}\neq\emptyset$, then equality in~\ref{cor:sublevel-up-set} implies equality in each application of Theorem~\ref{thm:upset-levels}.
In particular, $\abs{\level^{\bfk}_{2}} \abs{B_{1}} = \abs{\level^{\bfk}_{1}} \abs{B_{2}}$.

Decompose $d=i+i'+j+j'$, where $i+i' = \abs{\Set{m \given k_{m}=1}}$, $j+j' = \abs{\Set{m \given k_{m}>1}}$, $i = \abs{\Set{m \given k_{m}=1, e_{m}\in B}}$, $j = \abs{\Set{m \given k_{m}>1, e_{m}\in B}}$.
Then
\[
\abs{\level^{\bfk}_{1}} = d = i + i' + j + j', \quad
\abs{\level^{\bfk}_{2}} = \binom{d}{2} + j + j', \quad
\abs{B_{1}} = i+j, \quad
\abs{B_{2}} \geq \binom{d}{2} - \binom{i'+j'}{2} + j.
\]
Thus we obtain
\begin{equation}
\label{eq:2}
(i+j) (\binom{d}{2} + j + j') \geq d (\binom{d}{2} - \binom{i'+j'}{2} + j).
\end{equation}
We simplify this inequality as follows:
\begin{align*}
\eqref{eq:2}
&\iff
2 (i+j) (j+j') \geq (i'+j') d (d-1) - d (i'+j')(i'+j'-1) + 2jd
\\ &\iff
2 i (j+j') - 2j(i+i') \geq (i'+j') d (i+j)
\\ &\iff
2ij' - 2ji' \geq (i'+j') d (i+j).
\end{align*}
In the case $ij'\neq 0$ this implies $d\leq 2$, so in fact $i=j'=1$ and $i'=j=0$.
Thus, up to interchanging coordinates, we are in the situation $d=2$, $k_{1}=1$, $k_{2}>1$ of case~\ref{it:sublevel-up-set:d=2}.
There are four possibilities for the set $B_{1}$ in this case, two of which are covered by case~\ref{it:sublevel-up-set:full/empty}.
It is easy to check the remaining two.

In the case $ij'=0$ we obtain
\[
(i'+j')(i+j) = 0.
\]
In the case $i'+j'=0$ we have $B_{1} = \sublevel^{\bfk}_{1}$.
In the case $i+j=0$ we have $B_{1} = \emptyset$.
Since the densities of $B$ in all level sets $\level^{\bfk}_{l}$ coincide, the conclusion follows.
\end{proof}

\subsection{Vandermode type matrix rank estimate}
We are now in position to estimate the rank of the matrices that appear on the right-hand side of~\eqref{eq:1}.
\begin{lemma}
\label{lem:powers-rank-est}
Let $d\geq 1$, $\bfk\in\N^{d}$, $1\leq l < l'$.
Let $\calA \subseteq \sublevel^{\bfk}_{l'}$.
Then
\begin{equation}
\label{eq:powers-rank-est}
\rank_{\R} (\bfa^{\bfi})_{\bfa\in\calA, \bfi\in\sublevel^{\bfk}_{l}}
\geq \frac{\abs{\sublevel^{\bfk}_{l}}}{\abs{\sublevel^{\bfk}_{l'}}} \abs{\calA}.
\end{equation}
If $\level^{\bfk}_{l'} \neq \emptyset$, then equality in \eqref{eq:powers-rank-est} can only hold in the following cases:
\begin{enumerate}
\item\label{it:powers-rank-est:full/empty} $\calA \in \Set{\emptyset, \sublevel^{\bfk}_{l'}}$, or
\item\label{it:powers-rank-est:d=2} $d=2$, $1=k_{1}<k_{2}$, and $\calA = \Set{(0,1),\dotsc,(0,l')}$, or
\item\label{it:powers-rank-est:d=2'} $d=2$, $1=k_{2}<k_{1}$, and $\calA = \Set{(1,0),\dotsc,(l',0)}$.
\end{enumerate}
\end{lemma}

\begin{proof}
Let $Q := \abs{\sublevel^{\bfk}_{l}} - \rank_{\R} (\bfa^{\bfi})_{\bfa\in\calA, \bfi\in\sublevel^{\bfk}_{l}}$.
Then there is a subspace $W \subseteq \R^{\sublevel^{\bfk}_{l}}$ such that $\dim W = Q$ and for every vector $w \in W$ we have $\sum_{\bfi\in\sublevel^{\bfk}_{l}} w_{\bfi} \bfa^{\bfi} = 0$ for every $\bfa \in \calA$.
Choose a basis $w_{1},\dotsc,w_{d}$ of $W$ such that the indices
\[
\bfb_{q} = \operatorname{lex\,max} \Set{ \bfb \given w_{q,\bfb} \neq 0 }
\]
are pairwise distinct, where $\operatorname{lex\,max}$ denotes the maximum with respect to the lexicographic order on $\N^{d}$.
For each $1 \leq q \leq Q$ let $f_{q}(\bfx) := \sum_{\bfi\in\sublevel^{\bfk}_{l}} w_{q,\bfi} \bfx^{\bfi}$.
Then each polynomial $f_{q}$ vanishes on $\calA$ and also at $\vec{0}$, since it lacks a constant term.

We construct sets $R_{*;*;\bfb_{q}}$ with which we will be able to apply Lemma~\ref{lem:SZ}.
Fix $\bfb=\bfb_{q}=(b_{1},\dotsc,b_{d})$.
For $l=d,\dotsc,1$ (in descending order) and
\begin{equation}
\label{eq:3}
n_d \in \N \setminus R_{d; \bfb}, n_{d-1} \in \N \setminus R_{d-1; n_{d}; \bfb}, \dotsc, n_{l+1} \in \N \setminus R_{l+1; n_{l+2}, \ldots, n_d; \bfb}
\end{equation}
define inductively $R_{l; n_{l+1}, \ldots, n_{d}; \bfb}$ to be the set of those $n_{l}\in\N$ such that the coefficient of $x_{1}^{b_{1}} \dotsm x_{l-1}^{b_{l-1}}$ in $f_{q}(x_{1},\dotsc,x_{l-1},n_{l},\dotsc,n_{d})$ vanishes.

Then, by downward induction on $l$, one sees that, whenever \eqref{eq:3} holds, the lexicographically leading power of the polynomial $f_{q}(x_{1},\dotsc,x_{l},n_{l+1},\dotsc,n_{d})$ is $(b_{1},\dotsc,b_{l})$, and therefore $\abs{R_{l;n_{l},\dotsc,n_{d};\bfb}} \leq b_{l}$.
This verifies hypothesis~\ref{lem:SZ:1} of Lemma~\ref{lem:SZ}.

Let $\Dset := \sublevel^{\bfk}_{l'} \cup \Set{\vec{0}}$, $A := \calA \cup \Set{\vec{0}} \subset \N^{d}$, and $B := \Set{\bfb_{1}, \dotsc, \bfb_{Q}} \subseteq \sublevel^{\bfk}_{l} \subset \N^{d}\setminus\Set{\vec{0}}$.
Then hypothesis~\ref{lem:SZ:2} of Lemma~\ref{lem:SZ} holds with the sets $R_{*;*;\bfb}$ constructed above, since each $f_{q}$ vanishes on $A$.
By Lemma~\ref{lem:SZ}, we obtain
\begin{equation}
\label{eq:A<=D-B}
\abs{\calA} + 1
= \abs{A}
\leq \abs{\Dset \setminus \upset B}
= \abs{\sublevel^{\bfk}_{l'}} + 1 - \abs{\sublevel^{\bfk}_{l'} \cap \upset B}.
\end{equation}

By Corollary~\ref{cor:sublevel-up-set} and \eqref{eq:A<=D-B}, we have
\begin{equation}
\label{eq:Q<=}
Q
=
\abs{B}
\leq
\abs{\sublevel^{\bfk}_{l} \cap \upset B}
\leq
\frac{\abs{\sublevel^{\bfk}_{l}}}{\abs{\sublevel^{\bfk}_{l'}}} \abs{\sublevel^{\bfk}_{l'} \cap \upset B}
\leq
\frac{\abs{\sublevel^{\bfk}_{l}}}{\abs{\sublevel^{\bfk}_{l'}}}
(\abs{\sublevel^{\bfk}_{l'}} - \abs{\calA})
=
\abs{\sublevel^{\bfk}_{l}} - \frac{\abs{\sublevel^{\bfk}_{l}}}{\abs{\sublevel^{\bfk}_{l'}}} \abs{\calA}.
\end{equation}
Rearranging the inequality \eqref{eq:Q<=}, we obtain \eqref{eq:powers-rank-est}.
It remains to discuss the possible equality cases in \eqref{eq:powers-rank-est}.

Suppose that $\level^{\bfk}_{l'} \neq \emptyset$ and equality holds in \eqref{eq:powers-rank-est}.
Then in the above proof equality holds in \eqref{eq:Q<=} and \eqref{eq:A<=D-B}, and in particular we have equality in the above application of Corollary~\ref{cor:sublevel-up-set}.
We consider the possible equality cases in Corollary~\ref{cor:sublevel-up-set} separately.

In the case~\ref{it:sublevel-up-set:full/empty} of the equality condition in Corollary~\ref{cor:sublevel-up-set}, we have $\upset B\cap\sublevel^{\bfk}_{l'} \in \Set{\emptyset, \sublevel^{\bfk}_{l'}}$.
In the subcase $\upset B\cap\sublevel^{\bfk}_{l'} = \emptyset$ we have $B = \emptyset$, so $\rank_{\R} (\bfa^{\bfi})_{\bfa\in\calA, \bfi\in\sublevel^{\bfk}_{l}} = \abs{\sublevel^{\bfk}_{l}}$, and, since we have equality in \eqref{eq:powers-rank-est}, this implies $\calA = \sublevel^{\bfk}_{l'}$.
In the subcase $\upset B\cap\sublevel^{\bfk}_{l'} = \sublevel^{\bfk}_{l'}$ we have $\level^{\bfk}_{1} \subseteq B$.
By induction on $j=1,\dotsc,d$ we see that for every $\bfa \in \calA$ we have $a_{j}=0$.
Hence $\calA \subseteq \Set{0}$, but on the other hand $0 \not\in \sublevel^{\bfk}_{l'}$, so in fact $\calA=\emptyset$.
Thus both subcases fall in the case~\ref{it:powers-rank-est:full/empty} of the equality condition in Lemma~\ref{lem:powers-rank-est}.

In the case~\ref{it:sublevel-up-set:d=2} of the equality condition in Corollary~\ref{cor:sublevel-up-set}, we have $d=2$, $1=k_{1}<k_{2}$, $\upset B = \upset \Set{(1,0)}$, and $l' \leq k_{2}$.
We have to show that $\calA$ has the form claimed in the case~\ref{it:powers-rank-est:d=2} of the equality condition of Lemma~\ref{lem:powers-rank-est}.
In this case we have
\[
\card{\calD \setminus \upset B}
=
\card{\Set{(0,0),\dotsc,(0,l')}}
=
l'+1,
\]
and, since equality holds in \eqref{eq:A<=D-B}, we obtain $\abs{\calA}=l'$.
Since equality holds in \eqref{eq:powers-rank-est}, this implies
\[
\card{B}
=
Q
=
2l - \rank_{\R} (\bfa^{\bfi})_{\bfa\in\calA, \bfi\in\sublevel^{\bfk}_{l}}
=
2l - \frac{\abs{\sublevel^{\bfk}_{l}}}{\abs{\sublevel^{\bfk}_{l'}}} \abs{\calA}
=
2l - \frac{2l}{2l'} l'
=
l.
\]
Since $B \subseteq \sublevel^{\bfk}_{l} \cap \upset\Set{(1,0)}$, this can only hold if $B = \sublevel^{\bfk}_{l} \cap \upset\Set{(1,0)} = \Set{(1,0),\dotsc,(1,l-1)}$, so we may assume $\bfb_{q}=(1,q-1)$ for $1 \leq q \leq Q$ and write $f_{q}(\bfa) = f_{q,0}(a_{2})+a_{1} f_{q,1}(a_{2})$, where $f_{q,0},f_{q,1}$ are polynomials in one variable with $\deg f_{q,0} \leq l$ and $\deg f_{q,1} = q-1$.

Since $f_{1,1}$ is a non-zero constant, for each $a_{2}$ we can have $f_{1}(a_{1},a_{2})=0$ for at most one value of $a_{1}$.
Since $f_{1}$ vanishes on $A$, this implies that for each $a_{2}$ there is at most one value of $a_{1}$ such that $(a_{1},a_{2}) \in A$.
Since $\abs{A} = l'+1$ and $A \subset \sublevel^{\bfk}_{l'}$, for each $a_{2} \in \Set{0,\dotsc,l'}$ there is in fact \emph{exactly} one $a_{1}\in\Set{0,1}$ such that $(a_{1},a_{2}) \in A$.
In the case $a_{2}=l'$ this implies $(0,l') \in A$ since $(1,l') \not\in \sublevel^{\bfk}_{l'}$.
Moreover, $(0,0) \in A$, since $f_{1}(1,0) = f_{1,0}(0) + f_{1,1}(0) = f_{1,1}(0) \neq 0$.

We claim that $f_{1,0}=0$.
This will imply $f_{1}(1,a_{2})=f_{1,1}(a_{2}) \neq 0$ for all $a_{2}$, so that $(1,a_{2}) \not\in A$, and therefore $(0,a_{2}) \in A$ for all $a_{2} \in \Set{0,\dotsc,l'}$.

Since the polynomials $f_{q,1}$ have distinct degrees, they form a basis of the space of polynomials of degree $\leq l-1$.
Hence there is a basis $\tilde{f}_{1},\dotsc,\tilde{f}_{Q}$ of the space spanned by $f_{1},\dotsc,f_{Q}$ such that, writing $\tilde{f}_{q}(\bfa) = \tilde{f}_{q,0}(a_{2})+a_{1}\tilde{f}_{q,1}(a_{2})$, the one-variable polynomial $\tilde{f}_{q,1}$ is the Lagrange interpolation polynomial of degree $l-1$ that vanishes on $\Set{1,\dotsc,l}\setminus\Set{q}$ and takes the value $1$ at $q$.
It suffices to show that each $\tilde{f}_{q,0}$, which is a polynomial of degree $\leq l$ in one variable, vanishes identically.
We already know $\tilde{f}_{q,0}(0) = \tilde{f}_{q,0}(l') = 0$, since $(0,0),(0,l') \in A$.
Moreover, for each $a_{2} \in \Set{1,\dots,l} \setminus \Set{q}$, we have either $(0,a_{2}) \in A$ or $(1,a_{2}) \in A$.
Since $\tilde{f}_{q,1}(a_{2})=0$, in both cases we have $\tilde{f}_{q,0}(a_{2})=\tilde{f}_{q}(0,a_{2})=\tilde{f}_{q}(1,a_{2})=0$.
Therefore, $\tilde{f}_{q,0}$ vanishes on $\Set{0,\dots,l,l'} \setminus \Set{q}$.
Since this set has cardinality $l+1$ and $\tilde{f}_{q,0}$ has degree $\leq l$, the polynomial $\tilde{f}_{q,0}$ vanishes identically.
This finishes the proof of the claim and shows that the equality condition~\ref{it:powers-rank-est:d=2} of Lemma~\ref{lem:powers-rank-est} holds.

The case when the equality condition~\ref{it:sublevel-up-set:d=2'} of Corollary~\ref{cor:sublevel-up-set} holds is similar to the previous case.
\end{proof}

\begin{proof}[Proof of Theorem~\ref{thm:BLcond}]
The condition \eqref{eq:BLcond} does not depend on the choice of a spanning set $\Set{v_1,\dotsc,v_{H}}$ of $V$.
Hence, fixing a monomial order on $\N^{d}$, we may assume that $H=\dim V$ and the maximal indices of non-vanishing entries
\[
\bfa_{h} := \max \Set{ \bfa \given v_{h,\bfa} \neq 0 }
\]
are pairwise distinct, where the maximum is taken with respect to the fixed monomial order.
By Lemma~\ref{lem:rank-est-by-leading-monomials}, we obtain that the left-hand side of \eqref{eq:BLcond} is bounded below by
\[
\rank_{\R}
( \bfa_{h}^{\bfj} )_{\bfj \in \calD_{l},1\leq h \leq H}
=
\rank_{\R}
( \bfa_{h}^{\bfj} )_{\bfj \in \sublevel^{\bfk}_{l},1\leq h \leq H}.
\]
Also, $\calA = \Set{\bfa_{1},\dotsc,\bfa_{H}} \subseteq \calD_{k} = \sublevel^{\bfk}_{l'}$ with $l':=k$.
By Lemma~\ref{lem:powers-rank-est}, the latter rank is bounded below by
\[
\frac{\abs{\sublevel^{\bfk}_{l}}}{\abs{\sublevel^{\bfk}_{l'}}} \abs{\calA}
=
\frac{\card{\calD_{l}}}{\card{\calD_{k}}} \dim V.
\]
This shows \eqref{eq:BLcond}.

Suppose that equality holds in \eqref{eq:BLcond}.
Since $k \leq k_{1} + \dotsb + k_{d}$, we have $\level^{\bfk}_{l'} \neq \emptyset$, so one of the equality conditions of Lemma~\ref{lem:powers-rank-est} must hold.

If Condition~\ref{it:powers-rank-est:full/empty} of Lemma~\ref{lem:powers-rank-est} holds, then Condition~\ref{thm:BLcond:full/empty} of Theorem~\ref{thm:BLcond} holds.

If Condition~\ref{it:powers-rank-est:d=2} of Lemma~\ref{lem:powers-rank-est} holds, then we can repeat the above proof of the inequality \eqref{eq:BLcond} with the lexicographic order as the monomial order.
If Condition~\ref{it:powers-rank-est:d=2} of Lemma~\ref{lem:powers-rank-est} still holds, then we conclude that Condition~\ref{thm:BLcond:d=2} of Theorem~\ref{thm:BLcond} holds.

The case when Condition~\ref{it:powers-rank-est:d=2'} of Lemma~\ref{lem:powers-rank-est} holds is similar to the case when Condition~\ref{it:powers-rank-est:d=2} of Lemma~\ref{lem:powers-rank-est} holds.
\end{proof}

\section{Lower bounds}
\label{sec:lower}

In this section we substantiate the claim from Section~\ref{sec:multilinear} that the exponents $\tGamma_{\calD}(p)$ defined in \eqref{eq:tGamma-recursion} coincide with the exponents $\tgamma$ defined in \eqref{eq:tgamma}.
This will follow from Corollary~\ref{cor:upper-bd=lower-bd} and \eqref{eq:gamma-j-large}.

Recall that the exponents $\tGamma_{\calD}(p)$ naturally appear in the proof of Theorem~\ref{thm:main}.
The exponents $\tgamma$, which appear in the statement of Theorem~\ref{thm:main}, are the lower bounds suggested by the solution counting argument in \cite{MR3132907}.

Let $\calK_{\bfk,k} := \calK(\calD(\bfk,\leq k))$.
For a natural number $d \geq 0$ and $\bfk = (k_{1},\dotsc,k_{d})\in\N_{>0}^{d}$, for $k\geq 0$ and $2\leq p<\infty$, let
\begin{equation}
\label{eq:gamma-recursion}
\gamma_{\bfk,k}(p) := \begin{cases}
0 & \text{if } d=0 \text{ or } k=0,\\
\max(\frac{d}{2}, d - \frac{\calK_{\bfk,k}}{p}, \max\limits_{1 \leq j \leq d} \gamma_{\bfk_{(j)},k}(p) + \frac{1}{p})
& \text{otherwise,}
\end{cases}
\end{equation}
where $\bfk_{(j)} = (k_{1},\dotsc,k_{j-1},k_{j},\dotsc,k_{d})$.

By \cite[Theorem 3.1]{MR3132907} we have the lower bound
\begin{equation}
\label{eq:J-lower-est}
J_{s}(X;\calD(\bfk,\leq k)) \gtrsim X^{p \gamma_{\bfk,k}(p)}
\end{equation}
on the number of solutions in \eqref{eq:J-est}, where $p=2s$.

Since $\calK_{\bfk,1}=d$, by induction on $d$ we obtain
\begin{equation}
\label{eq:gamma:k=1}
\gamma_{\bfk,1}(p)
=
\max(\frac{d}{2}, d - \frac{d}{p}, \max_{1 \leq j \leq d} \gamma_{\bfk_{(j)},1}(p) + \frac{1}{p})
=
d \Bigl(1 - \frac{1}{p}\Bigr).
\end{equation}
For $k\geq 1$ we have
\begin{equation}
\label{eq:K>=d}
\calK_{\bfk,k} \geq \calK_{\bfk,1} = d,
\end{equation}
so, by induction on $d$, we obtain
\begin{multline}
\label{eq:gamma:p=2}
\gamma_{\bfk,k}(2)
=
\max(\frac{d}{2}, d - \frac{\calK_{\bfk,k}}{2}, \max\limits_{1 \leq j \leq d} \gamma_{\bfk_{(j)},k}(2) + \frac{1}{2})
=
\max(\frac{d}{2}, \max\limits_{1 \leq j \leq d} \gamma_{\bfk_{(j)},k}(2) + \frac{1}{2})
=
\frac{d}{2}.
\end{multline}
Abbreviating $\tGamma_{\bfk,k} := \tGamma_{\calD(\bfk,\leq k)}$, we obtain
\begin{equation}
\label{eq:tGamma-recursion-cube}
\tGamma_{\bfk,k}(p) =
\begin{cases}
0 & \text{if } d=0 \text{ or } k=0,\\
d \bigl( 1 - \frac{1}{p} \bigr) & \text{if } k=1,\\
\max( \max\limits_{1\leq j\leq d}\tGamma_{\bfk_{(j)},k}(p) + \frac{1}{p}, \tGamma_{\bfk,k-1}(\max(2,p \frac{\calK_{d,k-1}}{\calK_{d,k}})) )
& \text{otherwise.}
\end{cases}
\end{equation}
By induction on $d$ and $k$, one sees that $\tGamma_{\bfk,k}(2) = \frac{d}{2}$.

The upper bound \eqref{eq:J-est} for the number of solutions of multidimensional Vinogradov systems can only be consistent with the lower bound \eqref{eq:J-lower-est} if $\tGamma \ge \gamma$ for $p \in \Set{2,4,6,\dotsc}$.
We will now show that this inequality in fact holds for all $p\geq 2$.
\begin{lemma}
\label{lem:upper-bd>lower-bd}
For all $d,k\geq 0$ and $2\leq p < \infty$, we have
\begin{equation}
\label{eq:upper-bd>lower-bd}
\tGamma_{\bfk,k}(p) \ge \gamma_{\bfk,k}(p).
\end{equation}
\end{lemma}
\begin{proof}
By definition \eqref{eq:gamma-recursion}, \eqref{eq:gamma:k=1}, and \eqref{eq:gamma:p=2} we have equality in \eqref{eq:upper-bd>lower-bd} if $d=0$, or $k\leq 1$, or $p=2$.
In the other cases we proceed by induction on $k$.
Let $d \geq 1$ and $k\geq 2$ and suppose that \eqref{eq:upper-bd>lower-bd} is already known for smaller values of $k$ and $d$.
In view of the recursive formulas \eqref{eq:gamma-recursion} and \eqref{eq:tGamma-recursion-cube} and the inductive hypothesis, it suffices to verify
\begin{equation}
\label{eq:upper-bd>lower-bd:recursion}
d - \frac{\calK_{\bfk,k}}{p}
\le
\tGamma_{\bfk,k-1}(\max(2,p \frac{\calK_{\bfk,k-1}}{\calK_{\bfk,k}})).
\end{equation}
If $p \frac{\calK_{\bfk,k-1}}{\calK_{\bfk,k}} \geq 2$, then by the inductive hypothesis we have
\[
\tGamma_{\bfk,k-1}(p \frac{\calK_{\bfk,k-1}}{\calK_{\bfk,k}})
\geq
\gamma_{\bfk,k-1}(p \frac{\calK_{\bfk,k-1}}{\calK_{\bfk,k}})
\geq
d - \frac{\calK_{\bfk,k-1}}{p \frac{\calK_{\bfk,k-1}}{\calK_{\bfk,k}}}
=
d - \frac{\calK_{\bfk,k}}{p},
\]
and this implies \eqref{eq:upper-bd>lower-bd:recursion}.
If $p \frac{\calK_{\bfk,k-1}}{\calK_{\bfk,k}} < 2$, then \eqref{eq:upper-bd>lower-bd:recursion} follows from
\[
d - \frac{\calK_{\bfk,k}}{p}
<
d - \frac{\calK_{\bfk,k-1}}{2}
\leq
d - \frac{d}{2}
=
\frac{d}{2}
=
\tGamma_{\bfk,k-1}(2).
\qedhere
\]
\end{proof}

In all our examples, we in fact have equality in \eqref{eq:upper-bd>lower-bd}.
This relies on the following extension of \cite[Lemma 9.4]{arxiv:1804.02488}.
\begin{lemma}
\label{lem:gamma}
For every $d\geq 1$, $\bfk\in\N_{>0}^{d}$, $k\geq 2$, and $2\leq p < \infty$, we have
\[
\gamma_{\bfk,k-1}(p)
\leq
\gamma_{\bfk,k}(\frac{p \calK_{\bfk,k}}{\calK_{\bfk,k-1}}).
\]
\end{lemma}
\begin{corollary}
\label{cor:upper-bd=lower-bd}
For every $d\geq 1$, $\bfk\in\N_{>0}^{d}$, $k\geq 0$, and $2\leq p < \infty$ we have
\begin{equation}
\label{eq:upper-bd=lower-bd}
\tGamma_{\bfk,k}(p) = \gamma_{\bfk,k}(p).
\end{equation}
\end{corollary}
\begin{proof}[Proof of Corollary~\ref{cor:upper-bd=lower-bd}]
We already know \eqref{eq:upper-bd=lower-bd} if $d=0$ or $k\leq 1$.
We proceed by induction on $d$ and $k$.
Let $d\geq 1$, $k\geq 2$, and suppose that the claim is known for smaller values of $d$ and $k$.
In view of the lower bound \eqref{eq:upper-bd>lower-bd}, it remains to show
\[
\tGamma_{\bfk,k}(p) \leq \gamma_{\bfk,k}(p).
\]
By the recursive formula \eqref{eq:tGamma-recursion-cube} and the inductive hypothesis, this is equivalent to
\[
\max( \max\limits_{1\leq j\leq d}\gamma_{\bfk_{(j)},k}(p) + \frac{1}{p}, \gamma_{\bfk,k-1}(\max(2,p \frac{\calK_{\bfk,k-1}}{\calK_{\bfk,k}})) )
\leq \gamma_{\bfk,k}(p).
\]
The first term on the left-hand side is $\leq \gamma_{\bfk,k}(p)$ by definition \eqref{eq:gamma-recursion}.
In the second term we distinguish two cases.
If $p \frac{\calK_{\bfk,k-1}}{\calK_{\bfk,k}} \leq 2$, then this term equals $\frac{d}{2}$, and the claim follows by definition \eqref{eq:gamma-recursion}.
Otherwise we can conclude by Lemma~\ref{lem:gamma}.
\end{proof}

\subsection{Reduction to monotonic $\bfk$}
Before showing Lemma~\ref{lem:gamma}, we will obtain a more explicit description of $\gamma_{\bfk,k}(p)$.
This quantity is invariant under permutations of entries of $\bfk$, and it will be convenient to bring $\bfk$ in a canonical order.
This will be facilitated by the following result.
\begin{lemma}
\label{lem:gamma-monotone}
Let $d \geq 0$ and $\bfk, \bfk'\in\N_{>0}^{d}$ with $\bfk \preceq \bfk'$.
Then, for every $k\geq 0$ and $2\leq p<\infty$, we have
\begin{equation}
\label{eq:gamma-monotone}
\gamma_{\bfk,k}(p) \geq \gamma_{\bfk',k}(p).
\end{equation}
\end{lemma}
\begin{proof}[Proof of Lemma~\ref{lem:gamma-monotone}]
We use induction on $d$.
For $d=0$ both sides in \eqref{eq:gamma-monotone} equal $0$.
Suppose now $d>0$ and \eqref{eq:gamma-monotone} is known with $d$ replaced by $d-1$.
Recalling the definition \eqref{eq:gamma-recursion}, the claim \eqref{eq:gamma-monotone} follows from $\calK_{\bfk,k} \leq \calK_{\bfk',k}$ and $\bfk_{(j)} \preceq \bfk'_{(j)}$ for every $1\leq j\leq d$.
\end{proof}
\begin{corollary}
For $d \geq 1$ and $\bfk = (k_{1},\dotsc,k_{d})\in\N_{>0}^{d}$ with $k_{1} \leq k_{2} \leq \dotsb \leq k_{d}$, we have
\begin{equation}
\label{eq:gamma-recursion-monotone}
\gamma_{\bfk,k}(p) = \max(\frac{d}{2}, d - \frac{\calK_{\bfk,k}}{p}, \gamma_{\bfk',k}(p) + \frac{1}{p}),
\end{equation}
where $\bfk' = (k_{1},\dotsc,k_{d-1})$.
\end{corollary}
So far we have finished the reduction to monotonic $\bfk$.

\subsection{The case of monotonic $\bfk$}
In the remaining part of this section we assume
\begin{equation}
\label{eq:kj-increasing}
k_{1} \leq k_{2} \leq \dotsc
\end{equation}
and abbreviate
\begin{align*}
\calK_{d,k} &:= \calK(\calD((k_{1},\dotsc,k_{d}), \leq k)),\\
\gamma_{d,k}(p) &:= \gamma_{(k_{1},\dotsc,k_{d}), k}(p).
\end{align*}

Unwinding the recursion \eqref{eq:gamma-recursion-monotone}, we obtain
\begin{equation}
\label{eq:gamma-j-1}
\gamma_{d,k}(p) = \max(\frac{d}{2}, \max_{1\leq j\leq d} (j + \frac{d-j}{p} - \frac{\calK_{j,k}}{p})).
\end{equation}
By \eqref{eq:K>=d}, for $k\geq 1$ and $1 \leq j \leq d/2$ we have
\[
j + \frac{d-j}{p} - \frac{\calK_{j,k}}{p}
\leq
j + \frac{d-j}{p} - \frac{j}{p}
=
(d-2j) (\frac{1}{p}-\frac{1}{2}) + \frac{d}{2}
\leq
\frac{d}{2},
\]
so that in fact
\begin{equation}
\label{eq:gamma-j-large}
\gamma_{d,k}(p)
=
\max(\frac{d}{2}, \max_{(d+1)/2\leq j\leq d} (j + \frac{d-j}{p} - \frac{\calK_{j,k}}{p})).
\end{equation}

\begin{proof}[Proof of Lemma~\ref{lem:gamma}]
For $k=2$, by \eqref{eq:gamma:k=1}, we have
\[
\gamma_{d,k-1}(p)
=
d - \frac{d}{p}
=
d - \frac{\calK_{d,2}}{\frac{p \calK_{d,2}}{\calK_{d,1}}}
\leq
\gamma_{d,2}(\frac{p \calK_{d,2}}{\calK_{d,1}}),
\]
where we have used \eqref{eq:gamma-j-large} in the last step.
In the remaining part of the proof we will assume $k\geq 3$.

By \eqref{eq:gamma-j-large}, it suffices to show that, for every integer $(d+1)/2 \leq j \leq d$, we have
\[
j + \frac{d-j}{p} - \frac{\calK_{j,k-1}}{p}
\leq
j + \frac{d-j}{\frac{p \calK_{d,k}}{\calK_{d,k-1}}} - \frac{\calK_{j,k}}{\frac{p \calK_{d,k}}{\calK_{d,k-1}}}.
\]
This is equivalent to
\begin{equation}
\label{eq:K-est0}
(d-j) (\frac{\calK_{d,k}}{\calK_{d,k-1}} - 1 )
\leq
\calK_{j,k-1} (\frac{\calK_{d,k}}{\calK_{d,k-1}} - \frac{\calK_{j,k}}{\calK_{j,k-1}}).
\end{equation}
This trivially holds if $j=d$, so we may assume $j<d$.
In this case necessarily $d\geq 3$ and $j\geq 2$.

Let $\Lambda_{j,k} := \Lambda^{(k_{1},\dotsc,k_{j})}_{k}$ denote the cardinality of the $k$-th level set as in \eqref{eq:level-recursive}.
Using that $\calK_{d,k} = \calK_{d,k-1} + k \Lambda_{d,k}$, we can reformulate \eqref{eq:K-est0} as
\begin{equation}
\label{eq:K-est1}
(d-j) (\frac{\Lambda_{d,k}}{\calK_{d,k-1}} )
\leq
\calK_{j,k-1} (\frac{\Lambda_{d,k}}{\calK_{d,k-1}} - \frac{\Lambda_{j,k}}{\calK_{j,k-1}}).
\end{equation}
This is in turn equivalent to
\begin{equation}
\label{eq:K-est2}
\frac{\Lambda_{j,k}}{\calK_{j,k-1} - (d-j)}
\leq
\frac{\Lambda_{d,k}}{\calK_{d,k-1}}.
\end{equation}
By downward induction on $j$, the inequality \eqref{eq:K-est2} will follow from
\begin{equation}
\label{eq:K-est3}
\frac{\Lambda_{j,k}}{\calK_{j,k-1} - (d-j)}
\leq
\frac{\Lambda_{j+1,k}}{\calK_{j+1,k-1} - (d-(j+1))}
\end{equation}
for $(d+1)/2 \leq j < d$.
The inequality \eqref{eq:K-est3} can be equivalently written as
\begin{equation}
\label{eq:K-est5}
\calK_{j,k-1} \Lambda_{j+1,k} - \calK_{j+1,k-1} \Lambda_{j,k}
\geq
(d-j) \Lambda_{j+1,k} - (d-j-1) \Lambda_{j,k}.
\end{equation}
Expanding $\calK$ on the left-hand side and using the recursive formula \eqref{eq:level-recursive2} on the right-hand side, we write \eqref{eq:K-est5} as
\begin{equation}
\label{eq:K-est6}
\sum_{l=1}^{k-1} l \cdot (\Lambda_{j+1,k}\Lambda_{j,l}-\Lambda_{j,k}\Lambda_{j+1,l})
\geq
\Lambda_{j,k} + (d-j)(\Lambda_{j,k-1}+\dotsb+\Lambda_{j,k-k_{j+1}}).
\end{equation}
By \eqref{eq:level-recursive2} and Lemma~\ref{lem:level-set-concavity}, each summand on the left-hand side of \eqref{eq:K-est6} is non-negative:
\begin{equation}
\label{eq:4}
\Lambda_{j+1,k}\Lambda_{j,l}-\Lambda_{j,k}\Lambda_{j+1,l}
=
\sum_{i=0}^{k_{j+1}} \Lambda_{j,k-i}\Lambda_{j,l}-\Lambda_{j,k}\Lambda_{j,l-i}
\geq 0.
\end{equation}
We were able to apply Lemma~\ref{lem:level-set-concavity} since $k \geq \max(l,k-i)$.
Using \eqref{eq:4} and $k\geq 3$, the estimate \eqref{eq:K-est6} will follow from
\begin{equation}
\label{eq:K-est7}
\sum_{l=1}^{2} l \cdot (\Lambda_{j+1,k}\Lambda_{j,l}-\Lambda_{j,k}\Lambda_{j+1,l})
\geq
\Lambda_{j,k} + (d-j)(\Lambda_{j,k-1}+\dotsb+\Lambda_{j,k-k_{j+1}}).
\end{equation}
The $l=1$ term on the left-hand side of \eqref{eq:K-est7} equals
\begin{equation}
\begin{split}
\Lambda_{j+1,k}\Lambda_{j,1} - \Lambda_{j,k}\Lambda_{j+1,1}
& =
(\Lambda_{j,k}+\dotsb+\Lambda_{j,k-k_{j+1}}) j - \Lambda_{j,k}(j+1)\\
& =
(\Lambda_{j,k-1}+\dotsb+\Lambda_{j,k-k_{j+1}}) j - \Lambda_{j,k}.
\end{split}
\end{equation}
Thus, since $2j-d\geq 1$, \eqref{eq:K-est7} will follow from
\begin{equation}
\label{eq:K-est8}
2 (\Lambda_{j+1,k}\Lambda_{j,2}-\Lambda_{j,k}\Lambda_{j+1,2})
\geq
2 \Lambda_{j,k} - (\Lambda_{j,k-1}+\dotsb+\Lambda_{j,k-k_{j+1}}).
\end{equation}
We distinguish two cases.

\paragraph{Case I: $k_{j+1} \geq 2$}
Expanding the left-hand side of \eqref{eq:K-est8} using \eqref{eq:level-recursive2}, we see that \eqref{eq:K-est8} will follow from
\begin{equation}
\label{eq:K-est9}
2 \sum_{m=0}^{2} (\Lambda_{j,k-m}\Lambda_{j,2}-\Lambda_{j,k}\Lambda_{j,2-m})
\geq
2 \Lambda_{j,k} - 2\min(\Lambda_{j,k-1},\Lambda_{j,k-2}).
\end{equation}
The terms $m=0,1$ on the left-hand side of \eqref{eq:K-est9} are non-negative by Lemma~\ref{lem:level-set-concavity}.
Hence it suffices to show
\begin{equation}
\label{eq:K-est10}
\Lambda_{j,k-2}\Lambda_{j,2}-\Lambda_{j,k}\Lambda_{j,0}
\geq
\Lambda_{j,k} - \min(\Lambda_{j,k-1},\Lambda_{j,k-2}),
\end{equation}
which can be written as
\begin{equation}
\label{eq:K-est11}
\Lambda_{j,k-2}\Lambda_{j,2}
\geq
2 \Lambda_{j,k} - \min(\Lambda_{j,k-1},\Lambda_{j,k-2}).
\end{equation}
The inequality \eqref{eq:K-est11} can be verified by a double counting argument.
Indeed, $\Lambda_{j,k}$ counts the number of ways to write $k=\abs{\bfa}$ with $\bfa \preceq \bfk$.
Each such $\bfa$ can be written as $\bfa = \bfa' + \bfa''$ with $\abs{\bfa'}=k-2$ and $\abs{\bfa''}=2$.
Those $\bfa$ with at least two non-zero entries have at least two such decompositions (since $k\geq 3$).
Those $\bfa$ with only one non-zero entry have exactly one such decomposition, but the number of such $\bfa$ is bounded by $\min(\Lambda_{j,k-1},\Lambda_{j,k-2})$ (again since $k\geq 3$).
On the other hand, the total number of decompositions is counted by the left-hand side of \eqref{eq:K-est11}.
This finishes the proof of \eqref{eq:K-est8} in Case I.

\paragraph{Case II: $k_{j+1}=1$}
In this case, by \eqref{eq:level-recursive2}, the inequality \eqref{eq:K-est8} is equivalent to
\begin{equation}
\label{eq:K-est12}
l \cdot (\Lambda_{j,k-1}\Lambda_{j,l}-\Lambda_{j,k}\Lambda_{j,l-1})
\geq
2 \Lambda_{j,k} - \Lambda_{j,k-1},
\quad l=2.
\end{equation}
Since $k_{j+1}=1$, by \eqref{eq:kj-increasing} also $k_{1}=\dotsb=k_{j}=1$.
We may assume $k\leq j$, since otherwise $\Lambda_{j,k} = 0$, so the right-hand side of \eqref{eq:K-est8} is negative and we can conclude by \eqref{eq:4}.
In this case we have $\Lambda_{j,k}=\binom{j}{k}$ for all pairs of arguments $(j,k)$ that we use.
Hence we can write the left-hand side of the required inequality \eqref{eq:K-est12} in the form
\begin{equation}
\begin{split}
\label{eq:7}
l \cdot (\binom{j}{k-1}\binom{j}{l}-\binom{j}{k}\binom{j}{l-1})
& =
\binom{j}{k}\binom{j}{l-1} (\frac{k(j-l+1)}{j-k+1}-l)\\
& =
\binom{j}{k} \binom{j}{l-1} \frac{(k-l)(j+1)}{j-k+1}
\end{split}
\end{equation}
and the right-hand side of \eqref{eq:K-est12} in the form
\[
2\binom{j}{k} - \binom{j}{k-1}
=
\binom{j}{k} (2 - \frac{k}{j-k+1})
=
\frac{1}{j-k+1}\binom{j}{k} (2(j-k+1) - k).
\]
Hence the claim \eqref{eq:K-est12} reduces to
\[
\binom{j}{l-1} (k-l)(j+1)
\geq
2(j+1) - 3k,
\]
which is a valid inequality for $l=2$ and $3 \leq k \leq j$.
This finishes the proof of \eqref{eq:K-est8} in Case II.
\end{proof}

\appendix
\section{\texorpdfstring{$\ell^{2}L^{p}$}{\9041\023\9000\262 Lp} decoupling}
\label{sec:l2}
Let us call an estimate for $\avDec(\calD, p, q, \delta)$ an \emph{$\ell^{q}L^{p}$ decoupling inequality}.
Theorem~\ref{thm:main} is an $\ell^{p}L^{p}$ decoupling inequality.
While sharp $\ell^{p}L^{p}$ decoupling inequalities are adequate for counting solutions of Diophantine equations, they can be sometimes strengthened to $\ell^{2}L^{p}$ decoupling inequalities, as in the case of the paraboloid \cite{MR3374964} and the moment curve \cite{MR3658168,MR3548534}.
In these cases the sharp $\ell^{p}L^{p}$ inequalities follow from the sharp $\ell^{2}L^{p}$ inequalities by an application of H\"older's inequality in the sum over $\Part{\delta}$.

In this section we indicate how our argument yields the following $\ell^{2}L^{p}$ decoupling inequality:
\begin{equation}
\label{eq:l2Lp-dec}
\avDec(\calD, p, 2, \delta)
\lesssim_{\epsilon}
\delta^{-\tGamma_{\calD}^{(2)}(p)-\epsilon},
\end{equation}
where $\tGamma_{\calD}^{(2)}$ is defined by the recursive relations
\begin{equation}
\label{eq:tGamma(2)-recursion}
\tGamma_{\calD}^{(2)}(p) :=
\begin{cases}
0 & \text{if } d=0 \text{ or } k=0,\\
d \bigl( 1 - \frac{1}{p} \bigr) & \text{if } k=1,\\
\max( \max\limits_{1 \leq j \leq d} \tGamma_{\bfP_{j} \calD}^{(2)}(p) + \frac{1}{2}, \tGamma_{\calD\cap\calS_{k-1}}^{(2)}(\max(2,p \frac{\calK (\calD \cap \calS_{k-1})}{\calK (\calD)})) )
& \text{otherwise.}
\end{cases}
\end{equation}
Note that $\tGamma_{\calD}^{(2)}(p) \geq \tGamma_{\calD}(p)$, so in general one expects better estimates in \eqref{eq:J-est} from using $\ell^{p}L^{p}$ decoupling rather than $\ell^{2}L^{p}$ decoupling.
However, in dimension $d=1$ it turns out that $\tGamma_{\calD}^{(2)}(p) = \tGamma_{\calD}(p)$, and in fact we recover the result in \cite{MR3548534}.

Most arguments in Sections \ref{sec:multilinear}, \ref{sec:induction-on-scales}, and \ref{sec:transverse} work equally well when we consider $\ell^{2}L^{p}$ decoupling, upon replacing all $\ell^{p}$ sums by $\ell^{2}$ sums.
There are only two substantial changes.
\begin{enumerate}
\item Proposition~\ref{prop:bourgain-guth-arg} is applied with $q=2$ rather than $q=p$.
This leads to $1/p$ being replaced by $1/2$ in \eqref{eq:tGamma'}, and this in turn effects the change of the exponent \eqref{eq:tGamma-recursion} to \eqref{eq:tGamma(2)-recursion}.
\item The ball inflation Lemma~\ref{lem:ball-inflation} has to be replaced by the more general Corollary~\ref{cor:ball-inflation:ell-r} below.
\end{enumerate}
\begin{corollary}[Ball inflation, $\ell^{q}$ version]
\label{cor:ball-inflation:ell-r}
In the setting of Lemma~\ref{lem:ball-inflation}, let $1 \leq q \leq t < \infty$.
Then
\begin{multline}
\label{eq:ball-inflation:ell-r}
\avL^{p}_{x \in B} \avprod \ell^{q}_{J \in \Part[R_i]{\rho}} \norm{f_{J}}_{\avL^{t}(w_{B(x,\rho^{-l})})}\\
\lesssim \nu^{-n_{l}/(t n_{k})} \abs{\log_{+} \rho}^{K^{d}}
\avprod \ell^{q}_{J \in \Part[R_i]{\rho}} \norm{f_{J}}_{\avL^{t}(w_B)}.
\end{multline}
\end{corollary}

\begin{proof}
We use a dyadic pigeonholing argument from \cite{MR3374964}.
Partition
\[
\Part[R_i]{\rho} = \calJ_{i,\infty} \cup \bigcup_{k=0}^{\floor{\log_{+}\rho}} \calJ_{i,k},
\]
where
\[
\calJ_{i,k} := \Set{ J\in\Part[R_{i}]{\rho} \given C^{-k-1} < \frac{\norm{f_{J}}_{\avL^{t}(w_B)}}{\max_{J' \in \Part[R_i]{\rho}}\norm{f_{J'}}_{\avL^{t}(w_B)}} \leq C^{-k}}
\]
with a large constant $C$.
It suffices to show
\begin{equation}
\label{eq:ball-inflation:ell-r'}
\avL^{p}_{x \in B} \avprod \ell^{q}_{J \in \calJ_{i,k_{i}}} \norm{f_{J}}_{\avL^{t}(w_{B(x,\rho^{-l})})}
\lesssim \nu^{-n_{l}/(t n_{k})}
\avprod \ell^{q}_{J \in \calJ_{i}} \norm{f_{J}}_{\avL^{t}(w_B)}
\end{equation}
for every choice of $k_{1},\dotsc,k_{M} \in \Set{0,\dotsc,\floor{\log_{+}\rho},\infty}$.

Since $q \leq t$, by H\"older's inequality the left hand side of \eqref{eq:ball-inflation:ell-r'} is at most
\begin{equation}
\bigl( \avprod \abs{\calJ_{i,k_{i}}}^{\frac{1}{q}-\frac{1}{t}} \bigr)
\avL^{p}_{x \in B} \avprod \ell^{t}_{J \in \calJ_{i,k_{i}}} \norm{f_{J}}_{\avL^{t}(w_{B(x,\rho^{-l})})}.
\end{equation}
By Lemma~\ref{lem:ball-inflation}, this is dominated by
\[
\bigl( \avprod \abs{\calJ_{i,k_{i}}}^{\frac{1}{q}-\frac{1}{t}} \bigr) \nu^{-n_{l}/(t n_{k})}
\avprod \ell^{t}_{J \in \calJ_{i,k_{i}}} \norm{f_{J}}_{\avL^{t}(w_{B})}.
\]
It remains to observe that, by definition of $\calJ_{i,k}$, we have
\[
\abs{\calJ_{i,k_{i}}}^{\frac{1}{q}-\frac{1}{t}}
\ell^{t}_{J \in \calJ_{i,k}} \norm{f_{J}}_{\avL^{t}(w_{B})}
\lesssim
\ell^{q}_{J \in \Part[R_{i}]{\rho}} \norm{f_{J}}_{\avL^{t}(w_{B})}
\]
for every $k$.
Indeed, for $k \neq \infty$ this holds because all summands have comparable size, while for $k=\infty$ this holds because each summand on the left-hand side is bounded by $C^{-\log_{+}\rho}$ times the largest summand on the right-hand side, and the number of summands on the left-hand side is bounded by $C^{\log_{+}\rho}$ if $C$ is large enough.
\end{proof}

\section{Decoupling for $k=1$: $L^2$ orthogonality}
\label{sec:k=1}
Let $2\le p\le \infty$ and $\delta\in (0, 1]$.
For every $\theta\in \Part{\delta}$, let $f_\theta: \R^d\to \C$ be a function with $\supp \widehat{f_\theta} \subseteq \theta$.
In this section we will prove
\begin{equation}
\norm[\Big]{\sum_{\theta\in\Part{\delta}}f_\theta }_{L^p(w_B)}
\lesssim_p
\delta^{-d(1-\frac{1}{p})} \Big(\avsum_{\theta\in\Part{\delta}}\norm{f_\theta }_{L^p(w_B)}^{2} \Big)^{1/2}
\end{equation}
for every ball $B$ of radius $\delta^{-1}$.
This implies the case $k=1$ of Theorem~\ref{thm:main}, because the normalized $\ell^{2}$ norm is bounded by the normalized $\ell^{p}$ norm.

Let $\psi, \theta : \R^d\to \C$ be Schwartz functions with $\one_{B(0,10)} \leq \hat{\psi} \leq \one_{B(0,20)}$, $\abs{\theta} \geq \one_{B(0,1)}$, and $\supp\hat{\theta} \subset B(0,C)$.
Define $\psi_{\theta}$ by $\widehat{\psi_\theta}(\xi)=\hat{\psi}(\delta^{-1}(\xi-c_\theta))$, with $c_\theta$ being the center of the cube $\theta$.
We will prove that
\begin{equation}
\norm[\Big]{\sum_{\theta\in\Part{\delta}}F_\theta*\psi_\theta }_{L^p(w_B)}
\lesssim_p
\delta^{-d(1-\frac{1}{p})} \Big(\avsum_{\theta\in\Part{\delta}} \norm{F_\theta }_{L^p(w_B)}^{2} \Big)^{1/2}
\end{equation}
for arbitrary functions $F_\theta$.
By complex interpolation, it suffices to consider only $p=2$ and $p=\infty$.
The case $p=\infty$ follows from the Cauchy--Schwarz inequality.
To prove the case $p=2$, by Lemma~\ref{lem:1-w}, it suffices to prove
\begin{equation}
\label{eq:k=1:p=2}
\norm[\Big]{\sum_{\theta\in\Part{\delta}} F_\theta*\psi_\theta }_{L^2(B)}
\lesssim
\Big(\sum_{\theta\in\Part{\delta}} \norm{F_\theta }^2_{L^2(w_B)} \Big)^{1/2}.
\end{equation}
By definition of $\theta$, we have
\begin{equation}
\norm[\Big]{\sum_{\theta\in\Part{\delta}}F_\theta*\psi_\theta }_{L^2(B)}
\le
\norm[\Big]{\sum_{\theta\in\Part{\delta}} \big(F_\theta*\psi_\theta\big) \theta_{B} }_{L^2_{x}(\R^{d})},
\end{equation}
where $\theta_{B}(x) := \theta(\frac{x-c_B}{r_B})$, with $c_B$ and $r_B$ being the center and the radius of $B$, respectively.
The summands on the right-hand side have boundedly overlapping Fourier supports, with bound independent of $\delta$.
Hence, by $L^2$ orthogonality, the right hand side can be bounded by
\begin{equation}
\Big( \sum_{\theta\in\Part{\delta}} \norm[\Big]{ \big(F_\theta*\psi_\theta\big) \theta_{B} }^2_{L^2_{x}(\R^{d})} \Big)^{1/2}.
\end{equation}
This can be in turn estimated by the right-hand side of \eqref{eq:k=1:p=2} for every $E<\infty$.

\printbibliography
\end{document}